\documentclass[preprint,1p]{elsarticle}
\usepackage{amssymb,amsthm,verbatim,graphicx,eepic,xspace,psfrag,ifthen,comment}
\usepackage{amsmath}
\usepackage{color}

\newtheorem{thm}{Theorem}[section]
\newtheorem{defn}[thm]{Definition}

\newtheorem{prop}[thm]{Proposition}
\newtheorem{rem}[thm]{Remark}
\newtheorem{lem}[thm]{Lemma}
\newtheorem{cor}[thm]{Corollary}

\theoremstyle{remark}

\numberwithin{equation}{section}

\newcommand{\N}{\ensuremath{\mathbb{N}}}
\newcommand{\R}[1]{\ensuremath{\mathbb{R}^{#1}}}
\def\MM{\mathcal{M}}
\newcommand{\cone}[1]{\ensuremath{\mathfrak{C}_{#1}}}
\newcommand{\bfe}{\ensuremath{\mathbf{e}}}
\newcommand{\bfp}{\ensuremath{\mathbf{p}}}
\newcommand{\bfw}{\ensuremath{\mathbf{w}}}
\newcommand{\bfx}{\ensuremath{\mathbf{x}}}
\newcommand{\bfy}{\ensuremath{\mathbf{y}}}
\newcommand{\bfz}{\ensuremath{\mathbf{z}}}
\newcommand{\bfr}{\ensuremath{\mathbf{r}}}
\newcommand{\bfk}{\ensuremath{\mathbf{k}}}
\newcommand{\bfs}{\ensuremath{\mathbf{s}}}
\newcommand{\bft}{\ensuremath{\mathbf{t}}}
\newcommand{\cE}{\ensuremath{\mathcal{E}}}
\newcommand{\cF}{\ensuremath{\mathcal{F}}}
\newcommand{\cG}{\ensuremath{\mathcal{G}}}
\newcommand{\II}{\ensuremath{\mathcal{I}}}
\newcommand{\LL}{\ensuremath{\mathcal{L}}}
\newcommand{\cP}{\ensuremath{\mathcal{P}}}
\newcommand{\cT}{\ensuremath{\mathcal{T}}}
\newcommand{\cV}{\ensuremath{\mathcal{V}}}
\newcommand{\cS}{\ensuremath{\mathcal{S}}}

\newcommand{\eps}{\varepsilon}
\newcommand{\xeps}{\varepsilon_{\bfx}}
\newcommand{\Om}{\ensuremath{\Omega}}
\newcommand{\supp}{\operatorname*{\rm supp}}
\newcommand{\dist}{\operatorname*{\rm dist}}

\newcommand{\lb}{\ensuremath{\left\lvert}} 
\newcommand{\rb}{\ensuremath{\right\rvert}} 
\newcommand{\sn}[1]{\ensuremath{\lb#1\rb}} 
\newcommand{\lbb}{\ensuremath{\left\lVert}} 
\newcommand{\rbb}{\ensuremath{\right\rVert}} 
\newcommand{\vn}[1]{\ensuremath{\lbb#1\rbb}} 

\newcommand{\Ba}{\ensuremath{B_\alpha}}
\newcommand{\Bb}{\ensuremath{B_\beta}}

\newcommand{\Bax}{\ensuremath{B_{\alpha \eps(\bfx)} \left( \bfx \right)}}
\newcommand{\Bbx}{\ensuremath{B_{\beta \eps(\bfx)} \left( \bfx \right)}}

\newcommand{\refK}{\ensuremath{\widehat K}}

\newcommand{\diam}{{\rm diam }}
\newcommand{\eremk}{\hbox{}\hfill\rule{0.8ex}{0.8ex}}

\title{Local high-order regularization and applications to
  \textit{hp}-methods}
\author[puc]{M.~Karkulik}
\ead{mkarkulik@mat.puc.cl}
\author[tu]{J.M.~Melenk\corref{cor1}}
\ead{melenk@tuwien.ac.at}

\address[puc]{Pontificia Universidad Cat\'olica de Chile,
  Facultad de Matem\'aticas, Vicu\~na Mackenna 4860, Santiago, Chile.}
\address[tu]{Technische Universit\"at Wien, Institut f\"ur Analysis und
Scientific Computing,  Wiedner Hauptstrasse 8-10, Vienna, Austria.}
\cortext[cor1]{Corresponding author}

\begin{document}
\begin{abstract}
  We develop a regularization operator based on smoothing on a spatially varying length scale. 
  This operator is defined for functions $u \in L^1$ and has approximation properties that are given 
  by the local Sobolev regularity of $u$ and the local smoothing length scale. Additionally, 
  the regularized function satisfies inverse estimates commensurate with the approximation orders. 
  By combining this operator with a classical $hp$-interpolation operator, we obtain an 
  $hp$-Cl\'ement type quasi-interpolation operator, i.e., an operator that requires minimal smoothness 
  of the function to be approximated but has the expected approximation properties in terms of the 
  local mesh size and polynomial degree. As a  second application, we consider  residual error
  estimates in $hp$-boundary element methods that are explicit in the local
  mesh size and the local approximation order.
\end{abstract}
\begin{keyword}
  Cl\'ement interpolant, quasi-interpolation, $hp$-FEM, $hp$-BEM
  \MSC 65N30, 65N35, 65N50
\end{keyword}
\maketitle
\section{Introduction}
The regularization (or mollification or smoothing) of a function is a basic tool in analysis
and the theory of functions, cf., e.g.,~\cite{burenkov,Hilbert_Mollifier_MCOM73}. In its simplest form on the full space $\R{d}$
one chooses a compactly supported smooth mollifier $\rho$ with $\|\rho\|_{L^1(\R{d})} = 1$, introduces 
for $\eps \in (0,1)$ the scaled mollifier
$\rho_\eps(\bfx) = \eps^{-d} \rho(\bfx/\eps)$, and defines the regularization $u_\eps$ of a function 
$u \in L^1(\R{d})$ as the convolution of $u$ with the mollifier $\rho_\eps$, i.e., 
$u_\eps:= \rho_\eps \star u$. It is well-known that this regularized function satisfies certain 
``inverse estimates'' and has certain approximation properties if the function $u$ has some Sobolev regularity. 
That is, if one denotes by $\omega_\eps:= \cup_{\bfx \in \omega} B_\eps(\bfx)$ the ``$\eps$-neighborhood'' of a 
domain $\omega$, then one has with the usual Sobolev spaces $H^s$, 
\begin{align*}
\text{inverse estimate:} &\qquad  
\|u_\eps\|_{H^{k}(\omega)} \lesssim \eps^{m-k} \|u\|_{H^m(\omega_\eps)}, 
\quad k \ge m;  \\
\text{simultaneous approximation property:} &\qquad  
\|u - u_\eps\|_{H^{k}(\omega)} \lesssim \eps^{m-k} \|u\|_{H^m(\omega_\eps)}, 
\quad 0 \leq k \leq m.  
\end{align*}
The regularized function $u_\eps$ is obtained from $u$ by an averaging of $u$ 
on a \emph{fixed} length scale $\eps$. It is the purpose of the present paper to derive analogous 
estimates for operators that are based on averaging on a \emph{spatially varying} length scale 
(see Theorem~\ref{thm:hpsmooth}).
Let us mention that averaging on a spatially varying length scale
has been used in~\cite{shankov85} to obtain an inverse trace theorem.

For many purposes of numerical analysis, the tool corresponding to the regularization technique
in analysis is quasi-interpolation. In the finite element community, such operators are often 
associated with the names of Cl\'ement \cite{clement1} or Scott \& Zhang \cite{scott1}. Many variants
exist, but they all rely, in one way or another, on averaging on a length scale that is given by the local 
mesh size. The basic results
for the space $S^{1,1}({\mathcal T})$ of continuous, piecewise linear functions on 
a triangulation ${\mathcal T}$ of a domain $\Omega$ take the following form: 
\begin{align*}
\text{inverse estimate:} &\qquad  
\|I^{Cl} u\|_{H^{k}(K)} \lesssim h_K^{-k}  \|I^{Cl} u\|_{L^2(K)} \lesssim h_K^{-k} \|u\|_{L^2(\omega_K)}, 
\qquad k \in \{0,1\},\\ 
\text{approximation property:} &\qquad  
\|u - I^{Cl} u\|_{L^{2}(K)} \lesssim h_K  \|u\|_{H^1(\omega_K)}; 
\end{align*}
here, $h_K$ stands for the diameter of the element $K \in {\mathcal T}$, and 
$\overline{\omega_K} = \cup \{\overline{K^\prime}\,|\, \overline{K^\prime} \cap \overline{K} \ne \emptyset\}$ is the patch of 
neighboring elements of $K$. 
Quasi-interpolation operators of the above type are not restricted to 
piecewise linears on affine triangulations. The literature includes many extensions and refinements 
of the original construction of \cite{clement1} to account for boundary conditions,
isoparametric elements, Hermite elements, anisotropic meshes, or hanging nodes, 
see \cite{scott1,bernardi,girault1,cc1,cchu,apel2,a01,rand12,ern-guermond15}. Explicit constants for stability or approximation estimates 
for quasi-interpolation operators are given in \cite{verfuerth1}. It is worth stressing that the typical
$h$-version quasi-interpolation operators have \emph{simultaneous} approximation properties in a scale 
of Sobolev spaces including fractional order Sobolev spaces. 

In the $hp$-version of the finite element method (or the closely related  spectral element method) 
the quasi-interpolation operator maps into the space $S^{p,1}({\mathcal T})$ of 
continuous, piecewise polynomials of arbitrarily high degree $p$ on a mesh ${\mathcal T}$.
There the situation 
concerning quasi-interpolation with $p$-explicit approximation properties in scales of Sobolev spaces 
and corresponding inverse estimates is much less developed.  In particular, for inverse estimates it is 
well-known that, in contrast to the $h$-version,
elementwise polynomial inverse estimates do not match the approximation properties
of polynomials so that some appropriate substitute needs to be found. 

In the $hp$-version finite element method, the standard approach to the construction of 
piecewise polynomial approximants on unstructured meshes 
is to proceed in two steps: In a first step, polynomial approximations are constructed for every element separately; 
in a second step, the continuity requirements are enforced by using lifting operators. 
The first step thus falls into the realm of classical approximation theory and a plethora of results are 
available there, see, e.g., \cite{devore1}. Polynomial approximation results developed in the
$hp$-FEM/spectral element literature focused mostly (but not exclusively) on $L^2$-based spaces and include 
\cite{ainsworth1,guo2,canuto1,babuska1,demkowicz-buffa05,demkowicz-cao05} and 
\cite{canuto-quarteroni82,bernardi-maday92,bernardi-maday97,bernardi-dauge-maday99,bernardi-dauge-maday07,quarteroni84}. 
The second step is concerned with removing the interelement jumps. In the $L^2$-based setting, appropriate
liftings can be found in \cite{babuska2,sola1,bernardi-dauge-maday07,bernardi-dauge-maday92} although the 
key lifting goes back at least to \cite{gagliardo57}. While optimal (in $p$) convergence rates can be obtained 
with this approach, the function to be approximated is required to have some regularity since traces on the element
boundary need to be defined. In conclusion, this route does not appear
to lead to approximation operators for functions with minimal regularity (i.e., $L^2$ or even $L^1$).
It is possible, however, to construct quasi-interpolation operators in an $hp$-context as done
in \cite{mel1}. There, the construction is performed patchwise instead of elementwise and thus
circumvents the need for lifting operators. 

The present work takes a new approach to the construction of quasi-interpolation operators suitable for an $hp$-setting. 
These quasi-interpolation operators are constructed as the concatenation of two operators, namely, a smoothing 
operator and a classical polynomial interpolation operator. The smoothing operator turns an $L^1$-function into
a $C^\infty$-function by a local averaging procedure just as in the case of constant $\eps$ mentioned at the beginning 
of the introduction. The novel aspect is that the length scale on which the averaging
is done may be linked to the {\em local} mesh size $h$ and the {\em local} approximation order $p$; essentially, 
we select the local length scale $\eps \sim h/p$. The resulting function $\II_\eps u$ is smooth, and one can quantify 
$u - \II_\eps u$ locally in terms of the local regularity of $u$ and the local length scale $h/p$. Additionally, the 
averaged function $\II_\eps u$ satisfies appropriate inverse estimates. The smooth function $\II_\eps u$ can be 
approximated by piecewise polynomials using classical interpolation operators, whose approximation properties 
are well understood. In total, one arrives at a quasi-interpolation operator.

Our two-step construction that is based on first smoothing and then employing a classical interpolation operator 
has several advantages. The smoothing operator $\II_\eps$ is defined merely in terms 
of a length scale function $\eps$ and not explicitly in terms of a mesh. Properties of the mesh are only required
for the second step, the interpolation step. Hence, quasi-interpolation operators for a variety of meshes 
including those with hanging nodes can be constructed; the requirement is that a classical interpolation operator
for smooth functions be available with the appropriate approximation properties. Also in $H^1$-conforming settings
of regular meshes (i.e., no hanging nodes), the two-step construction can lead to improved results: 
In \cite{mps13}, an $hp$-interpolation operator is constructed that leads to optimal $H^1$-conforming approximation 
in the broken $H^2$-norm under significant smoothness assumptions. The present technique allows us to reduce
this regularity requirement to the minimal $H^2$-regularity. Finally, we mention that on a technical side, 
the present construction leads to a tighter domain of dependence for the quasi-interpolant than the 
construction in \cite{mel1}. 

Another feature of our construction
is that it naturally leads to simultaneous approximation results in scales of (positive order)
Sobolev spaces. 
Such simultaneous approximations have many applications, for example in connection with
singular perturbation 
problems, \cite{melenk-wihler14}. The simultaneous approximation properties in a scale of
Sobolev spaces makes 
$hp$-quasi-interpolation operators available for (positive order) fractional order Sobolev 
spaces, which are useful in $hp$-BEM. As an application, we employ our $hp$-quasi-interpolation operator 
for the {\sl a posteriori} error estimation in $hp$-BEM (on shape regular meshes) involving the hypersingular operator, following
\cite{cmps04} for the $h$-BEM. 

Above, we stressed the importance of inverse estimates satisfied by the classical low order
quasi-interpolation operators. The smoothed function $\II_\eps u$ satisfies inverse estimates
as well. This can be used as a substitute for the lack of an inverse estimate for the
$hp$-quasi-interpolant. We illustrate
how this inverse estimate property of $\II_\eps u$ can be exploited in conjunction with
(local) approximation properties 
of $\II_\eps$ for {\sl a posteriori} error estimation in $hp$-BEM.
Specifically, we generalize the reliable $h$-BEM {\sl a posteriori} error estimator of
\cite{c97,cms01} for the single layer BEM operator to the $hp$-BEM setting. 

We should mention a restriction innate to our approach. 
Our averaging operator $\II_\eps$ is based on volume averaging. In this way, the operator can be defined on 
$L^1$. However, this very approach limits the ability to incorporate boundary conditions. We note that the classical
$h$-FEM Scott-Zhang operator \cite{scott1} successfully deals with boundary conditions by using averaging on 
boundary faces instead of volume elements. While such a technique could be employed here as well, it is beyond 
the scope of the present paper. Nevertheless, we illustrate in Theorem~\ref{thm:homogeneous-bc} and 
Corollary~\ref{cor:homogeneous-bc} what is possible within our framework of pure volume averaging. 

This work is organized as follows: In Section~\ref{sec:notation-main-results} we present the main result 
of the paper, that is, the averaging operator $\II_\eps$. This operator is defined in terms of 
a spatially varying length scale, which we formally introduce in Definition~\ref{def:varepsilon}. 
The stability and approximation properties of $\II_\eps$ are studied locally and collected in 
Theorem~\ref{thm:hpsmooth}.  The following Section~\ref{sec:applications} is devoted to applications 
of the operator $\II_\eps$. 
In Section~\ref{sec:quasi-interpolation-FEM} (Theorem~\ref{thm:qi}) 
we show how to generate a quasi-interpolation operator from $\II_\eps$ and a classical interpolation operator. 
In this construction, one has to define a length scale function from the local mesh size and the local approximation order. 
This is done in Lemma~\ref{lem:lsf}, which may be of independent interest.  
Section~\ref{sec:BEM} is devoted to {\sl a posteriori} error estimation in $hp$-BEM: 
Corollary~\ref{cor:single-layer-BEM} addresses the single layer operator and 
Corollary~\ref{cor:hypersingular-BEM} deals with the hypersingular operator. The remainder of the paper is 
devoted to the proof of Theorem~\ref{thm:hpsmooth}. Since the averaging is performed 
on a spatially varying length scale, we will require variations of the Fa\`a Di Bruno formula in 
Lemmas~\ref{lem:faa-di-bruno-1}, \ref{lem:faa-1-estimate}. We conclude the paper with Appendix~\ref{appendix} in which 
we show for domains that are star-shaped with respect to a ball, that the constants in the Sobolev embedding
theorems (with the exception of certain limiting cases) can be controlled solely in terms of 
the diameter and the ``chunkiness'' of the domain. This result is obtained by a careful tracking of constants 
in the proof of the Sobolev embedding theorem. However, since this statement does not seem to be explicitly 
available in the literature, we include its proof in Appendix~\ref{appendix}.
\section{Notation and main result}
\label{sec:notation-main-results}
Points in physical space $\R{d}$ are denoted by small boldface letters, e.g.,
$\bfx=\left( x_1, \dots, x_d \right)$.
Multiindices in $\N_0^{d}$ are also denoted by small boldface letters, e.g.,
$\bfr$, and are used for partial derivatives, e.g., $D^\bfr u$,
which have order $r = |\bfr| = \sum_{i=1}^d \bfr_i$.
We also use the notation $\bfx^\bfr = \prod_{i=1}^d x_i^{r_i}$.
A ball with radius $r$ centered at $\bfx \in \R{d}$ is denoted by
$B_r(\bfx) = \left\{ \bfy \in \R{d} \mid \sn{\bfy-\bfx} < r \right\}$, 
and we abbreviate $B_r := B_r(0)$.
For open sets $\Om \subset \R{d}$,
$C^\infty(\Om)$ is the space of functions with derivatives of all order, and
$C^\infty_0(\Om)$ is the space of functions with derivatives of all orders and
compact support in $\Om$. By $W^{r,p}(\Om)$ for $r \in \N_0$
and $p \in [1,\infty]$ we denote the standard Sobolev space of functions with distributional
derivatives of order $r$ being in $L^p(\Om)$,
with norm $\vn{u}_{r,p,\Om}^p = \sum_{\sn{\bfr} \leq r}\vn{D^\bfr u}_{L^p(\Om)}^p$
and seminorm
$\sn{u}_{r,p,\Om}^p = \sum_{\sn{\bfr} = r} \vn{D^\bfr u}_{L^p(\Om)}^p$.
We will also work with fractional order spaces:
for $\sigma \in (0,1)$ and $p \in [1,\infty)$ we define Aronstein-Slobodeckij seminorms
\begin{align*}
  \sn{u}_{\sigma,p,\Om} = \left( \int_{\Omega} \int_{\Omega}
  \frac{|u(\bfx) - u(\bfy)|^p}{|\bfx - \bfy|^{\sigma p+d}}\,d\bfy\,d\bfx \right)^{1/p}; 
\end{align*}
for $s = \lfloor s\rfloor + \sigma$ with $\lfloor s\rfloor = \sup \{n \in \N_0\,|\, n \leq s\}$ and 
$\sigma \in (0,1)$ we set $\sn{u}_{s,p,\Om}^p =
\sum_{\sn{\bfs}  = \lfloor s\rfloor} |D^\bfs u|_{\sigma,p,\Omega}^p$. 
The full norm on $W^{s,p}(\Omega)$ is given by
$\vn{u}_{s,p,\Om}^p =
\|u\|_{\lfloor s\rfloor,p,\Omega}^p + |u|_{s,p,\Omega}^p$.
By $A \lesssim B$ we mean that there is a constant $C>0$ that is independent of relevant 
parameters such as the mesh size, polynomial degree and the like with 
$A \leq C\cdot B$.

In order to state our main result, we need the following definition.

\begin{defn}[$\Lambda$-admissible length scale function]
\label{def:varepsilon}
Let $\Omega\subset\R{d}$ be a domain, and let 
  $\Lambda:= \bigl( \LL,\left( \Lambda_{\bfr} \right)_{\bfr\in\N_0^d}
  \bigr)$, where $\LL\in\R{}$ is non-negative and 
  $\left( \Lambda_{\bfr} \right)_{\bfr\in\N_0^d}$ is a sequence of
  non-negative numbers.
  A function $\eps:\Omega\rightarrow \R{}$ is called a
  \emph{$\Lambda$-admissible length scale function}, if
  \begin{enumerate}[(i)]
    \item \label{item:lsf-i}  $\eps \in C^{\infty}(\Om)$,
    \item \label{item:lsf-ii}  $0 < \eps \leq \Lambda_\mathbf{0}$,
    \item \label{item:lsf-iii}  
  $\sn{D^{\bfr}\eps} \leq \Lambda_{\bfr} \sn{\eps}^{1-\sn{\bfr}}$
  pointwise in $\Omega$ for all $\bfr\in\N_0^d$,
    \item \label{item:lsf-iv} $\eps$ is Lipschitz continuous with constant
    $\LL$.
  \end{enumerate}
\end{defn}

\begin{rem}
If $\Omega$ is assumed to be a bounded Lipschitz domain, then the condition (\ref{item:lsf-iii}) implies
$\|\varepsilon\|_{L^\infty(\Omega)} < \infty$ and the Lipschitz continuity of $\varepsilon$, i.e., 
(\ref{item:lsf-iv}). Nevertheless, 
we include items (\ref{item:lsf-ii}) and (\ref{item:lsf-iv}) in Definition~\ref{def:varepsilon} in order to 
explicitly introduce the parameters $\Lambda_{\mathbf{0}}$ and ${\mathcal L}$ as they appear frequently in the proofs below.
\eremk
\end{rem}
The following theorem is the main result of the present work. Its proof will be given in
Section~\ref{section:proof:hpsmooth} below.
\begin{thm}\label{thm:hpsmooth}
  Let $\Om\subset\R{d}$ be a bounded Lipschitz domain. Let $k_{\max}\in\N_0$ and
  $\Lambda:=\bigl( \LL,\left( \Lambda_\bfr \right)_{\bfr\in\N_0^d}\bigr)$.
  Then, there exists a constant $0<\beta<1$
  such that for every $\Lambda$-admissible
  length scale function $\eps \in C^\infty(\Omega)$ there exists a linear operator
  $\II_\eps: L^1_{loc}(\Om) \rightarrow C^\infty(\Om)$ with the following
  properties (\ref{item:thm:hpsmooth-i})--(\ref{item:thm:hpsmooth-iv}). 
  In the estimates below, $\omega \subset \Omega$ is an arbitrary open set and  
  $\omega_\eps\subset\Omega$ denotes its ``neighborhood'' given by 
  \begin{align}
\label{eq:omega_eps}
    \omega_\eps := \Om \cap \bigcup_{\bfx\in \omega} B_{\beta \eps (\bfx)}(\bfx).
  \end{align}
\begin{enumerate}[(i)]
\item 
\label{item:thm:hpsmooth-i}
  Suppose that the pair $(s,p)\in\N_0 \times [1,\infty]$ satisfies $s\leq k_{\max}+1$. Assume that 
  the pair $(r,q) \in \R{} \times [1,\infty]$ satisfies 
  $(0 \leq s\leq r\in\R{}\text{ and } q\in[1,\infty))$ or
  $(0 \leq s\leq r\in\N_0\text{ and }q\in[1,\infty])$. Then 
  \begin{align}\label{thm:hpsmooth:stab}
    \sn{\II_\eps u}_{r,q,\omega} &\leq C_{r,q,s,p,\Lambda,\Om}
    \sum_{\sn{\bfs}\leq s} \vn{\eps^{s-r+d(1/q-1/p)}D^\bfs u}_{0,p,\omega_\eps},
  \end{align}
  where $C_{r,q,s,p,\Lambda,\Om}$ depends only on $r$, $q$, $s$, $p$,
  $\LL$,
  $(\Lambda_{\bfr'})_{\sn{\bfr'}\leq \lceil r \rceil}$, and $\Omega$. 
\item 
\label{item:thm:hpsmooth-ii}
  Suppose $0\leq s\in\R{}$, $r\in\N_0$ with $s\leq r\leq k_{\max}+1$,
  and $1\leq p\leq q<\infty$. Define $\mu:=d(p^{-1}-q^{-1})$. Assume that either 
  $(r=s+\mu\text{ and }p>1)$ or $(r>s+\mu)$. Then
  \begin{align}\label{thm:hpsmooth:apx}
    \sn{u - \II_\eps u}_{s,q,\omega} &\leq C_{s,q,r,p,\Lambda,\Om}
    \sum_{\sn{\bfr}\leq r} \vn{\eps^{r-s+d(1/q-1/p)}D^\bfr u}_{0,p,\omega_\eps},
  \end{align}
  where $C_{s,q,r,p,\Lambda,\Om}$ depends only on $s$, $q$, $r$, $p$,
  $\LL$,
  $(\Lambda_{\bfs'})_{\sn{\bfs'}\leq \lceil s \rceil}$, and $\Omega$. 
\item 
\label{item:thm:hpsmooth-iii}
  Suppose $s$, $r\in\N_0$ with $s\leq r\leq k_{\max}+1$, and $1\leq p < \infty$.
  Define $\mu:=d/p$. Assume that either $(r=s+\mu\text{ and }p=1)$ or $(r>s+\mu
  \text{ and } p>1)$.
  Then 
  \begin{align}\label{thm:hpsmooth:apx:infty}
    \sn{u - \II_\eps u}_{s,\infty,\omega} &\leq C_{s,r,p,\Lambda,\Om}
    \sum_{\sn{\bfr}\leq r} \vn{\eps^{r-s-d/p}D^\bfr u}_{0,p,\omega_\eps},
  \end{align}
  where $C_{s,r,p,\Lambda,\Om}$ depends only on $s$, $r$, $p$,
  $\LL$,
  $(\Lambda_{\bfs'})_{\sn{\bfs'}\leq s}$, and $\Omega$. 
\item 
\label{item:thm:hpsmooth-iv}
  If $\eps\in C^\infty(\overline\Omega)$ and $\eps>0$ on $\overline\Omega$,
  then $\II_\eps: L^1_{loc}(\Omega) \rightarrow C^\infty(\overline\Omega)$.
\end{enumerate}
\end{thm}
A few comments concerning Theorem~\ref{thm:hpsmooth} are in order: 
\begin{rem}
\begin{enumerate}
\item 
The stability properties (part (\ref{item:thm:hpsmooth-i})) and the approximation
properties (parts (\ref{item:thm:hpsmooth-ii}), (\ref{item:thm:hpsmooth-iii})) involve 
{\em unweighted} (possibly fractional) Sobolev norms on the left-hand side and 
weighted integer order norms on the right-hand side. Our reason for admitting fractional
Sobolev spaces only on one side of the estimate (here: the left-hand side) is that in this
case the local length scale can be incorporated fairly easily into the estimate. 
\item
The pairs $(s,q)$ for the left-hand side and $(r,p)$ for the right-hand side 
in part (\ref{item:thm:hpsmooth-ii}) are linked to each other. Essentially, the parameter 
combination of $(s,q)$ and $(r,p)$ in part (\ref{item:thm:hpsmooth-ii}) is the one known from 
the classical Sobolev embedding theorems; the only possible exception 
are certain cases related to the limiting case $p = 1$. This connection to 
the Sobolev embedding theorems arises from the proof of Theorem~\ref{thm:hpsmooth}, which employs 
Sobolev embedding theorems and scaling arguments. 
\item 
In the classical Sobolev embedding theorems, the embedding into $L^\infty$-based spaces is 
special, since the embedding is actually into a space of continuous functions. 
Part (\ref{item:thm:hpsmooth-ii}) therefore excludes the case $q = \infty$, 
and some results for the special case $q = \infty$ are collected in 
part (\ref{item:thm:hpsmooth-iii}).  
\eremk
\end{enumerate}
\end{rem}
The following variant of Theorem~\ref{thm:hpsmooth} allows for the preservation of 
homogeneous boundary conditions: 
\begin{thm}
\label{thm:homogeneous-bc}
The operator $\II_\eps$ in Theorem~\ref{thm:hpsmooth} can be modified such that the following 
is true for all $u \in L^1_{loc}(\Omega)$:  
\begin{enumerate}[(i)]
\item 
\label{item:thm:homogeneous-bc-i}
The statements (\ref{item:thm:hpsmooth-i})---(\ref{item:thm:hpsmooth-iii}) of Theorem~\ref{thm:hpsmooth} 
are valid, if one 
replaces $\omega_\eps$ of the right-hand sides with 
$\widetilde \omega_\eps:= \cup_{\bfx \in \omega} B_{\beta \eps(\bfx)}(\bfx)$ and simultaneously 
replaces $u$ on the right-hand sides with $\widetilde u:= u \chi_\Omega$ (i.e., $u$ is extended by 
zero outside $\Omega$).  This assumes that $\widetilde u$ is as regular on $\widetilde \omega_\varepsilon$ 
as the
right hand-sides of~(\ref{item:thm:hpsmooth-i})---(\ref{item:thm:hpsmooth-iii}) dictate. 
\item 
\label{item:thm:homogeneous-bc-ii}
The function $\II_\eps u$ vanishes near $\partial\Omega$. More precisely, 
assuming that $\eps \in C(\overline{\Omega})$, then there is $\lambda > 0$
such that 
$$
\II_\eps u = 0 \quad \mbox{ on $\cup_{\bfx \in \partial\Omega} B_{\lambda \eps(\bfx)}(\bfx).$} 
$$
\end{enumerate}
\end{thm}
\begin{proof}
The proof follows by a modification of the proof of Theorem~\ref{thm:hpsmooth}. In 
Theorem~\ref{thm:hpsmooth}, the value $(\II_\eps v)(\bfx)$ for an $\bfx$ near $\partial\Omega$ 
is obtained by an averaging of $v$ on a ball ${\mathbf b} + B_{\delta \eps(\bfx)}(\bfx)$ where 
${\mathbf b}$ is chosen (in dependence on $\bfx$ and $\eps(\bfx))$ in a such a way that 
${\mathbf b} + B_{\delta \eps(\bfx)}(\bfx) \subset \Omega$. In order to ensure that $\II_\eps v$ vanishes
near $\partial\Omega$, one can modify this procedure: one extends $v$ by zero outside $\Omega$ and 
selects the translation ${\mathbf b}$ so that the averaging region 
${\mathbf b} + B_{\delta \eps(\bfx)}(\bfx) \subset \R{d} \setminus \Omega$. In this way, the desired condition 
(\ref{item:thm:homogeneous-bc-ii}) is ensured. The statement
(\ref{item:thm:homogeneous-bc-i}) follows in exactly the same way as in the proof of Theorem~\ref{thm:hpsmooth}. 
\end{proof}
\section{Applications to $hp$-methods} 
\label{sec:applications} 
Let $\Omega \subset \R{d}$ be a bounded Lipschitz domain.
A \emph{partition} $\cT = \left\{ K \right\}_{K\in\cT}$ of $\Om$ 
is a collection of open, bounded, and pairwise disjoint sets
such that $\overline\Omega = \cup_{K\in\cT}\overline K$.
We set 
$$
\omega_K:= \operatorname*{interior} 
\left( \cup \{\overline{K^\prime}\,|\, \overline{K^\prime } \cap \overline{K} \ne \emptyset \}\right). 
$$
In order to simplify the notation, we will write also
``$K' \in \omega_K$'' to mean $K' \in \cT$ such that $K' \subset \omega_K$.
A finite partition $\cT$ of $\Om$ is called a mesh, if
every element $K$ is the image of the reference simplex $\refK \subset \R{d}$ under 
an affine map $F_K:\refK\rightarrow K$. 
A mesh $\cT$ is assumed to be regular in the sense of Ciarlet \cite{ciarlet76a}, i.e., it is not allowed 
to have hanging nodes (this restriction is not essential and only made for convenience of 
presentation--see Remark~\ref{rem:hanging-nodes} below).
To every element $K$ we associate the mesh width $h_K := {\rm diam}(K)$,
and we define $h\in L^\infty(\Om)$ as $h(\bfx) = h_K$ for $\bfx\in K$.
We call a mesh $\gamma$-shape regular if 

\begin{align*}
  h_K^{-1}  \vn{F'_K}_{} + h_K \vn{(F'_K)^{-1}}_{} \leq \gamma
  \quad\text{ for all } K\in\cT.
\end{align*}
A $\gamma$-shape regular mesh is locally quasi-uniform, i.e.,
there is a constant $C_\gamma$ which depends only on $\gamma$ such that
\begin{align}
\label{eq:shape-regular}
  C_\gamma^{-1} h_K \leq h_{K'} \leq C_\gamma h_K
  \quad\text{ for } \overline{K}\cap\overline{K'}\neq\emptyset.
\end{align}
A polynomial degree distribution $\bfp$ on a partition $\cT$ is a multiindex
$\bfp=(p_K)_{K\in\cT}$ with $p_K \in \N_0$.
A polynomial degree distribution is said to be $\gamma_p$-shape regular
if
\begin{align}
\label{eq:shape-regular-p}
\gamma_p^{-1}(p_K +1) \leq p_{K'} +1 \leq \gamma_p (p_K+1)
  \quad\text{ for } \overline{K}\cap\overline{K'}\neq\emptyset.
\end{align}
We define a function $p\in L^\infty(\Om)$ by $p(\bfx) = p_K$ for $\bfx\in K$.
For $r \in \{0,1\}$, a mesh $\cT$, and a polynomial degree distribution 
$\bfp$ we introduce 
\begin{align}
\label{eq:Sp}
  \cS^{\bfp,r}(\cT) = \left\{ u \in H^r(\Om) \mid \forall K \in\cT:
    u|_K \circ F_K \in \cP_{p_K}(\refK) \right\},
 \mbox{ where } \cP_{p}(\refK) = 
    {\rm span}\left\{ \bfx^\bfk \mid \sn{\bfk} \leq p \right\}. 
\end{align}
The next lemma shows that shape-regular meshes and polynomial degree
distributions allow for the construction of length scale functions that are essentially 
given by $h/p$ and for which 
$\Lambda = \bigl( \LL, (\Lambda_{\bfr})_{\bfr \in \N_0^{d}} \bigr)$ depends
only on the mesh parameters $\gamma$ and $\gamma_p$  and on the domain
$\Omega$.

\begin{lem}\label{lem:lsf}
  Let $\Om\subset\R{d}$ be a bounded Lipschitz domain and $\cT$ be a partition of $\Om$.
Define, for $\delta > 0$, the extended sets $K_\delta$ by $K_\delta = \cup_{\bfx \in K}
B_{\delta}(\bfx)$. 
\begin{enumerate}[(i)]
\item 
\label{item:lem:lsf-i}
Let $\widetilde\eps \in L^\infty(\Omega)$ be piecewise constant on the partition $\cT$, i.e., 
$\widetilde \eps|_K \in \R{}$ for each $K \in \cT$. Assume that for  some 
$C_\eps$, $C_{\rm reg} > 0$, $M \in \N$, the following two conditions are satisfied: 
\begin{align}
\label{eq:lem:lsf-10}
K\in \cT \quad \Longrightarrow \quad
\operatorname*{card} \{K^\prime \in \cT \colon 
K_{C_{\rm reg} \widetilde \eps|_K} \cap K_{C_{\rm reg} \widetilde
\eps|_{K^\prime}}^\prime \neq \emptyset\} &\leq M,  \\
\label{eq:lem:lsf-20}
K, K^\prime \in \cT \mbox{ and }\  
K_{C_{\rm reg} \widetilde \eps|_K} \cap K_{C_{\rm reg} \widetilde
\eps|_{K^\prime}}^\prime \ne \emptyset &\Longrightarrow 
0 < \widetilde \eps|_K \leq C_\eps \widetilde \eps|_{K^\prime}. 
\end{align}
  Then, there exists a $\Lambda$-admissible length scale function $\eps \in C^\infty(\Omega)$ 
  such that for all $K\in\cT$ it holds
  \begin{align}
    \label{eq:lem:lsf-1}
    \eps|_K \sim \widetilde \eps|_K. 
  \end{align}
  The implied constants in~\eqref{eq:lem:lsf-1} as well as the sequence
  $\Lambda = \bigl( \LL, \left( \Lambda_\bfr \right)_{\bfr\in\N_0^d} \bigr)$
  can be controlled in terms of $\Omega$ and
  the parameters $C_\eps$, $C_{\rm reg}$, $\vn{\widetilde
  \eps}_{L^{\infty}(\Omega)}$, and $M$ only. 
\item 
\label{item:lem:lsf-ii}
Let $\cT$ be a $\gamma$-shape regular mesh on a bounded Lipschitz domain $\Omega$. 
Let $\bfp$ be a $\gamma_p$-shape regular polynomial degree distribution. Then there exists 
a $\Lambda$-admissible length scale function $\eps \in
C^\infty(\overline\Omega)$ with $\eps>0$ on $\overline\Omega$ such that for every $K \in \cT$
\begin{equation}
\label{eq:lem:lsf-50} 
\eps|_K \sim \frac{h_K}{p_K+1}
\qquad \mbox{ and } \qquad K_{\eps}:=\cup_{\bfx \in K} B_{\eps(\bfx)}(\bfx)  \subset \omega_K.   
\end{equation}
The implied constants in~\eqref{eq:lem:lsf-50} as well as the sequence
$\Lambda = \bigl( \LL, \left( \Lambda_\bfr \right)_{\bfr\in\N_0^d} \bigr)$
depend solely on the shape-regularity parameters $\gamma$ and $\gamma_p$ as well as $\Omega$. 
\end{enumerate}
\end{lem}
\begin{proof}
{\em Proof of (\ref{item:lem:lsf-i}):} 
Introduce the abbreviation $\delta_K:= \frac{1}{2} C_{\rm reg} \widetilde \eps|_K$.  
Let $\rho$ be the standard mollifier given by $\rho(\bfx) = C_\rho\exp(-1/(1-\sn{\bfx}^2))$
  for $\sn{\bfx}<1$; the constant $C_\rho>0$ is chosen such that $\int_{\R{d}}\rho(\bfx)\,d\bfx=1$,
  and $\rho_\delta(\bfx) := \delta^{-d}\rho(\bfx/\delta)$. Define, with $\chi_A$ denoting the characteristic
function of the set $A$,
\begin{equation}
\label{eq:lem:lsf-100}
\eps:= \sum_{K \in \cT} \widetilde \eps|_{K} \rho_{\delta_K} \star \chi_{K_{\delta_K}}. 
\end{equation}
By assumption, the sum is locally finite with not more than $M$ terms in the sum for each $x \in \Omega$. 
In fact, when restricting the attention to 
$\operatorname*{supp} \rho_{\delta_{K^\prime}} \star \chi_{K^\prime_{\delta_{K^\prime}}}$ for a fixed 
$K^\prime$, the sum reduces to a finite one with at most $M$ terms. Hence, 
$\eps \in C^\infty(\Omega)$. Furthermore,  
$\rho_{\delta_{K^\prime}} \star \chi_{K^\prime_{\delta_{K^\prime}}} \equiv 1$ on $K^\prime$. Since 
all terms in the sum (\ref{eq:lem:lsf-100}) are non-negative, we conclude 
$$
\eps|_K \ge \widetilde \eps|_K \qquad \forall K \in \cT. 
$$
Since the sum is locally finite, we get in view of 
(\ref{eq:lem:lsf-20}) that $\eps|_K \leq (1 + (M-1)C_\eps) \widetilde \eps|_K$. To see that 
$\eps$ is $\Lambda$-admissible, we note for each $K$ that on 
$\widetilde K:= \operatorname*{supp} (\rho_{\delta_K} \star \chi_{K_{\delta_K}})$ the sum reduces to 
at most $M$ terms where the corresponding values $\widetilde \eps|_{K^\prime}$ are all comparable to 
$\widetilde \eps|_K$ by (\ref{eq:lem:lsf-20}).
 The observation
$D^\bfr (u\star\rho_{\delta}) = u\star (D^\bfr
\rho)_{\delta}\delta^{-\sn{\bfr}}$ implies the key property (\ref{item:lsf-iii}) of Definition~\ref{def:varepsilon}.
In particular, $\|\nabla \varepsilon\|_{L^\infty(\Omega)} < \infty$. Since $\Omega$ is a bounded Lipschitz domain, 
this implies the Lipschitz continuity of $\varepsilon$ (i.e., (\ref{item:lsf-iv}) of Definition~\ref{def:varepsilon}) 
with ${\mathcal L}$ depending only on $\|\nabla \varepsilon\|_{L^\infty(\Omega)}$ and $\Omega$.

{\em Proof of (\ref{item:lem:lsf-ii}):} We apply (\ref{item:lem:lsf-i}). We define 
$\widetilde \eps$ elementwise by $\widetilde \eps|_K:= h_K/(p_K+1)$. 
By the $\gamma$-shape regularity of $\cT$ and since $\widetilde \eps |_K \leq h_K$ for all $K \in \cT$, 
we can find $C_{\rm reg}$, which depends solely on $\gamma$, such that 
$\Omega \cap K_{C_{\rm reg} \widetilde \eps|_K} \subset \omega_K$ for every $K \in \cT$. This readily implies 
(\ref{eq:lem:lsf-10}). The $\gamma_p$-shape regularity of the polynomial 
degree distribution then ensures (\ref{eq:lem:lsf-20}). An application of (\ref{item:lem:lsf-i}) yields 
the desired $\Lambda$-admissible length scale function $\eps$ such that the first condition 
in (\ref{eq:lem:lsf-50}) is satisfied. The final step consists in multiplying the thus obtained $\eps$ 
with $C_{\rm reg}^{-1}$ to ensure the second condition in (\ref{eq:lem:lsf-50}). 
If the triangulation is finite, then $\inf_{x \in \Omega} \widetilde \eps > 0$, and this implies 
$\eps \in C^\infty(\overline{\Omega})$ and $\eps>0$ on $\overline\Omega$.
\end{proof}
\begin{rem}
Lemma~\ref{lem:lsf}, (\ref{item:lem:lsf-ii}) is formulated for meshes,
i.e., for affine, simplicial partitions.
However, since Lemma~\ref{lem:lsf}, (\ref{item:lem:lsf-ii}) relies heavily on
Lemma~\ref{lem:lsf}, (\ref{item:lem:lsf-i}), which allows for very general partitions of $\Omega$, 
analogous results can be formulated for partitions based on quadrilaterals 
or partitions that include curved elements. 
\eremk
\end{rem}
\subsection{Quasi-interpolation operators}
\label{sec:quasi-interpolation-FEM}
The following theorem shows how to construct a quasi-interpolation operator
by combining the smoothing operator $\II_\eps$ from Theorem~\ref{thm:hpsmooth} 
with a classical interpolation operator, which is allowed to require significant smoothness (e.g., 
point evaluations). 
\begin{thm}\label{thm:qi}
  Let $\Om\subset\R{d}$ be a bounded Lipschitz domain. Let $\cT$
  be a $\gamma$-shape regular mesh on $\Om$, and let $\bfp$ be a $\gamma_p$-shape regular
  polynomial degree distribution on $\cT$ with $p_K \ge 1$ for all $K \in \cT$.
  
  Suppose that $p\in[1,\infty)$, $r'\in\N_0$, and
  $\Pi^{hp}:W^{r',p}(\Om)\rightarrow \cS^{\bfp,1}(\cT)$ is bounded and
  linear. Then, there exists a linear operator
  $\II^{hp}:L^1_{\rm loc}(\Om)\rightarrow\cS^{\bfp,1}(\cT)$ with the following approximation properties: If
$\Pi^{hp}$ has the local approximation property 
  \begin{align}\label{cor:hp:apx}
    \sn{u-\Pi^{hp} u}_{s,p,K}^p \leq
    C_{\Pi} \left( \frac{h_K}{p_K} \right)^{p(r'-s)}\vn{u}_{r',p,\omega_K}^p  
\qquad \forall K \in \cT
  \end{align}
  for some $s \in \N_0$ with $0 \leq s \leq r'$, then for all $r \in \N_0$ with $s \leq r \leq r^\prime$
  \begin{align*}
    \sn{u-\II^{hp} u}_{s,p,K}^p \leq
    C C_{\Pi} \left( \frac{h_K}{p_K} \right)^{p(r-s)}
    \sum_{K' \in \omega_K} \vn{u}_{r,p,\omega_{K'}}^p 
\qquad \forall K \in \cT. 
  \end{align*}
  The constant $C$ depends only on the mesh parameters $\gamma$, $\gamma_p$, 
  and on $s$, $r$, $r^\prime$, $p$, and $\Omega$. 
\end{thm}
\begin{proof}
  Choose $\eps$ from Lemma~\ref{lem:lsf} and $k_{\max} = r'-1$.
  We use the operator $\II_\eps$ of Theorem~\ref{thm:hpsmooth} and define
  $\II^{hp} := \Pi^{hp}\circ \II_\eps$. Then, according
  to~\eqref{thm:hpsmooth:apx} and~\eqref{cor:hp:apx}, it holds
  \begin{align*}
    \sn{u-\II^{hp}u}_{s,p,K} &\leq
    \sn{u-\II_\eps u}_{s,p,K} + \sn{\II_\eps u - \Pi^{hp}\II_\eps u}_{s,p,K}
\lesssim 
    \left( \frac{h_K}{p_K} \right)^{r-s}\vn{u}_{r,p,\omega_K}
    +\left( \frac{h_K}{p_K} \right)^{r'-s}\vn{\II_\eps u}_{r',p,\omega_K}. 
  \end{align*}
  Theorem~\ref{thm:hpsmooth} implies 
  \begin{align*}
    \vn{\II_\eps u}_{r',p,\omega_K}^p
    &\lesssim \sum_{K'\in\omega_{K}}
    \left(
    \sum_{j=0}^{r}\sn{\II_\eps u}_{j,p,K'}^p
    + \sum_{j=r}^{r'}\sn{\II_\eps u}_{j,p,K'}^p \right)
\lesssim \sum_{K'\in\omega_{K}}
    \left( 
      \sum_{j=0}^{r}\sn{u}_{j,p,\omega_{K'}}^p
      + \sum_{j=r}^{r'}\left( \frac{h_K}{p_K} \right)^{p(r-j)}\sn{u}_{r,p,\omega_{K'}}^p
      \right)\\
    &\lesssim \sum_{K'\in\omega_{K}}
    \left(\frac{h_K}{p_K} \right)^{p(r-r')} \vn{u}_{r,p,\omega_{K'}}^p,
  \end{align*}
  where we additionally used that $K_\eps \subset \omega_K$ due to~\eqref{eq:lem:lsf-1}.
  This shows the result.
\end{proof}

\begin{rem}
Theorem~\ref{thm:qi} is formulated so as to accommodate several types of operators $\Pi^{hp}$ that can
be found in the literature. In this remark, we focus on the $H^1$-setting $s = 1$, $p = 2$
and emphasize the high order aspect. The available operators can be divided into two groups.
In the first group, the operator $\Pi^{hp}$ is constructed
in two steps: in a first step, a discontinuous piecewise polynomial is constructed in an elementwise fashion,
disregarding the interelement continuity requirement. Typically, this is achieved by some (local) projection or an
interpolation. In a second step, the interelement jumps are corrected. For this, polynomial liftings
are required. Early polynomial liftings include \cite{babuska1,babuska-suri87a}, which did not
have the optimal lifting property $H^{1/2}(\partial K) \rightarrow H^1(K)$. We hasten to add that, given sufficient
regularity of the function to be approximated, these lifting still lead to the optimal rates both in $h$ and $p$.
Interelement corrections based on liftings $H^{1/2}(\partial K) \rightarrow H^1(K)$ were constructed subsequently in
\cite{babuska2}
for $d = 2$ and \cite{sola1} for $d = 3$; closely related liftings can be found in
\cite{maday89,belgacem94,bernardi-dauge-maday07,mel1,demkowicz-gopalakrishnan-schoeberl08}.
Since the construction of $\Pi^{hp}$ is based on joining approximations on neighboring elements,
the approximation on $K$ is defined in terms of the function values on $\omega_K$.
The second group of operators reduces the domain of dependence: $(\Pi^{hp} u)|_K$ is determined
by $u|_{\overline{K}}$. This is achieved by requiring, for example in the case $d = 3$, the following
conditions to ensure $C^0$-interelement continuity:
The value $(\Pi^{hp} u)(V)$ in each vertex $V$ is given by $u(V)$; the values
$(\Pi^{hp} u)|_e$ for each edge $e$ are completely determined by $u|_e$; the values
$(\pi^{hp} u)|_f$ for each face $f$ are completely determined by $u|_f$.
Such a procedure is worked out in the {\em projection-based interpolation}
approach \cite{demkowicz08,demkowicz-kurtz-pardo-paszynski-rachowicz-zdunek08}
and in \cite{melenk-sauter10,mps13} and requires $r^\prime > d/2$. We point out 
that the procedure of \cite{mps13} is refined in Corollary~\ref{cor:MPS} below.
We mention that several polynomial approximation operators were developed in the contexts
of the {\em spectral method} and the {\em spectral element method}
(see, e.g., \cite{bernardi-maday97,canuto1}).

It is worth stressing that the two approaches outlined above can successfully deal with non-affine elements and that
they are not restricted to so-called conforming meshes. That is, meshes with ``hanging nodes'' can be dealt
with (see, e.g., \cite{oden-demkowicz-rachowicz-hardy89,demkowicz-kurtz-pardo-paszynski-rachowicz-zdunek08}
and \cite[Sec.~{4.5.3}]{schwab1}).
\eremk
\end{rem}
As an application of Theorem~\ref{thm:qi}, we construct an
interpolation operator with simultaneous approximation properties
on regular meshes. The novel feature of the operator of Corollary~\ref{cor:MPS}
is that it provides the optimal rate of convergence in the broken $H^2$-norm, which 
can be of interest in the analysis of $hp$-Discontinuous Galerkin methods. 
\begin{cor}
\label{cor:MPS}
  Let $\Om\subset\R{d}$, $d \in \{2,3\}$ be a polygonal/polyhedral domain. 
  Fix $r_{\max} \in \N_0$. Let $\cT$ be a $\gamma$-shape regular
  mesh on $\Om$ and $\bfp$ be a $\gamma_p$-shape regular polynomial degree distribution 
  with $p_K \ge 1$ for all $K \in \cT$.
  Define 
  $\widehat p_K:= \min\{p_{K^\prime}\,|\, K^\prime \subset \omega_K\}$. 
  Then, there is a linear  operator
  $\II^{hp}:L^1_{\rm loc}(\Omega)\rightarrow \cS^{\bfp,1}(\cT)$ such that 
  for every $r \in \{0,1,\ldots,r_{\max}\}$ 
  \begin{align}\label{eq:qi}
    \sn{u-\II^{hp}u}_{\ell,2,K} \leq C h_K^{\min\{\widehat p_K+1,r\}-\ell} p_K^{-(r-\ell)}
    \sum_{K'\in\omega_K}\vn{u}_{r,2,\omega_{K'}}
    \quad\text{ for } \ell = 0,1,\ldots,\min\{2,r\}. 
  \end{align}
The constant $C$ depends only on the shape regularity constants $\gamma$, $\gamma_p$, on 
$r_{\max}$, and on $\Omega$. 
\end{cor}
\begin{proof}
See Appendix~\ref{sec:appendixB} for details.
\end{proof}
\begin{rem}(non-regular meshes/hanging nodes)
\label{rem:hanging-nodes}
The meshes in Theorem~\ref{thm:qi} and Corollary~\ref{cor:MPS} are assumed to be regular, i.e., no hanging
nodes are allowed. Furthermore, the meshes are assumed to be affine and simplicial. The proof of Theorem~\ref{thm:qi} 
shows that these restrictions are not essential: It relies on the smoothing operator $\II_\eps$ (which is 
essentially independent of the meshes) and some suitable polynomial approximation operator on the mesh $\cT$. 
If a polynomial approximation operator is available on meshes with hanging nodes, or on meshes that contain
other types of elements (e.g., quadrilaterals) or non-affine elements, 
then similar arguments as in the proof of Theorem~\ref{thm:qi} can be applied. 
\eremk
\end{rem}
So far, we have only made use of Theorem~\ref{thm:hpsmooth}. Its modification, 
Theorem~\ref{thm:homogeneous-bc}, allows for the incorporation of boundary conditions. It is worth pointing 
out that regularity of the zero extension $\chi_{\Omega} u$ of $u$ is required, which limits the useful
parameter range. Nevertheless, Theorem~\ref{thm:homogeneous-bc} allows us to develop an $hp$-Cl\'ement 
interpolant that preserves homogeneous boundary conditions: 
\begin{cor}
\label{cor:homogeneous-bc}
Let the hypotheses on the mesh $\cT$ and polynomial degree distribution $\bfp$ be as in Theorem~\ref{thm:qi}. Then 
there exists a linear operator $\II^{hp}:L^1_{loc}(\Omega) \rightarrow \cS^{\bfp,1}(\cT) \cap H^1_0(\Omega)$ 
such that for all $K \in \cT$ and all $u \in H^1_0(\Omega)$: 
\begin{eqnarray*}
\|\II^{hp} u \|_{0,2,K} &\leq& C \|u\|_{0,2,\omega_K}, \\
|u - \II^{hp} u |_{\ell,2,K} &\leq& C \left(\frac{h_K}{p_K}\right)^{1-\ell}
\sum_{K'\in\omega_K}\|u\|_{1,2,\omega_{K'}}, 
\qquad \ell \in \{0,1\}.
\end{eqnarray*}
The constant $C$ depends only on the mesh parameters $\gamma$, $\gamma_p$ and $\Omega$. 
\hfill\qed
\end{cor}
\subsection{Residual error estimation in $hp$-boundary element methods}
\label{sec:BEM}
Let $\Gamma:=\partial\Omega$ be the boundary of a bounded Lipschitz domain 
$\Omega \subset \R{d}$, $d \in \{2,3\}$. Assume that $\Gamma$ is connected. 
If $d=2$, we assume additionally $\diam(\Om)<1$.
Two basic 
problems in boundary element methods (BEM) involve the \emph{single layer operator 
$\operatorname*{V}:H^{-1/2}(\Gamma) \rightarrow H^{1/2}(\Gamma)$}
and the \emph{hypersingular operator $\operatorname*{D}:H^{1/2}(\Gamma) \rightarrow H^{-1/2}(\Gamma)$}. 
We refer to~\cite{Costabel_88_BIO,hw08,mclean00,Nedelec_88_AIE} for a detailed discussion of these
operators and to the monographs \cite{s08,sasc:11} for boundary element methods in general.
In the simplest BEM settings, one studies the following two problems: 
\begin{align}
\label{eq:single-layer-equation}
\mbox{ Find $\varphi \in H^{-1/2}(\Gamma)$ s.t. } 
& \operatorname*{V} \varphi = f; \\
\label{eq:hypersingular-equation}
\mbox{ Find $u \in H^{1/2}(\Gamma)$ s.t. } 
& \operatorname*{D} u = g;  
\end{align}
here, the right-hand sides are given data with 
$f \in H^{1/2}(\Gamma)$ and $g \in H^{-1/2}(\Gamma)$ such that $\langle g,1\rangle_\Gamma = 0$. 
In a conforming Galerkin setting, one takes finite-dimensional subspaces $V_N \subset H^{-1/2}(\Gamma)$ 
and $W_N \subset H^{1/2}(\Gamma)$ and defines the Galerkin approximations $\varphi_N \in V_N$ and 
$u_N \in W_N$ by 
\begin{align}
\label{eq:BEM-V}
\mbox{ Find $\varphi_N \in V_N$ s.t. } \langle \operatorname*{V} \varphi_N,v\rangle_\Gamma = 
\langle f,v\rangle_\Gamma \qquad \forall v \in V_N,  \\
\label{eq:BEM-D}
\mbox{ Find $u_N \in W_N$ s.t. } \langle \operatorname*{D} u_N,v\rangle_\Gamma = 
\langle g,v\rangle_\Gamma \qquad \forall v \in W_N.  
\end{align}
Residual {\sl a posteriori} error estimation for these Galerkin approximations is based on bounding 
\begin{align}
\label{eq:bem-residuals}
\|f - \operatorname*{V} \varphi_N\|_{1/2,2,\Gamma} 
\qquad \mbox{ and } 
\qquad 
\|g - \operatorname*{D} u_N\|_{-1/2,2,\Gamma}. 
\end{align}
These norms are non-local, which results in two difficulties.
First, it makes them hard to evaluate in a computational 
environment. Second, they cannot be used as indicators for local mesh refinement.
The ultimate goal of residual error estimation is to obtain a fully localized,
computable error estimator based on the equation's residual.
Several residual error estimators have been presented in an $h$-version BEM context, 
for example, 
\cite{rank86a,rank89,r93,cs95,cs96,c97,cms01,cmps04}.
The localization of norms in the $hp$-BEM context is a more delicate question,
and to our knowledge there are no results on fully localized residual error
estimates. The only result that we are aware of is~\cite{cfs96},
were it is shown that the error can be bounded reliably by
the product of two localized residual error estimators.
In the following, we generalize the local residual estimators of \cite{cms01} (for the single layer operator) 
and \cite{cmps04} (for the hypersingular operator), which were developed in an $h$-BEM setting, to the 
$hp$-BEM in Corollaries~\ref{cor:single-layer-BEM} and \ref{cor:hypersingular-BEM}.
\subsubsection{Spaces and meshes on surfaces}

\subsubsection*{Sobolev spaces and local parametrizations by charts}

The Sobolev spaces on surfaces $\Gamma:=\partial\Omega$ for bounded
Lipschitz domains $\Omega \subset \R{d}$ 
are defined as in \cite[p.~{89}, p.~{96}]{mclean00}, using the fact that, locally, $\Gamma$ is a hypograph. 
We recall some facts relevant for our purposes: There are Lipschitz continuous functions 
$\widetilde \chi_j:\R{d-1} \rightarrow \R{}$, $j=1,\ldots,n$, Euclidean changes of coordinates 
$Q_j:\R{d} \rightarrow \R{d}$, $j=1,\ldots,n$, as well as an open cover 
${\mathcal U} = \{U_j\}_{j=1}^n$ of $\Gamma$ such that the 
bi-lipschitz mappings $\widehat \chi_j:\R{d-1} \times \R{} \rightarrow \R{d}$ given by 
$(x,t) \mapsto Q_j(x,\widetilde \chi_j(x)+t)$ satisfy 
$\widehat \chi_j(\R{d-1} \times \{0\}) \cap U_j = \Gamma \cap U_j$, $j=1,\ldots,n$. Define 
$V_j:= \widehat\chi_j^{-1} (U_j \cap \Gamma) \subset \R{d-1} \times \{0\}$ and identify this set in the canonical
way with a subset of $\R{d-1}$. At the same time, define $\chi_j:V_j \rightarrow \Gamma$ by 
$\chi_j(x):= \widehat \chi_j(x,0)$, which are local parametrizations of the boundary $\Gamma$. 
Note that the maps $\chi_j$ are bi-lipschitz maps between $V_j$ and $\chi_j(V_j)$. 

An important observation to be made is for polygonal/polyhedral domains $\Omega$: 
\begin{rem} 
\label{rem:polyhedra}
Let the bounded Lipschitz domain $\Omega$ be a polyhedron, i.e., the intersection of half-spaces
defined by hyperplanes. Let the boundary $\Gamma = \partial\Omega$ be comprised of ``faces'' 
$\Gamma_i$, $i=1,\ldots,N$, i.e., pieces of hypersurfaces. Then the maps $\widehat \chi_j$, $j=1,\ldots,n$,
that define the local parametrizations can be chosen to be piecewise affine. More precisely: 
If $\Gamma_i \cap U_j \ne \emptyset$, then the restriction 
$\chi_j: \chi_j^{-1}(\Gamma_i \cap U_j)  \rightarrow \Gamma_i \cap U_j$ is affine. 
\eremk
\end{rem}

Let $\{\beta_j\}_{j=1}^n$ be a smooth partition of unity on $\Gamma$
subordinate to the cover ${\mathcal U}$. The Sobolev spaces $H^s(\Gamma)$, $s \in [0,1]$,  are then 
defined by the norm 
$\vn{u}_{s,2,\Gamma}^2 := \sum_{j=1}^n\vn{(\beta_ju)\circ\chi_j}_{s,2,V_j}^2$.
The space $L^2(\Gamma)$ is defined equivalently with respect
to the surface measure $d\sigma$, and 
the space $H^1(\Gamma)$ is defined equivalently via the norm
$\vn{u}_{1,2,\Gamma}^2 := \vn{u}_{0,2,\Gamma}^2 + \vn{\nabla_\Gamma u}_{0,2,\Gamma}^2$,
where $\nabla_\Gamma$ denotes the surface gradient.
The space $L^2(\Gamma)$ is equipped with the scalar product
$\int_\Gamma u\cdot v\;d\sigma$. The space $H^{-1/2}(\Gamma)$ is the dual space
of $H^{1/2}(\Gamma)$ with respect to the extended $L^2(\Gamma)$-scalar product.
Additionally, spaces of fractional order can be defined via interpolation
between $L^2(\Gamma)$ and $H^1(\Gamma)$ with norm equivalent to $\|\cdot\|_{s,2,\Gamma}$.

\subsubsection*{Meshes and piecewise polynomial spaces}
As in the case of volume discretizations discussed above, we restrict our attention to affine, 
regular (in the sense of Ciarlet) 
triangulations of $\Gamma$. A triangulation $\cT$ of $\Gamma$ is a partition of $\Gamma$ into 
(relatively) open disjoint elements $K$. Every element $K$ is the image of the reference 
simplex $\refK \subset \R{d-1}$ under an affine element map
$F_K:\refK \rightarrow K \subset \Gamma \subset \R{d}$. 
The mesh width $h_K$ is given by $h_K: = \operatorname*{diam}(K)$.
Since the (affine) element maps $F_K$ map from 
$\R{d-1}$ to $\R{d}$, the shape-regularity requirement takes the following form: 
\begin{equation}
\label{eq:shape-regularity-Gamma}
h_K^{-1} \|F^\prime_K\|_{}  + h_K^2 \|\left( (F_K^\prime)^\top  F_K^\prime\right)^{-1}\|_{} \leq \gamma 
\quad \mbox{ for all $K \in \cT$.}
\end{equation}
The local comparability (\ref{eq:shape-regular}) is ensured by (\ref{eq:shape-regularity-Gamma}).
The spaces $\cS^{\bfp,0}(\cT)$ and $\cS^{\bfp,1}(\cT)$ are defined as in (\ref{eq:Sp}),
but with $\Gamma$ instead of $\Omega$. We will also require the local comparability
of the polynomial degree spelled out in (\ref{eq:shape-regular-p}). As at the outset of 
Section~\ref{sec:applications}, we denote by $h$ and $p$ the piecewise constant functions given by 
$h|_K = h_K$ and $p|_K = p_K$.
\subsubsection{Single layer operator}
\begin{lem}\label{bem:smooth}
  Let $\Omega\subset\R{d}$ be a bounded Lipschitz domain and let $\Gamma=\partial\Om$ be its boundary.
  Let $\cT$ be a $\gamma$-shape regular mesh on $\Gamma$, and let $\bfp$ be a $\gamma_p$-shape regular 
  polynomial degree distribution on $\cT$. Then it holds
  \begin{align*}
    \vn{u}_{1/2,2,\Gamma} \lesssim \vn{h^{-1/2}\hat p^{1/2}u}_{0,2,\Gamma}
    + \vn{h^{1/2}\hat p^{-1/2}\nabla_\Gamma u}_{0,2,\Gamma}
    \quad\text{ for } u\in H^1(\Gamma),
  \end{align*}
  where $\hat p := \max(1,p)$
  and the hidden constant depends only on $\gamma$ and $\gamma_p$.
\end{lem}
\begin{proof}

  Define $\cT_j := \{ \chi_j^{-1}(K\cap U_j) \mid K\in\cT \}$, $j=1,\ldots,n$. 
  As $\chi_j$ is bi-lipschitz, $\cT_j$ is a partition of 
  $V_j$. Set $\nu := \min\{1, \min_{j=1,\ldots,n} \dist(\supp(\beta_j \circ \chi_j),\partial V_j)\}$. 
Define on the partition ${\mathcal T}_j$ the function $\widetilde\eps \in L^\infty(V_j)$ 
  by $\widetilde\eps|_{\chi_j^{-1}(K\cap U_j)}:= \min\{\nu,h_K/(p_K+1)\}$, $K \in \cT$, where 
$h_K$ and $p_K$ are the element diameter and polynomial degree of 
$K \in \cT$. 

We wish to apply Lemma~\ref{lem:lsf}. Using the shape-regularity of ${\mathcal T}$ select $C>0$ such that for all 
$K$, $K^\prime \in {\mathcal T}$ the following holds: $K_{C h_K} \cap K^\prime_{C h_{K^\prime}} \ne \emptyset 
\Longrightarrow \overline{K} \cap \overline{K^\prime} \ne \emptyset$. 
(We set $K_{\delta}:=\cup_{x \in K} B_\delta(x)\subset \R{d}$.) Let $\widehat K:= \chi_j^{-1}(K \cap U_j)$
and $\widehat K^\prime:= \chi_{j}^{-1}(K^\prime \cap U_j)$. 
For $C_1 > 0$ sufficiently small (depending only on $\chi_j$ and the shape-regularity of ${\mathcal T}$) 
we claim that $\widehat K_{C_1 C h_K} \cap \widehat K^\prime_{C_1 C h_{K^\prime}} \ne \emptyset$ implies 
$\overline{K} \cap \overline{K^\prime} \ne \emptyset$. To see this, 
let $\widehat x \in \widehat K$ and 
$\widehat x^\prime \in \widehat K^\prime$ with 
$B_{C_1 C h_K}(\hat x) \cap B_{C_1 C h_{K^\prime}}(\widehat x^\prime) \ne \emptyset$. Since $\chi_j$ is 
bilipschitz, the corresponding points $x = \chi_j(\widehat x) \in K$, 
$x^\prime  = \chi_j(\widehat x^\prime) \in K^\prime$ satisfy 
$\dist(K,K^\prime) \leq \dist(x,x^\prime) \lesssim \dist(\widehat x,\widehat x^\prime) \lesssim 
C_1 C (h_K + h_{K^\prime})$. If $C_1$ is sufficiently small, the shape-regularity of ${\mathcal T}$ 
then implies that this last estimate in fact ensures $\overline{K} \cap \overline{K^\prime} \ne \emptyset$. 
Now that the key condition 
``$\widehat K_{C_1 C h_K} \cap \widehat K^\prime_{C_1 C h_{K^\prime}} \ne \emptyset \Longrightarrow 
\overline{K} \cap \overline{K^\prime} \ne \emptyset$'' is proved, and since the 
function $\widetilde \varepsilon$ satisfies $\widetilde \varepsilon|_K \leq h_K$ for all $K \in {\mathcal T}$, 
we see that by selecting the parameter $C_{\rm reg}$ in Lemma~\ref{lem:lsf} sufficiently small
the conditions (\ref{eq:lem:lsf-10}), (\ref{eq:lem:lsf-20}) can be met; the $\gamma_p$-shape regularity of 
$\bfp$ has to be used as well. 

Hence, Lemma~\ref{lem:lsf}, (\ref{item:lem:lsf-i}) 
  provides a family $\left\{ \eps_j \right\}_{j=1}^n$
  of $\Lambda$-admissible length scale functions on the sets $V_j$ with
  $\eps_j|_{\chi_j^{-1}(K\cap U_j)} \sim \min\{\nu,h_K / p_K\}$ 
  whenever $\chi_j^{-1}(K \cap U_j) \ne \emptyset$. In fact, multiplying by a constant if necessary we may additionally assume 
  that $\eps_j \leq \nu$ on $V_j$. 
  With Theorem~\ref{thm:hpsmooth} we construct operators
  $\II_{\eps_j}:L^1_{\rm loc}(V_j)\rightarrow C^\infty(V_j)$, and the condition $\eps_j \leq \nu$
  shows $\supp(\II_{\eps_j}(\beta_j u\circ\chi_j)) \subset V_j$. 
  Theorem~\ref{thm:hpsmooth} then shows 
  \begin{align*}
    \vn{\beta_j u\circ \chi_j}_{1/2,2,V_j}
    &\leq
    \vn{\II_{\eps_j}(\beta_j u\circ \chi_j)}_{1/2,2,V_j}
    + \vn{\beta_j u\circ \chi_j -
      \II_{\eps_j}(\beta_ju\circ\chi_j)}_{1/2,2,V_j}
\\ &
\lesssim
    \vn{\eps_j^{-1/2} \beta_ju\circ\chi_j}_{0,2,V_j}
    + \vn{\eps_j^{1/2}\nabla(\beta_ju\circ\chi_j)}_{0,2,V_j}
\lesssim
    \vn{\eps_j^{-1/2} u\circ\chi_j}_{0,2,V_j}
    + \vn{\eps_j^{1/2}\nabla(u\circ\chi_j)}_{0,2,V_j}\\
    &\lesssim \vn{h^{-1/2}\hat p^{1/2}u}_{0,2,U_j\cap\Gamma}
    + \vn{h^{1/2}\hat p^{-1/2}\nabla_\Gamma u}_{0,2,U_j\cap\Gamma},
  \end{align*}
  where we used that $\chi_j$ is bi-lipschitz. 
  Taking the sum over $j$ concludes the proof.
\end{proof}
The following can be seen as a generalization of the results
of~\cite{c97,cms01} to obtain a residual {\sl a posteriori}
error estimator for $hp$-boundary elements for weakly singular integral equations.
\begin{thm}
\label{thm:single-layer}
  Let $\Omega\subset\R{d}$ be a bounded Lipschitz domain and let $\Gamma=\partial\Om$ be its boundary.
  Let $\cT$ be a $\gamma$-shape regular mesh on $\Gamma$, and let $\bfp$ be a $\gamma_p$-shape regular 
  polynomial degree distribution on $\cT$.  Suppose that $u\in H^1(\Gamma)$ satisfies
  \begin{align}\label{cor:res:galorth}
    \int_\Gamma u\cdot\phi_{hp}\,d\sigma=0
    \quad\text{ for all }\phi_{hp}\in\cS^{\bfp,0}(\cT).
  \end{align}
  Then, with $\hat p := \max(1,p)$, 
  \begin{align*}
    \vn{u}_{1/2,2,\Gamma} \leq C_{\gamma,\gamma_p}
    \vn{h^{1/2}\hat p^{-1/2}\nabla_\Gamma u}_{0,2,\Gamma},
  \end{align*}
  where the constant $C_{\gamma,\gamma_p}$ depends only on the shape-regularity constants, 
  $\gamma$, $\gamma_p$, and on $\Gamma$. 
\end{thm}
\begin{proof}
  Denote by $\Pi$ the $L^2(\Gamma)$-orthogonal projection onto $\cS^{\bfp,0}(\cT)$.
  Then, due to the orthogonality~\eqref{cor:res:galorth}, it holds
  \begin{align*}
    \vn{h^{-1/2}\hat p^{1/2}u}_{0,2,\Gamma} = 
    \vn{h^{-1/2}\hat p^{1/2}(1-\Pi)u}_{0,2,\Gamma}
    \lesssim \vn{h^{1/2}\hat p^{-1/2}\nabla_\Gamma u}_{0,2,\Gamma},
  \end{align*}
  where the last estimate follows from well-known approximation properties of the elementwise 
  $L^2$-projection. Finally, Lemma~\ref{bem:smooth} concludes the proof.
\end{proof}
We explicitly formulate the residual error estimate that results from Theorem~\ref{thm:single-layer}:
\begin{cor}[$hp$-{\sl a posteriori} error estimation for single layer operator]
\label{cor:single-layer-BEM}
Let $f \in H^1(\Gamma)$ and let $\varphi \in H^{-1/2}(\Gamma)$ solve 
(\ref{eq:single-layer-equation}).
Suppose that $\cT$ is a $\gamma$-shape regular mesh on $\Gamma$ and $\bfp$ is a $\gamma_p$-shape
regular polynomial degree distribution.
Let $V_N = \cS^{\bfp,0}({\mathcal T})$ in (\ref{eq:BEM-V}) and let  
$\varphi_N$ be the solution of (\ref{eq:BEM-V}). Then with the residual 
$R_N:= f - \operatorname*{V} \varphi_N$ the Galerkin error $\varphi - \varphi_N$ satisfies 
$$
\|\varphi - \varphi_N\|_{-1/2,2,\Gamma} \leq C \|R_N\|_{1/2,2,\Gamma} 
\leq C \|\left( h/\hat p\right)^{1/2} \nabla_\Gamma R_N\|_{0,2,\Gamma}. 
$$
\end{cor}
\begin{proof}
The first estimate expresses the boundedness of the operator $\operatorname*{V}^{-1}$. The second 
estimate follows from Theorem~\ref{thm:single-layer} and the Galerkin orthogonalities. 
\end{proof}
\subsubsection{Hypersingular operator}
\begin{lem}\label{lem:W:aux}
  Let $\Omega\subset\R{d}$, $d \in \{2,3\}$, be a bounded Lipschitz domain, and let $\Gamma=\partial\Om$.  
  Let $\cT$ be a $\gamma$-shape
  regular mesh on $\Gamma$, and let $\bfp$ be a $\gamma_p$-shape regular polynomial degree distribution
  on $\cT$ with $p_K \ge 1$ for all $K \in \cT$. 
  Then, there exists a linear operator $\II_\cT^\bfp: H^{1/2}(\Gamma)\rightarrow \cS^{\bfp,1}(\cT)$
  such that
  \begin{subequations}
  \begin{align}
    \vn{h^{-1/2}p^{1/2} (u-\II_\cT^\bfp u)}_{0,2,\Gamma}
    &\lesssim \vn{u}_{1/2,2,\Gamma}\label{W:aux:eq1},\\
    \vn{u-\II_\cT^\bfp u}_{1/2,2,\Gamma}
    &\lesssim
    \|{ h^{1/2}p^{-1/2}u}\|_{0,2,\Gamma}
    +
    \|{h^{1/2}p^{-1/2}\nabla_\Gamma u}\|_{0,2,\Gamma}.
    \label{W:aux:eq2}
  \end{align}
  \end{subequations}
\end{lem}
\begin{proof}
  Define the linear smoothing operator $\II_\cT$ by
  $\II_\cT u := \sum_{j=1}^n \left(\II_{\eps_j}(\beta_ju\circ\chi_j)\right)\circ\chi_j^{-1}$,
  where the operators $\II_{\eps_j}$ are as in the proof of Lemma~\ref{bem:smooth}.
Recall from Remark~\ref{rem:polyhedra} that the maps $\chi_j$ are piecewise affine. More precisely, 
 for every planar side $\Gamma_k$ of $\Gamma$, one has that 
 $\chi_j:V_j \cap \chi_j^{-1}(\Gamma_k \cap U_j) \rightarrow \Gamma_k\cap U_j$ is affine. This implies 
  that $\II_\cT u\in H^1(\Gamma)$ and $\II_\cT u \in C^{\infty}(\overline{\Gamma_k})$ on every
  planar side $\Gamma_k$ of $\Gamma$. Since $d \in \{2,3\}$, the approximation operator of 
  \cite[Lemma~{B.3}]{melenk-sauter10} is applicable to the function $\II_\cT u$, which is in $H^1(\Gamma)$
  and elementwise in $H^2$. \cite[Lemma~{B.3}]{melenk-sauter10} produces a piecewise polynomial 
  approximation $\II_\cT^\bfp u\in\cS^{\bfp,1}(\cT)$ in an element-by-element fashion from 
  $\II_\cT u$ such that 
  \begin{align*}
    \vn{\II_\cT u - \II_\cT^\bfp u}_{0,2,K}
    \lesssim h_K^{2}p_K^{-2}\sn{\II_\cT u}_{2,2,K}.
  \end{align*}
  For $s\in\R{}$ we obtain with Theorem~\ref{thm:hpsmooth}
  \begin{align*}
    \vn{h^{-s}p^s\,(\II_\cT u -\II_\cT^\bfp u)}_{0,2,\Gamma}^2
    &= \sum_{K\in\cT}\vn{h^{-s}p^s\,(\II_\cT u-\II_\cT^\bfp u)}_{0,2,K}^2
    \lesssim \sum_{K\in\cT} h_K^{2(2-s)}p_K^{-2(2-s)}\sn{\II_\cT u}_{2,2,K}^2 \\
    &\lesssim \sum_{j=1}^n\sum_{K\in\cT} h_K^{2(2-s)}p_K^{-2(2-s)}
    \|{\II_{\eps_j}(\beta_ju\circ\chi_j)\circ\chi_j^{-1}}\|_{2,2,K\cap U_j}^2\\
    &\lesssim \sum_{j=1}^n\sum_{K\in\cT} h_K^{2(2-s)}p_K^{-2(2-s)}
    \|{\II_{\eps_j}(\beta_ju\circ\chi_j)}\|_{2,2,V_j\cap\chi_j^{-1}(K)}^2\\
    &\lesssim \sum_{j=1}^n\sum_{K\in\cT} h_K^{2(1-s)}p_K^{-2(1-s)}
    \left( \|{\beta_ju\circ\chi_j}\|_{0,2,V_j\cap\chi_j^{-1}(\omega_K)}^2
    +
    \|{\nabla(\beta_ju\circ\chi_j)}\|_{0,2,V_j\cap\chi_j^{-1}(\omega_K)}^2\right)\\
    &\lesssim \sum_{j=1}^n
    \left( \|{ \varepsilon_j^{1-s}u\circ\chi_j}\|_{0,2,V_j}^2
    +
    \|{\varepsilon_j^{1-s}\nabla(u\circ\chi_j)}\|_{0,2,V_j}^2\right)\\
    &\lesssim 
    \|{ h^{1-s}p^{-(1-s)}u}\|_{0,2,\Gamma}^2
    +
    \|{h^{1-s}p^{-(1-s)}\nabla_\Gamma u}\|_{0,2,\Gamma}^2.
  \end{align*}
  Analogously, we obtain
  \begin{align*}
    \vn{h^{-s}p^s (u-\II_\cT u)}_{0,2,\Gamma}^2
    &= \vn{h^{-s}p^s \sum_{j=1}^n \left(\beta_j u -
      \II_{\eps_j}(\beta_ju\circ\chi_j)\circ\chi_j^{-1}\right)}_{0,2,\Gamma}^2
\lesssim \sum_{j=1}^n \vn{\eps_j^{-s} \left(\beta_j u\circ\chi_j -
      \II_{\eps_j}(\beta_ju\circ\chi_j)\right)}_{0,2,V_j}^2
\\ &
\lesssim \sum_{j=1}^n
    \left( \|\varepsilon_j^{1-s} u\circ\chi_j\|_{0,2,V_j}^2
    +
    \|\varepsilon_j^{1-s}\nabla(u\circ\chi_j)\|_{0,2,V_j}^2\right)
\\ &
\lesssim 
    \|{ h^{1-s}p^{-(1-s)}u}\|_{0,2,\Gamma}^2
    +
    \|{h^{1-s}p^{-(1-s)}\nabla_\Gamma u}\|_{0,2,\Gamma}^2.
  \end{align*}
  Here, the second estimate follows from Theorem~\ref{thm:hpsmooth} since we can bound the approximation error
  of $\II_{\eps_j}$ locally on every element $\chi_j^{-1}(K\cap U_j)$.
  For $s=1$, the above estimates read
  $\vn{h^{-1}p\,(\II_\cT u -\II_\cT^\bfp u)}_{0,2,\Gamma} \lesssim \vn{u}_{1,2,\Gamma}$
  and
  $\vn{h^{-1}p\,(u -\II_\cT u)}_{0,2,\Gamma} \lesssim \vn{u}_{1,2,\Gamma}$. 
  A similar reasoning as above shows
  $\vn{\II_\cT u -\II_\cT^\bfp u}_{0,2,\Gamma}\lesssim\vn{u}_{0,2,\Gamma}$
  and $\vn{u-\II_\cT u}_{0,2,\Gamma} \lesssim \vn{u}_{0,2,\Gamma}$.
  An interpolation argument and the triangle inequality
  show~\eqref{W:aux:eq1}.
  Choosing $s=1/2$ in the above estimates yields 
  \begin{align*}
    \vn{h^{-1/2}p^{1/2}(u -\II_\cT^\bfp u)}_{0,2,\Gamma}\lesssim
    \|{ h^{1/2}p^{-1/2}u}\|_{0,2,\Gamma}
    +
    \|{h^{1/2}p^{-1/2}\nabla_\Gamma u}\|_{0,2,\Gamma}.
  \end{align*}
  In addition, similar arguments as above show that
  \begin{align*}
    \vn{h^{1/2}p^{-1/2}\nabla_\Gamma \II_\cT^\bfp u}_{0,2,\Gamma}\lesssim
    \|{ h^{1/2}p^{-1/2}u}\|_{0,2,\Gamma}
    +
    \|{h^{1/2}p^{-1/2}\nabla_\Gamma u}\|_{0,2,\Gamma}.
  \end{align*}
  Together with Lemma~\ref{bem:smooth}, this proves~\eqref{W:aux:eq2}.

\end{proof}
The following can be seen as a generalization of the results of~\cite{cmps04}
to obtain a residual {\sl a posteriori} error estimator for $hp$-boundary elements
for hypersingular integral equations.
\begin{thm}
\label{thm:hypersingular-operator}
  Let $\Om\subset\R{d}$, $d \in \{2,3\}$, be a bounded Lipschitz domain and let $\Gamma=\partial\Omega$
  be its boundary. Let $\cT$ be a $\gamma$-shape regular mesh on $\Gamma$, 
  and let $\bfp$ be a $\gamma_p$-shape regular polynomial degree distribution on
  $\cT$ with $p_K\geq 1$ for all $K\in\cT$. Suppose that $u\in L^2(\Gamma)$ satisfies
  \begin{align}
    \int_\Gamma u\cdot \phi_{hp}\;d\sigma = 0 \
    \quad\text{ for all }\phi_{hp}\in\cS^{\bfp,1}(\cT).
    \label{cor:resW:galorth}
  \end{align}
  Then, for a constant $C_{\gamma,\gamma_p}$ that depends only on the shape-regularity constants $\gamma$, 
  $\gamma_p$ of the mesh and on $\Gamma$,
  \begin{align*}
    \vn{u}_{-1/2,2,\Gamma} \leq C_{\gamma,\gamma_p}\vn{h^{1/2}p^{-1/2}u}_{0,2,\Gamma}.
  \end{align*}
\end{thm}
\begin{proof}
  The orthogonality~\eqref{cor:resW:galorth}, Lemma~\ref{lem:W:aux},
  and Cauchy-Schwarz show for any $v\in H^{1/2}(\Gamma)$
  \begin{align*}
    \int_\Gamma u\cdot v\;d\sigma
    = \int_\Gamma u\cdot (v-\II_\cT^\bfp v)\;d\sigma
    &\lesssim \vn{h^{1/2}p^{-1/2}u}_{0,2,\Gamma}
    \vn{h^{-1/2}p^{1/2}(v-\II_\cT^\bfp v)}_{0,2,\Gamma}
\lesssim \vn{h^{1/2}p^{-1/2}u}_{0,2,\Gamma} \vn{v}_{1/2,2,\Gamma}.
  \end{align*}
  The definition of $H^{-1/2}(\Gamma)$ as dual space of $H^{1/2}(\Gamma)$
  shows the result.
\end{proof}
\begin{cor}[$hp$-{\sl a posteriori} error estimation for hypersingular operator]
\label{cor:hypersingular-BEM}
Let $\Gamma = \partial\Omega$ be connected. 
Let $g \in L^{2}(\Gamma)$ with $\langle g,1\rangle_\Gamma = 0$ and let $u \in H^{1/2}(\Gamma)$ 
be defined by (\ref{eq:hypersingular-equation}). 
Suppose that $\cT$ is a $\gamma$-shape regular mesh on $\Gamma$ and $\bfp$ is a $\gamma_p$-shape
regular polynomial degree distribution with $p_K\geq1$ for all $K\in\cT$.
Let $W_N = \cS^{\bfp,1}({\mathcal T})$ and $u_N \in W_N$ be given by (\ref{eq:BEM-D}). Then with the 
residual $R_N:= g - \operatorname*{D} u_N$ 
$$
\sn{u - u_N}_{1/2,2,\Gamma} \leq C \|R_N\|_{-1/2,2,\Gamma} \leq C \|(h/p)^{1/2}
R_N\|_{0,2,\Gamma}. 
$$
\end{cor}
\begin{proof}
The first estimate follows from the (semi-)ellipticity of the hypersingular operator 
$\operatorname*{D}: H^{1/2}(\Gamma) \rightarrow H^{-1/2}(\Gamma)$. The second estimate
follows from Theorem~\ref{thm:hypersingular-operator} and Galerkin orthogonality.  
\end{proof}
\section{Technical results for the proof of Theorem~\ref{thm:hpsmooth}}
\label{section:preliminairies}
This section provides tools and technical results that will be used in the
remainder.
The letter $\rho$ will always denote a \textit{mollifier}, i.e., 
a function $\rho \in C_0^{\infty}(\R{d})$
with \textit{(i)} $\rho(\bfx) = 0$ for $\sn{\bfx} \geq 1$,
and \textit{(ii)} $\int_{\R{d}} \rho(\bfx) d\bfx = 1$.
For $\delta  >0$, we write
$\rho_{\delta}(\bfx) := \rho \left( \bfx / \delta \right) \delta^{-d}$, so that
$\rho_{\delta}(\bfx) = 0$ for $\sn{\bfx} \geq \delta$ and $\int_{\R{d}} \rho_{\delta}(\bfx) d\bfx = 1$.
A mollifier $\rho$ is said to be of order $k_{\max}\in\N_0$ if
\begin{align}
\label{eq:order-condition-mollifier}
  \int_{\R{d}} \bfy^{\bfs} \rho(\bfy) d\bfy = 0 \qquad
  \text{ for every multi-index } \bfs \in \N_0^d \text{ with } 
  1 \leq \sn{\bfs} \leq k_{\max}. 
\end{align}
(Note that this condition is void if $k_{\max} = 0$.)
The condition (\ref{eq:order-condition-mollifier}) implies that a convolution with a mollifier 
of order $k_{\max}$ reproduces polynomials of degree up to $k_{\max}$.

Many results will be proved in a local fashion.
In order to transform these local results into global ones, we will use
Besicovitch's covering theorem, see \cite{eva1}. It is recalled here for
the reader's convenience: 
\begin{prop}[Besicovitch covering theorem]
\label{thm:besicovitch}
  There is a constant $N_d$ (depending only on the spatial dimension
  $d$) such that the following holds: 
  For any collection $\cF$ of non-empty, closed balls in $\R{d}$ with
$\displaystyle     \sup \left\{ {\rm diam }\,B \mid B \in \cF \right\} < \infty,
$
  and the set $A$ of the mid-points of the balls $B \in\cF$, there are
  subsets $\cG_1, \dots, G_{N_d} \subset \cF$ such that for each 
  $i = 1, \dots, N_d$, the family $\cG_i$ is a countable set of
  pairwise disjoint balls and
  \begin{align*} 
    A \subset \bigcup_{i=1}^{N_d} \bigcup_{B \in \cG_i} B.
\tag*{\qed}
  \end{align*}
\end{prop}
An open set $S\subset\R{d}$ is said to be \textit{star-shaped with respect to
a ball B}, if the closed convex hull of $\left\{ \bfx \right\} \cup B$ is a
subset of $S$ for every $\bfx \in S$. The \textit{chunkiness parameter}
of $S$ is defined as $\eta(S):= \diam(S)/\rho_{\max}$, where
\begin{align*}
  \rho_{\max} := \sup\left\{ \rho \mid S
    \text{ is star-shaped with respect to a ball of radius } \rho \right\},
\end{align*}
cf.~\cite[Def.~4.2.16]{brenner2008}.
We will frequently employ Sobolev embedding theorems, and it will be necessary
to control the constants in terms of the chunkiness of the underlying domain.
Results of this type are well-known for integer order spaces,
cf.~\cite[Ch.~4]{adams-fournier03}. In Appendix~\ref{appendix}, we give
a self-contained proof that also for fractional order spaces the constants of the Sobolev embedding
theorem for star-shaped domains can be controlled in terms of 
the chunkiness parameter and the diameter of the domain. This results in the following 
embedding theorem. 
\begin{thm}[embedding theorem]
\label{thm:embedding}
  Let $\eta>0$, and let $\Om\subset\R{d}$ be a bounded domain with chunkiness parameter 
$\eta(\Om)\leq \eta$.
  Let $s,r,p,q\in\R{}$ with $0\leq s \leq r < \infty$ and $1\leq p \leq q < \infty$
  and set $\mu := d(p^{-1}-q^{-1})$.
  Assume that $(r=s+\mu\text{ and } p>1)$
  or $(r>s+\mu)$.
  Then there exists a constant $C_{s,q,r,p,\eta,d}$ (depending only on the quantities indicated) such that

  \begin{align*}
    \sn{u}_{s,q,\Om} \leq C_{s,q,r,p,\eta,d}\; \diam(\Om)^{-\mu}
    \Bigl\{ \diam(\Om)^{r - s}\sn{u}_{r,p,\Om}
    + \sum_{\substack{r' \in \N_0:\\ \lceil s \rceil \leq r' \leq r}} 
     \diam(\Om)^{r' - s}\sn{u}_{r',p,\Om}
    \Bigr\}.
  \end{align*}
  Furthermore, if $s$, $r\in\N_0$, $s\leq r$, set $\mu^\prime := d/p$.
  Assume that $(r=s+\mu^\prime\text{ and } p=1)$ or
  $(r>s+\mu^\prime \text{ and } p > 1)$.
  Then there exists a constant $C_{s,r,p,\eta,d}$ (depending only on the quantities indicated) such that
  \begin{align*}
    \sn{u}_{s,\infty,\Om} \leq C_{s,r,p,\eta,d}\; \diam(\Om)^{-\mu^\prime}
    \sum_{r' = s}^r  \diam(\Om)^{r' - s}\sn{u}_{r',p,\Om}. 
  \end{align*}
\end{thm}
\begin{proof}
  As $\Om$ is star-shaped with respect to a ball of radius $\diam(\Om)/(2\eta)$,
  the scaled domain $\widehat\Om:= \diam(\Om)^{-1}\Om$ is star-shaped with respect
  to a ball of radius $1/(2\eta)$. For the first result, we employ 
  Theorem~\ref{thm:sobolev-embedding} and scaling arguments to obtain the stated 
  right-hand side with the sum extending over $r' \in \{0,1,\ldots,\lfloor r\rfloor\}$
  instead of $\{\lceil s\rceil,\ldots,\lfloor r \rfloor\}$. 
  The restriction of the summation to $r' \in \{\lceil s\rceil,\ldots,\lfloor r \rfloor\}$
  follows from the observation that the left-hand side vanishes for polynomials of degree 
  $\lceil s \rceil-1$. Hence, one can use of polynomial approximation result of \cite{dupont-scott80} 
  in the usual way by replacing $u$ with $u - \pi$, where $\pi$ is the polynomial approximation
  given in \cite{dupont-scott80}. 

  The $L^\infty$-estimate follows from~\cite[Lem.~4.3.4]{brenner2008} and
  scaling arguments.
\end{proof}
Another tool we will use is the classical Bramble-Hilbert Lemma.
\begin{lem}[Bramble-Hilbert, \protect{\cite[Lemma~{4.3.8}]{brenner2008}}]
\label{lem:bramblehilbert}
  Let $\Om\subset\R{d}$ be a bounded domain with chunkiness parameter $\eta(\Om)\leq\eta<\infty$.
  Then, for all $u\in W^{m,p}(S)$ with $p\geq 1$, there is a polynomial
  $\pi\in\cP^{m-1}(S)$ such that
  \begin{align*}
    \sn{u-\pi}_{k,p,S}\leq C_{m,d,\eta}\; \diam(S)^{m-k} \sn{u}_{m,p,S},
    \qquad\text{ for all } k = 0, \dots, m.
  \end{align*}
  The constant $C_{m,d,\eta}$ depends only on $m$, $d$, and $\eta$.\qed
\end{lem}
The next lemma is a version of the Fa\`a di Bruno formula,
which is a formula for computing higher derivatives
of composite functions. For $s,\ell\in\N_0$
we denote by $\MM_{s,\ell}$ a set of multi-indices 
$\MM_{s,\ell} = \left\{ \bft_{i} \right\}_{i=1}^\ell \subset \N_0^d $ such that
$\sn{\bft_i}\geq 1$ and $\sum_{i=1}^\ell \left( \sn{\bft_i}-1 \right) = s$.
\begin{lem}[Fa\`a di Bruno]
\label{lem:faa-di-bruno-1}
For every $\bfs \in \N_0^d$ with $\sn{\bfs}\geq 1$ and every $\bfr \in \N_0^d$ with $\sn{\bfr} \leq \sn{\bfs}$
and every set $\MM_{\sn{\bfs}-\sn{\bfr},\ell}$ with $1\leq\ell\leq\sn{\bfr}$
there is a polynomial $P_{\bfs,\bfr,\MM_{\sn{\bfs}-\sn{\bfr},\ell}}:\R{d} \rightarrow \R{}$
of degree $|\bfr|$ such that the following is true:
For any $\eps \in C^\infty(\R{d})$, $\bfz \in \R{d}$, and $u \in C^\infty(\R{d})$
the derivative $D^\bfs_\bfx u(\bfx')$, $\bfx':=\bfx+\bfz\eps(\bfx)$,
can be written in the form 
\begin{equation} 
\label{lem:faa-di-bruno-1:eq2}
D^\bfs_\bfx u(\bfx') = 
(D^\bfs u)(\bfx^\prime) + 
\sum_{|\bfr| \leq |\bfs|} (D^\bfr u)(\bfx^\prime) 
\sum_{\substack{\MM_{\sn{\bfs}-\sn{\bfr},\ell}\\1\leq\ell\leq\sn{\bfr}}}
P_{\bfs,\bfr,\MM_{\sn{\bfs}-\sn{\bfr},\ell}}(\bfz) 
\prod_{\bft\in\MM_{\sn{\bfs}-\sn{\bfr},\ell}} D^{\bft} \eps(\bfx). 
\end{equation}

We employ the convention that empty sums take the value 
zero and empty products the value $1$.
Furthermore, if $P_{\bfs,\bfr,\MM_{\sn{\bfs}-\sn{\bfr},\ell}}$ is constant, 
then $P_{\bfs,\bfr,\MM_{\sn{\bfs}-\sn{\bfr},\ell}} \equiv 0$. 
\end{lem}
\begin{proof}
  Introduce the function $\widetilde u$ as 
\begin{equation}
\label{eq:utilde}
\widetilde u(\bfx):= u(\bfx+\bfz \eps(\bfx)). 
\end{equation}

  We use the shorthand $\partial_i = \partial \cdot / \partial x_i$ and compute for $|\bfs| = 1$
  \begin{eqnarray*}
  \partial_i \widetilde u(\bfx) = (\partial_i u)(\bfx^\prime)(1 + \bfz_i \partial_i \eps(\bfx)) + 
  \sum_{j \ne i} (\partial_j u)(\bfx^\prime) \bfz_j \partial_i \eps (\bfx)
  =(\partial_i u)(\bfx^\prime) + \sum_{j=1}^d (\partial_j u)(\bfx^\prime) \bfz_j \partial_i \eps(\bfx), 
  \end{eqnarray*}
  which we recognize to be of the form (\ref{lem:faa-di-bruno-1:eq2}). We now proceed by induction. 
  To that end, we assume that formula (\ref{lem:faa-di-bruno-1:eq2}) is true for all
  multiindices $\bfs^\prime \in \N_0^d$ with $|\bfs^\prime| \leq n$. Then for 
  $\bfs=(\bfs_1^\prime,\ldots,\bfs_{i-1}^\prime,\bfs_i^\prime+1,\bfs_{i+1}^\prime,\ldots,\bfs_d^\prime)$
  we compute with the induction hypothesis: 
  \begin{eqnarray*}
  D^{\bfs} \widetilde u(\bfx^\prime) &=& 
  \partial_i D^{\bfs^\prime} \widetilde u(\bfx^\prime) = 
  (D^{\bfs} u)(\bfx^\prime) + 
  \sum_{j=1}^d (D^{\bfs^\prime} \partial_j u)(\bfx^\prime) \bfz_j \partial_i \eps(\bfx), \\
  && \mbox{} + 
  \partial_i \left( 
  \sum_{|\bfr| \leq |\bfs^\prime|} (D^\bfr u)(\bfx^\prime) 
  \sum_{\substack{\MM_{\sn{\bfs'}-\sn{\bfr},\ell}\\1\leq\ell\leq\sn{\bfr}}}
  P_{\bfs,\bfr,\MM_{\sn{\bfs'}-\sn{\bfr},\ell}}(\bfz) 
  \prod_{\bft\in\MM_{\sn{\bfs'}-\sn{\bfr},\ell}} D^{\bft} \eps(\bfx) 
  \right) 
  =: T_1 + T_2 + T_3. 
  \end{eqnarray*}
  Hence, $T_1 + T_2$ consists of terms of the desired form. 
  For the term $T_3$, we compute 
  \begin{eqnarray*}
    \partial_i (D^{\bfr} u)(\bfx^\prime) &=& (D^{\bfr} \partial_i u)(\bfx^\prime) + 
    \sum_{j^\prime=1}^d (D^\bfr \partial_{j^\prime} u)(\bfx^\prime) \bfz_{j^\prime} \partial_{i} \eps(\bfx) \\
    \partial_i 
    \prod_{\bft\in\MM_{\sn{\bfs'}-\sn{\bfr},\ell}} D^{\bft} \eps(\bfx) 
    &=& 
    \sum_{\bft\in\MM_{\sn{\bfs'}-\sn{\bfr},\ell}} (D^{\bft} \partial_i \eps(\bfx)) \prod_{\substack{\bft^\prime\in\MM_{\sn{\bfs'}-\sn{\bfr},\ell} \\ \bft^\prime \ne \bft}} D^{\bft^\prime} \eps(\bfx).
  \end{eqnarray*}
  Hence, $T_3$ has the form 
  \begin{eqnarray*}
  T_3 &=& 
  \sum_{|\bfr| \leq |\bfs^\prime|} D^\bfr \partial_i u(\bfx^\prime) 
  \sum_{\substack{\MM_{\sn{\bfs'}-\sn{\bfr},\ell}\\1\leq\ell\leq\sn{\bfr}}}
  P_{\bfs',\bfr,\MM_{\sn{\bfs'}-\sn{\bfr},\ell}}(\bfz) 
  \prod_{\bft\in\MM_{\sn{\bfs'}-\sn{\bfr},\ell}} D^{\bft} \eps(\bfx) \\
  && \mbox+ 
  \sum_{|\bfr| \leq |\bfs^\prime|} \sum_{j^\prime=1}^d D^\bfr \partial_{j^\prime} u(\bfx^\prime) 
  \sum_{\substack{\MM_{\sn{\bfs'}-\sn{\bfr},\ell}\\1\leq\ell\leq\sn{\bfr}}}
  \bfz_{j^\prime}\partial_i\eps(\bfx)P_{\bfs',\bfr,\MM_{\sn{\bfs'}-\sn{\bfr},\ell}}(\bfz) 
  \prod_{\bft\in\MM_{\sn{\bfs'}-\sn{\bfr},\ell}} D^{\bft} \eps(\bfx) \\
  && \mbox{}+ 
  \sum_{|\bfr| \leq |\bfs^\prime|} 
  D^\bfr u(\bfx^\prime) 
  \sum_{\substack{\MM_{\sn{\bfs'}-\sn{\bfr},\ell}\\1\leq\ell\leq\sn{\bfr}}}
  P_{\bfs',\bfr,\MM_{\sn{\bfs'}-\sn{\bfr},\ell}}(\bfz) 
    \sum_{\bft\in\MM_{\sn{\bfs'}-\sn{\bfr},\ell}} (D^{\bft} \partial_i \eps(\bfx)) \prod_{\substack{\bft^\prime\in\MM_{\sn{\bfs'}-\sn{\bfr},\ell} \\ \bft^\prime \ne \bft}} D^{\bft^\prime} \eps(\bfx).
  \end{eqnarray*}
  Since $|\bfs| = |\bfs^\prime| + 1$, each of the three sums has the stipulated form. 
  This concludes the induction argument. 
\end{proof}
For the function $\widetilde u$ given by (\ref{eq:utilde}), the 
next lemma quantifies the difference $u - \widetilde u$ if $\eps$ is 
a $\Lambda$-admissible length scale function.
\begin{lem}
  \label{lem:faa-1-estimate}
Let $\Omega \subset \R{d}$ be a domain,  and let $\eps \in C^\infty(\Omega)$ be a 
  $\Lambda$-admissible length scale function. Let $u \in C^\infty(\R{d})$. 
  Then, for a multiindex $\bfs \in \N_0^d$, 
  the derivative $D^\bfs_\bfx u(\bfx')$, $\bfx':=\bfx+\bfz\eps(\bfx)$,
  can be written in the form 
  \begin{align}
    D^\bfs_\bfx u(\bfx') = (D^\bfs u)(\bfx^\prime) + \sum_{|\bfr| \leq |\bfs|} 
    \left(D^\bfr u\right)(\bfx^\prime) E_{\bfs,\bfr}(\bfz,\bfx), 
  \end{align}
  where the smooth functions $E_{\bfs,\bfr}$ are polynomials of degree $\sn{\bfr}$ in the first component and satisfy 
  \begin{align}
    \sup_{\sn{\bfz}\leq R}|D^\bft_\bfz E_{\bfs,\bfr} (\bfz,\bfx)| \leq C_{\Lambda,\bfs,R} |\eps(\bfx)|^{|\bfr| - |\bfs|} 
    \qquad \forall \bfx \in \Omega. 
  \end{align}
  The constants $C_{\Lambda,\bfs,R}$ depend only on $\bfs$, $R$, and $(\Lambda_{\bfs'})_{\sn{\bfs'}\leq\sn{\bfs}}$.
\end{lem}
\begin{proof}
  Apply Lemma~\ref{lem:faa-di-bruno-1} and define
  \begin{align*}
    E_{\bfr}(\bfz,\bfx) =
    \sum_{\substack{\MM_{\sn{\bfs}-\sn{\bfr},\ell}\\1\leq\ell\leq\sn{\bfr}}}
    P_{\bfs,\bfr,\MM_{\sn{\bfs}-\sn{\bfr},\ell}}(\bfz) 
    \prod_{\bft\in\MM_{\sn{\bfs}-\sn{\bfr},\ell}} D^{\bft} \eps(\bfx). 
  \end{align*}
  Clearly, $E_{\bfr}$ is smooth and is a polynomial of degree $|\bfr|$ in the first component. Let $\bfx \in \Omega$.
  According to the properties of a $\Lambda$-admissible length scale function and
  the definition of the set $\MM_{\sn{\bfs}-\sn{\bfr},\ell}$ for
  $1\leq\ell\leq\sn{\bfr}$, there holds
  \begin{align*}
    \begin{split}
      \prod_{\bft\in\MM_{\sn{\bfs}-\sn{\bfr},\ell}} \sn{D^{\bft} \eps(\bfx)}
      &\leq \prod_{\bft\in\MM_{\sn{\bfs}-\sn{\bfr},\ell}} \Lambda_\bft \sn{\eps(\bfx)}^{1-\sn{\bft}}
      \leq
      \max_{\sn{\bfs'}\leq\sn{\bfs}}\left(\Lambda_{\bfs'}\right)^{\ell}
      \prod_{\bft\in\MM_{\sn{\bfs}-\sn{\bfr},\ell}} \sn{\eps(\bfx)}^{1-\sn{\bft}}
      =
      \max_{\sn{\bfs'}\leq\sn{\bfs}}\left(\Lambda_{\bfs'}\right)^{\ell}
      \sn{\eps(\bfx)}^{\sn{\bfr}-\sn{\bfs}}.
    \end{split}
  \end{align*}
  We can conclude the proof by setting 
  \begin{align*}
    C_{\Lambda,\bfs,R} :=
\sup_{\sn{\bfz}\leq R}
    \sum_{\substack{\MM_{\sn{\bfs}-\sn{\bfr},\ell}\\1\leq\ell\leq\sn{\bfr}}}
    \max_{\sn{\bfs'}\leq\sn{\bfs}}\left(\Lambda_{\bfs'}\right)^{\ell} \,
    \sn{D^\bft P_{\bfs,\bfr,\MM_{\sn{\bfs}-\sn{\bfr},\ell}}(\bfz)}.
\tag*{\qedhere}
\end{align*}
\end{proof}
\section{Higher-order volume regularization}\label{section:regularization}
\subsection{Regularization on a reference ball}
Throughout this section, $\eps$ denotes a $\Lambda$-admissible length scale
function, and $\Lambda = \bigl( \LL, \left( \Lambda_\bfr
\right)_{\bfr\in\N_0^d}\bigr)$.
Our goal is to construct operators that, given $\bfz\in\R{d}$, 
employ regularization on a length scale $\eps(\bfz)$, which therefore
determines the quality of the approximation at $\bfz$.
We will analyze such operators on balls $B_r(\bfx)$ that have a radius
which is comparable to the value of $\eps(\bfx)$. Hence, it will be
convenient to use reference configurations and scaling arguments.
For fixed $\bfx$, we define the scaling map 
\begin{align*}
  T_{\bfx}: \bfz \mapsto \bfx + \eps(\bfx)\bfz.
\end{align*}
In classical finite element approximation theory, the \textit{pull-back}
$u \circ T_{\bfx}$ of a function $u$
is approximated on a reference configuration. This approximation is also
analyzed on the reference configuration, and scaling arguments provide
the current quality of the approximation to $u$
(given by powers of the underlying length scale).
As stated above, the construction that is carried out in this work
defines the approximation of the pull-back also in terms
of the local length scale.
In order to obtain a fixed length scale on our reference configuration,
i.e., to make the approximation properties on the reference configuration
independent of a specific length scale, it will be convenient to define
the function $\eps_{\bfx}$ by
\begin{align*}
  \bfz \mapsto \eps_{\bfx}(\bfz):= \frac{\eps (T_{\bfx}(\bfz))}{\eps(\bfx)}.
\end{align*}
The next lemma shows that $\eps_\bfx$ does, in essence, only depend on
$\Lambda$, but not on $\bfx$.
We construct parameters $\alpha, \beta, \delta > 0$, where $\delta$ will
be used to define the regularization operator and $\alpha,\beta$ will
be used to define balls on which the regularization error will be analyzed. In
subsequent sections, these parameters need to be adjusted also according
to properties of the domain of interest $\Omega$, more precisely, its Lipschitz character
and in particular the Lipschitz constant $L_{\partial\Omega}$ of $\partial\Omega$. 
Hence, the parameters $\alpha$, $\beta$, $\delta$
will be chosen in dependence on $L_{\partial\Om}$ and an additional parameter $L$.
We recall that $B_\alpha$ is a shorthand for $B_\alpha(0)$. 

\begin{lem}\label{lem:eps}
  Let $\Omega\subset\R{d}$ be an (arbitrary) domain and 
  $\eps\in C^\infty(\Om)$ be a $\Lambda$-admissible length scale
  function.
Then: 
\begin{enumerate}[(i)]
\item  
\label{item:lem:eps-i}
For fixed $\alpha \in \left( 0,\min\left( 1,\LL^{-1}/2 \right)\right)$  let 
  $\bfx \in \Omega$ be such that $T_{\bfx}(B_\alpha) \subset \Omega$. 
  Then 
  \begin{align} 
    2^{-1}  \leq \eps_\bfx(\bfz) \leq  2 \qquad
    \text{ for all } \bfz \in B_\alpha.
    \label{conv1:lem1:eq1}
  \end{align}
\item 
\label{item:lem:eps-ii}
One may choose 
  $0< \alpha$, $\delta$, $\beta < \min\left( 1,\LL^{-1}/2 \right)$
  with 
\begin{equation}
\label{eq:lem:eps-10}
2\delta + \alpha < \beta < \min\left(1, \LL^{-1}/2\right). 
\end{equation}
The parameters $\alpha$, $\beta$, and $\delta$ depend only on $\LL$. 
\item 
\label{item:lem:eps-iii}
Given (arbitrary) parameters $L_{\partial\Om}$, $L>0$ one may choose $\alpha$, $\beta$, $\delta>0$ such that 
  \begin{align}\label{conv1:lem1:eq2}
    \max\left( (L_{\partial\Om}+1)(\delta+\alpha)+\delta, \alpha +
    \delta(L+1)(1+\LL\alpha) \right) < \beta < \min\left( 1,\LL^{-1}/2 \right).
  \end{align}
The parameters depend only on $\LL$, $L$, and $L_{\partial\Omega}$. 
\item 
\label{item:lem:eps-iv}
Let $\overline{\Omega^\prime} \subset \Omega$ be compact.
  Then the parameters $\alpha$, $\beta$, $\delta$ can be chosen such that 
in addition to (\ref{conv1:lem1:eq2}) the following property holds: 
\begin{equation}\label{conv1:lem1:eq3}
\bfx \in \Omega \qquad \Longrightarrow \left( \mbox{ either }\quad
\overline{B_{2 \beta\eps(\bfx)}(\bfx)} \subset \Omega \quad \mbox{ or }\quad 
\overline{B_{2\beta\eps(\bfx)}(\bfx)} \cap \overline{\Omega^\prime}  = \emptyset \right).
\end{equation}
\end{enumerate}
\end{lem}
\begin{proof}
  Fix $0 < \alpha < \min(1,\LL^{-1}/2)$ and $\bfx \in \Omega$ such that $T_{\bfx}(B_\alpha) \subset \Omega$. 
  For $\bfz \in B_\alpha$ we conclude with the reverse triangle inequality
  \begin{align*} 
    \pm \left( \eps(T_{\bfx} (\bfz)) - \eps(\bfx)\right)
    \leq \sn{\eps(T_{\bfx} (\bfz)) - \eps(\bfx)} \leq \LL
    \sn{T_{\bfx}(\bfz)-\bfx}
    \leq \LL \alpha \eps(\bfx).
  \end{align*}
  This yields
  \begin{align*} 
    \frac{\eps(T_{\bfx} \bfz)}{\eps(\bfx)}
    \leq \LL\alpha + 1 \qquad \text{ and }
    \qquad \frac{\eps(T_{\bfx} \bfz)}{\eps(\bfx)} \geq 1- \LL\alpha,
  \end{align*}
  from which~\eqref{conv1:lem1:eq1} follows.
  The additional features~\eqref{conv1:lem1:eq2} can be achieved 
  by adjusting $\alpha$ and $\delta$. 
  Finally, (\ref{item:lem:eps-iv}) follows from compactness of
  $\overline{\Omega^\prime}$. 
\end{proof}
Next, we prove an approximation result on the reference configuration.
\begin{lem}
  \label{lem:approx:refel}
  Let $\Omega \subset \R{d}$ be an (arbitrary) domain. 
  Let $\rho$ be a mollifier of order $k_{\max}$, 
  and let $\eps \in C^\infty(\Omega)$ be a $\Lambda$-admissible length scale
  function, $\Lambda= \bigl( \LL,\left( \Lambda_{\bfr} \right)_{\bfr\in\N_0^d}
\bigr)$.
  Choose $\alpha$, $\beta$, $\delta$ such that (\ref{eq:lem:eps-10}) holds.
  For $\bfx \in \Omega$ such that $T_\bfx(B_\alpha) \subset \Omega$ and a function $v \in L^1_{loc}(B_\beta)$ define
  \begin{align}
    \bfz \mapsto \cE_\bfx v(\bfz) := \int_{\bfy\in B_\beta } v(\bfy) \rho_{\delta \eps_\bfx(\bfz)}(\bfz-\bfy).
    \label{lem:approx:refel:eq1}
  \end{align}
\begin{enumerate}[(i)]
\item 
  \label{item:lem:approx:refel-i}
  Let $(s,p) \in \N_0 \times [1,\infty]$ satisfy  $s\leq k_{\max}+1$. Assume 
  $(s\leq r\in\R{}\text{ and } q\in[1,\infty))$ or
  $(s\leq r\in\N_0\text{ and }q\in[1,\infty])$. Then it holds
  \begin{subequations}\label{lem:approx:refel:eq2}
  \begin{align}\label{lem:approx:refel:eq2:c}
    \sn{\cE_\bfx v}_{r,q,\Ba} &\leq C_{r,q,s,p,\Lambda} \sn{v}_{s,p,\Bb},
  \end{align}
  where $C_{r,q,s,p,\Lambda}$ depends only on $r$, $q$, $s$, $p$, $\LL$, and
  $(\Lambda_{\bfr'})_{\sn{\bfr'}\leq \lceil r \rceil}$
as well as $\rho$, $k_{\max}$, and $\alpha$, $\beta$, $\delta$. 
\item
  \label{item:lem:approx:refel-ii}
  Suppose $0\leq s\in\R{}$, $r\in\N_0$ with $s\leq r\leq k_{\max}+1$,
  and $1\leq p\leq q<\infty$. Define $\mu:=d(p^{-1}-q^{-1})$. Assume that
  $(r=s+\mu\text{ and }p>1)$ or $(r>s+\mu)$. Then it holds that
  \begin{align}\label{lem:approx:refel:eq2:b}
    \sn{v - \cE_\bfx v}_{s,q,\Ba} &\leq C_{s,q,r,p,\Lambda} \sn{v}_{r,p,\Bb},
  \end{align}
  where $C_{s,q,r,p,\Lambda}$ depends only on $s$, $q$, $r$, $p$,
  $\LL$, and
  $(\Lambda_{\bfs'})_{\sn{\bfs'}\leq \lceil s \rceil}$
as well as $\rho$, $k_{\max}$, and $\alpha$, $\beta$, $\delta$. 
\item
  \label{item:lem:approx:refel-iii}
  Suppose $s$, $r\in\N_0$ with $s\leq r\leq k_{\max}+1$, and $1\leq p < \infty$.
  Define $\mu:=d/p$. Assume that $(r=s+\mu\text{ and }p=1)$ or $(r>s+\mu \text{ and } p > 1)$.
  Then it holds that
  \begin{align}\label{lem:approx:refel:eq2:d}
    \sn{v - \cE_\bfx v}_{s,\infty,\Ba} &\leq C_{s,r,p,\Lambda} \sn{v}_{r,p,\Bb},
  \end{align}
  \end{subequations}
  where $C_{s,r,p,\Lambda}$ depends only on $s$, $r$, $p$, $\LL$,
  and $(\Lambda_{\bfs'})_{\sn{\bfs'}\leq s}$
as well as $\rho$, $k_{\max}$, and $\alpha$, $\beta$, $\delta$. 
\end{enumerate}
\end{lem}
\begin{proof}
  For a multi-index $\bfr$ we have
  \begin{align}
    \sn{D^\bfr_\bfz \xeps(\bfz)} = \eps(\bfx)^{-1} \sn{D^\bfr_\bfz \eps(\bfx+\eps(\bfx)\bfz)} =
    \eps(\bfx)^{-1} \sn{\left( D^\bfr \eps \right)(\bfx+\eps(\bfx)\bfz)} \sn{\eps(\bfx)}^{\sn{\bfr}}
    \leq \Lambda_\bfr \sn{\eps_{\bfx}(\bfz)}^{1 - \sn{\bfr}},
    \label{lem:approx:refel:eq6}
  \end{align}
  from which we conclude that $\xeps \in C^\infty(B_\alpha)$ is also a $\Lambda$-admissible
  length scale function. For $v \in L^1(B_\beta)$, we may assume in view of density 
  $v \in C^\infty_0(B_\beta)$ 
  so that we can interchange differentiation and integration.
  Setting $\bfz' := \bfz - \delta\eps_\bfx(\bfz)\bfy$,
  the Fa\`a di Bruno formula from Lemma~\ref{lem:faa-1-estimate} shows 
  \begin{align}
    \begin{split}
      D^{\bfr}_{\bfz} v\left(\bfz'\right)
      &= \left( D^{\bfr}v \right)\left(\bfz'\right) +
      \sum_{\sn{\bft}\leq \sn{\bfr}}\left( D^{\bft}v \right)\left(\bfz'\right)
      E_{\bfr,\bft}\left( -\delta \bfy,\bfz \right)\\
      &=\left( -\delta\eps_\bfx(\bfz) \right)^{-\sn{\bfr}} D^{\bfr}_\bfy v\left(\bfz'\right) +
      \sum_{\sn{\bft}\leq \sn{\bfr}}
      \left( -\delta\eps_\bfx(\bfz) \right)^{-\sn{\bft}}
      D^{\bft}_\bfy v\left(\bfz'\right)
      E_{\bfr,\bft}\left( -\delta \bfy,\bfz \right).
    \end{split}
  \end{align}
  We obtain with integration by parts, the product rule, and the support properties
  of $\rho$
  \begin{align*}
    D^\bfr_{\bfz} \cE_\bfx v(\bfz)
    &= 
    (-1)^{\sn{\bfr}}
    \int_{\bfy \in B_1(0)}
      v\left(\bfz'\right)
      D^{\bfr}\rho(\bfy)
      \left( -\delta\eps_\bfx(\bfz) \right)^{-\sn{\bfr}}\\
    &\quad+
    \sum_{\sn{\bft}\leq\sn{\bfr}}(-1)^{\sn{\bft}}
    \sum_{\bfs\leq\bft}\binom{\bfs}{\bft}
    \int_{\bfy \in B_1(0)}
      v(\bfz') D^\bfs_\bfy E_{\bfr,\bft}(-\delta\bfy,\bfz)
      D^{\bft-\bfs}\rho(\bfy)
      \left( -\delta\eps_\bfx(\bfz) \right)^{-\sn{\bft}}.
  \end{align*}
  Taking into account Lemmas~\ref{lem:faa-1-estimate} and~\ref{lem:eps}, we obtain
  for $r\in\N_0$ the estimate
  \begin{align}\label{lem:approx:refel:eq:30}
    \sn{\cE_\bfx v}_{r,\infty,\Ba} \leq C_{\Lambda,r} \vn{v}_{0,1,\Bb}.
  \end{align}

  H\"older's inequality then shows (\ref{lem:approx:refel:eq2:c}) for the case $s = 0$ and integer $r$. 
  For $r\notin\N_0$ and $q \in [1,\infty)$, we use the Embedding Theorem~\ref{thm:embedding}
  (although for the present case of the ball $B_\alpha$, a simpler argument could be used)
  in the following way: observing $0 < \lceil r \rceil - r < 1$, we select $1 < p_\star < q$ such that 
$$
\left( \frac{\lceil r\rceil - r}{d} > \frac{1}{p_\star} - \frac{1}{q}\right). 
$$
  Then the Embedding Theorem~\ref{thm:embedding} and estimate~\eqref{lem:approx:refel:eq:30} show
  \begin{align}\label{lem:approx:refel:eq:31}
    \sn{\cE_\bfx v}_{r,q,\Ba} \lesssim \vn{\cE_\bfx v}_{\lceil r \rceil,p_\star,\Ba}
    \lesssim C_{\Lambda,\lceil r \rceil}\vn{v}_{0,1,\Bb}.
  \end{align}
  Again, H\"older's inequality shows \eqref{lem:approx:refel:eq2:c} for the case $s = 0$. 
  Next, let $1\leq s \leq k_{\max}+1$ and $s \leq r$. Then for any polynomial $\pi \in \cP_{s-1}(\Bb)$
  we have $\sn{\pi}_{r,q,\Bb} = 0$ and $\cE_\bfx \pi = \pi$.
  Estimates~\eqref{lem:approx:refel:eq:30} or~\eqref{lem:approx:refel:eq:31}
  and the Bramble-Hilbert Lemma~\ref{lem:bramblehilbert} then show
  \begin{align*}
    \sn{\cE_\bfx v}_{r,q,\Ba} \leq
    C_{\Lambda,\lceil r \rceil} \inf_{\pi \in \cP_{s-1}(\Bb)} \vn{v-\pi}_{0,1,\Bb}
    \leq C_{\Lambda,r,s} \sn{v}_{s,1,\Bb},
  \end{align*}
  from which~\eqref{lem:approx:refel:eq2:c} follows for $s \ge 1$ again
  by application of H\"older's inequality.

  {\sl Proof of (\ref{item:lem:approx:refel-ii}), (\ref{item:lem:approx:refel-iii}):}
  Since $\cE_\bfx\pi=\pi$ for any polynomial $\pi \in \cP_{r-1}(\Bb)$, the triangle
  inequality yields
  \begin{align*}
    \sn{v - \cE_\bfx v}_{s,q,\Ba} &\leq \sn{v-\pi}_{s,q,\Ba} + \sn{\cE_\bfx(v-\pi)}_{s,q,\Ba}
    \lesssim \vn{v-\pi}_{r,p,\Ba} + \sn{v-\pi}_{\lfloor s \rfloor,p,\Bb}
    \leq \vn{v-\pi}_{r,p,\Bb},
  \end{align*}
  where we used the embedding results of Theorem~\ref{thm:embedding}
  as well as~\eqref{lem:approx:refel:eq2:c} in the second step.
  The estimates~\eqref{lem:approx:refel:eq2:b} and~\eqref{lem:approx:refel:eq2:d}
  now again follow by application of the Bramble-Hilbert Lemma~\ref{lem:bramblehilbert}.
\end{proof}
\subsection{Regularization in the interior of $\Omega$}
\label{sec:regularization-in-interior}
With the results of the last section, we can define a regularization
operator $\cE$ that will determine $\II_\eps$ of Theorem~\ref{thm:hpsmooth} in the interior of $\Omega$; 
near $\partial \Omega$, we will need a modification introduced in Section~\ref{sec:regularization-near-boundary} below.
The properties of $\II_\eps$ in the interior will be analyzed using scaling arguments.
We show that the regularization
operator $\cE$ satisfies inverse estimates (Lemma~\ref{lem:approx}, (\ref{item:lem:approx-i}))
in addition to having approximation properties (Lemma~\ref{lem:approx},
(\ref{item:lem:approx-ii}) and (\ref{item:lem:approx-iii})). 
\begin{lem}
  \label{lem:approx}
  Let $\rho$ be a mollifier of order $k_{\max}$, $p \in [1,\infty)$.
  Let $\Omega\subset \R{d}$ be an (arbitrary) domain and let $\eps \in C^\infty(\Omega)$
  be a $\Lambda$-admissible length scale function,
  $\Lambda= \bigl( \LL,\left( \Lambda_{\bfr} \right)_{\bfr\in\N_0^d}
  \bigr)$.
  Choose $\alpha$, $\beta$, $\delta$ such that (\ref{eq:lem:eps-10}) holds.
Define
$$
\Omega_\eps:= \{\bfx \in \Omega\,|\, T_\bfx(B_\beta) \subset \Omega\}. 
$$
For a function $u \in L^1_{loc}(\R{d})$ define
  \begin{align*}
    \cE u(\bfz) :=  \int_{\bfy} u(\bfy) \rho_{\delta \eps(\bfz)}(\bfz-\bfy), 
\qquad \mbox{ for $\bfz \in \Omega$.}
  \end{align*}
   Then, for $\bfx \in \Omega_\eps$:
\begin{enumerate}[(i)]
\item 
  \label{item:lem:approx-i}
  Suppose $(s,p) \in \N_0 \times [1,\infty]$ satisfies $s\leq k_{\max}+1$. Assume 
  $(s\leq r\in\R{}\text{ and } q\in[1,\infty))$ or
  $(s\leq r\in\N_0\text{ and }q\in[1,\infty])$. Then it holds
  \begin{subequations}\label{lem:approx:eq2}
  \begin{align}\label{lem:approx:eq2:a}
    \sn{\cE u}_{r,q,\Bax} &\leq C_{r,q,s,p,\Lambda} \eps(\bfx)^{s-r+d(1/q-1/p)}\sn{u}_{s,p,\Bbx},
  \end{align}
  where $C_{r,q,s,p,\Lambda}$ depends only on $r$, $q$, $s$, $p$,
  $\LL$, and
  $(\Lambda_{\bfr'})_{\sn{\bfr'}\leq \lceil r \rceil}$ 
  as well as $\rho$, $k_{\max}$, and $\alpha$, $\beta$, $\delta$.
\item 
  \label{item:lem:approx-ii}
  Suppose $0\leq s\in\R{}$, $r\in\N_0$ with $s\leq r\leq k_{\max}+1$,
  and $1\leq p\leq q<\infty$. Define $\mu:=d(p^{-1}-q^{-1})$. Assume that
  $(r=s+\mu\text{ and }p>1)$ or $(r>s+\mu)$. Then it holds 
  \begin{align}\label{lem:approx:eq2:b}
    \sn{u - \cE u}_{s,q,\Bax} &\leq C_{s,q,r,p,\Lambda}\eps(\bfx)^{r-s+d(1/q-1/p)} \sn{u}_{r,p,\Bbx},
  \end{align}
  where $C_{s,q,r,p,\Lambda}$ depends only on $s$, $q$, $r$, $p$,
  $\LL$, and
  $(\Lambda_{\bfs'})_{\sn{\bfs'}\leq \lceil s \rceil}$
  as well as $\rho$, $k_{\max}$, and $\alpha$, $\beta$, $\delta$.
\item
  \label{item:lem:approx-iii}
  Suppose $s$, $r\in\N_0$ with $s\leq r\leq k_{\max}+1$, and $1\leq p < \infty$.
  Define $\mu:=d/p$. Assume that $(r=s+\mu\text{ and }p=1)$ or $(r>s+\mu  \text{ and } p > 1)$.
  Then it holds 
  \begin{align}\label{lem:approx:eq2:c}
    \sn{v - \cE v}_{s,\infty,\Bax} &\leq C_{s,r,p,\Lambda} \eps(\bfx)^{r-s-d/p}\sn{v}_{r,p,\Bbx}.
  \end{align}
  where $C_{s,r,p,\Lambda}$ depends only on $s$, $r$, $p$,
  $\LL$, and $(\Lambda_{\bfs'})_{\sn{\bfs'}\leq s}$
  as well as $\rho$, $k_{\max}$, and $\alpha$, $\beta$, $\delta$.
  \end{subequations}
\end{enumerate}
\end{lem}
\begin{proof}
  First, we check that $(\cE u)\circ T_\bfx = \cE_\bfx (u \circ T_\bfx)$. To that end, 
  let $\bfx \in \Omega_\eps$ and $\bfz \in B_\beta$. 
  We employ the substitution $\bfy = T_\bfx (\bfw)$ and note that $T_\bfx (\bfz) - T_\bfx (\bfw) = 
  \eps(\bfx)(\bfz-\bfw)$:
  \begin{align*}
    \begin{split}
      (\cE u)(T_\bfx (\bfz)) &= (\delta \eps(T_\bfx(\bfz)))^{-d}
\int_\bfy u(\bfy)\rho \left( \frac{T_\bfx(\bfz)-\bfy}{\delta \eps(T_\bfx(\bfz))} \right) 
      = (\delta \eps(T_\bfx(\bfz)))^{-d}
\int_\bfw u(T_\bfx \bfw) \rho \left( \frac{T_\bfx (\bfz) - T_\bfx (\bfw)}{\delta \eps(T_\bfx(\bfz))} \right) \eps(\bfx)^{d}\\
      &= (\delta \xeps(\bfz))^{-d}
\int_\bfw (u \circ T_\bfx)(\bfw) \rho \left(\frac{\bfz-\bfw}{\delta \xeps(\bfz)}\right)  = \cE_\bfx (u \circ T_\bfx)(\bfz).
    \end{split}
  \end{align*}
  Now, the estimates \eqref{lem:approx:eq2} follow from
  Lemma~\ref{lem:approx:refel} using scaling arguments, cf.~also \cite{h14}.
\end{proof}
\subsection{Regularization on a half-space with a Lipschitz-boundary}
\label{sec:regularization-near-boundary}
Lemma~\ref{lem:approx} focuses on regularizing a given function $v$ in the {\em interior} of 
$\Omega$.
If we want to take full advantage of $v$'s regularity, this approach does not
extend up to the boundary. The reason lies in the construction: 
the value $(\cE v)(\bfx)$ is defined by an averaging process in a 
ball of radius $\delta \eps(\bfx)$ around $\bfx$ and thus requires $v$ to be defined in the ball 
$B_{\delta \eps(\bfx)}(\bfx)$. Hence, it cannot be defined in
the point $\bfx$ if $\delta\eps(\bfx)$ is bigger than $\bfx$'s distance to the
boundary $\partial\Omega$.
In the present section, we therefore propose a modification of the averaging 
operator that is based on averaging not on the ball $B_{\delta\eps(\bfx)}(\bfx)$ but on the ball 
${\mathbf b} + B_{\delta\eps(\bfx)}(\bfx)$, where the vector ${\mathbf b}$ (which depends on $\bfx$ and 
$\eps(\bfx)$) is such that ${\mathbf b} + B_{\delta\eps(\bfx)}(\bfx) \subset \Omega$ 
(see (\ref{lem:G:def}) for the precise definition), and an application of Taylor's formula.

Following Stein, \cite{stein1}, we call 
$\Om$ a special Lipschitz domain, if it has the form
\begin{align*}
  \Om = \left\{ \bfx \in \R{d} \mid x_d > f_{\partial\Om}(\bfx_{d-1}) \right\},
\end{align*}
where $f_{\partial\Om}:\R{d-1} \rightarrow \R{}$ is a Lipschitz continuous function with
Lipschitz constant $L_{\partial\Om}$. In this coordinate system, every point
$\bfx \in \R{d}$ has the form $\bfx = (\bfx_{d-1},x_d) \in \R{d}$.
For each $\bfx  = (\bfx_{d-1},x_d) \in \R{d}$, we can then define the set
\begin{align*}
  \cone{\bfx} := \{(\bfy_{d-1},y_d) \in \R{d} \colon  (y_d - x_d) > L_{\partial\Om}
  |\bfy_{d-1} - \bfx_{d-1}| \} \subset\Om.
\end{align*}
It is easy to see that $\cone{\bfx}$ is a convex cone with apex at $\bfx$.
It will also be convenient to introduce the vector
$\bfe_d  = (0,0,\ldots,1) \in \R{d}$.  
We will again analyze the regularization operator near the boundary on
balls $B_{\alpha\eps(\bfx)}(\bfx)$.
Due to the construction mentioned above, the
regularization operator involves the values
of the underlying function on sets which are not balls anymore.
We call these sets $\veebar_\bfx$ (cf. (\ref{eq:veebar})).  
The next two lemmas analyze their properties.
\begin{lem}\label{lem:balls}
  Let $\Om$ be a special Lipschitz domain with Lipschitz constant $L_{\partial\Om}$.
  The following two statements hold:
  \begin{enumerate}[(i)]
    \item
    \label{item:lem:balls-i}
      If $\mu>0$ and $\tau > (L_{\partial\Om}+1)(\mu+1)$,
      then for all $r>0$ and $\bfx\in\Om$ the set
      \begin{align*}
        \cone{\bfx}' := \bigcup_{\bfz\in B_r(\bfx)\cap\Om} \cone{\bfz}
      \end{align*}
      is star-shaped with respect to the ball $B_{\mu r}(\bfx + \tau r \bfe_d)$.
    \item
    \label{item:lem:balls-ii}
      There is a constant $L>1$, which depends only
      on $L_{\partial\Om}$, such that for all $r>0$ and $\bfx\in\Om$,
      all $\bfz\in B_{r}(\bfx)\cap\Om$, and all $r_0>0$ it holds that
      \begin{align*}
        B_{r_0}(\bfz + L r_0 \bfe_d) \subset \cone{\bfx}'.
      \end{align*}
  \end{enumerate}
\end{lem}
\begin{proof}
  First we show (\ref{item:lem:balls-i}):
  Let $\mu>0$ and $\tau > (L_{\partial\Om}+1)(\mu+1)$. Let $r>0$ and
  $\bfx\in\Om$ be arbitrary. It suffices to show that
  $B_{\mu r}(\bfx+\tau r \bfe_d) \subset \cone{\bfz}$ for all $\bfz\in B_{r}(\bfx)$.
  To that end, let $\bfy\in B_{\mu r}(\bfx+\tau r \bfe_d)$ and note that
  \begin{align*}
    y_{d} - z_{d} &= y_{d} - (x_{d} +\tau r) + (x_{d} +\tau r - x_{d} ) + x_{d} - z_{d} 
    \ge - \mu r + \tau r - r = (\tau - \mu - 1)r, \\
    |\bfy_{d-1} - \bfz_{d-1}| &\leq |\bfy_{d-1} - \bfx_{d-1} | + |\bfx_{d-1} - \bfz_{d-1}| 
    \leq \mu r + r  = (\mu + 1) r. 
  \end{align*}
  By the choice of $\tau$, we conclude $\bfy \in \cone{\bfz}$. 

  In order to prove~\eqref{item:lem:balls-ii}, we choose $L> L_{\partial\Om}+1$.
  For $\bfy\in B_{r_0}(\bfz + L r_0 \bfe_d)$, we compute
  \begin{align*}
    y_d - z_d \geq Lr_0 - r_0 = (L-1)r_0\quad\text{ and } \quad
    |\bfy_{d-1} - \bfz_{d-1}| \leq r_0,
  \end{align*}
  and infer $\bfy\in\cone{\bfz}$. Hence, $B_{r_0}(\bfz + L r_0 \bfe_d) \subset \cone{\bfz}$.  
  The assertion (\ref{item:lem:balls-ii}) follows.
\end{proof}
\begin{lem}
  \label{lem:local-balls}
  Let $\Om$ be a special Lipschitz domain with Lipschitz constant $L_{\partial\Om}$
  and let $\eps \in C^\infty(\Omega)$ be a $\Lambda$-admissible length scale
  function. Fix a compact set $\overline{\Omega^\prime}\subset\Omega$.
  Then there are  
  $\alpha$, $\beta$, $\delta$, $L$, $\tau>0$
  such that~\eqref{conv1:lem1:eq1}--\eqref{conv1:lem1:eq3} hold,
  and, additionally, with 
  \begin{equation}
    \label{eq:veebar}
    \veebar_\bfx := B_{\beta \eps(\bfx)}(\bfx)
    \cap \bigcup_{\bfz \in B_{\alpha \eps(\bfx)}(\bfx) \cap \Om}\cone{\bfz}
    \quad\text{ for } \bfx\in\Om,
  \end{equation}
  the following statements 
  (\ref{item:lem:local-balls-i})--(\ref{item:lem:local-balls-iii}) are true: 
  \begin{enumerate}[(i)]
    \item \label{item:lem:local-balls-i}  For $\bfx_0 \in \Om$, the point
      $\bfx_0 + \delta \eps(\bfx_0)L \bfe_d$ satisfies 
      \begin{align*}
	{\rm dist}(\bfx_0 + \delta \eps(\bfx_0)L \bfe_d,\partial \Om) > \delta \eps(\bfx_0). 
      \end{align*}
    \item\label{item:lem:local-balls-ii} 
      For every $\bfx \in \Omega$ and every $\bfx_0 \in B_{\alpha \eps(\bfx)}(\bfx) \cap \Om$ there holds
      $B_{\delta \eps(\bfx_0)}(\bfx_0 + L\delta \eps(\bfx_0) \bfe_d) \subset \veebar_\bfx$.
    \item\label{item:lem:local-balls-iii} The set $\veebar_\bfx$ is star-shaped with respect to
      $B_{\delta\eps(\bfx)}(\bfx+\tau\alpha\eps(\bfx)\bfe_d)$.
  \end{enumerate}
\end{lem}
\begin{proof}
  First, choose $L$ from Lemma~\ref{lem:balls}, (\ref{item:lem:balls-ii}). 
  Then, choose $\alpha$, $\beta$, $\delta>0$ according to
  Lemma~\ref{lem:eps} so that \eqref{conv1:lem1:eq1}--\eqref{conv1:lem1:eq3} hold.
   Set $\mu:= \delta/\alpha > 0$. Since $\beta > (L_{\partial\Om}+1)(\delta/\alpha + 1) \alpha + \delta
 = (L_{\partial\Om}+1)(1+\mu) \alpha +\delta$, we can choose $\tau > (L_{\partial\Om}+1)(1+\mu)$ 
  (i.e., such that the condition in Lemma~\ref{lem:balls}, (\ref{item:lem:balls-i}) is satisfied) and simultaneously
  $\beta>\tau\alpha+\delta$.  This shows
  \begin{align}\label{lem:local-balls:eq1}
      B_{\delta\eps(\bfx)}(\bfx+\tau\alpha\eps(\bfx)\bfe_d)
      &\subset B_{\beta\eps(\bfx)}(\bfx).
  \end{align}
  Furthermore, by Lipschitz continuity of $\eps$, we have 
  $\eps(\bfx_0)\leq(1+\LL\alpha)\eps(\bfx)$ for all $\bfx_0\in \Bax$. Hence,
  in view of \eqref{conv1:lem1:eq2}, we infer 
  \begin{align*}
    \beta\eps(\bfx) 
    \geq \left[ \alpha+\delta(L+1)(1+\LL\alpha) \right]\eps(\bfx)
    \geq \alpha\eps(\bfx) + (L+1)\delta\eps(\bfx_0),
  \end{align*}
  which implies
  \begin{align}\label{lem:local-balls:eq2}
    B_{\delta \eps(\bfx_0)}(\bfx_0 + L\delta \eps(\bfx_0)\bfe_d)
    &\subset B_{\beta\eps(\bfx)}(\bfx) \quad \text{ for all } \bfx_0\in \Bax.
  \end{align}
  We now prove statement (\ref{item:lem:local-balls-ii}). 
  For $\bfx_0 \in B_{\alpha \eps(\bfx)}(\bfx)\cap \Om$ we choose $r = \alpha \eps(\bfx)$
  and $r_0=\delta\eps(\bfx_0)$ in Lemma~\ref{lem:balls}, (\ref{item:lem:balls-ii}) and obtain
  \begin{align*}
    B_{\delta \eps(\bfx_0)}(\bfx_0 + L\delta \eps(\bfx_0)\bfe_d)
    \subset \bigcup_{\bfz \in B_{\alpha \eps(\bfx)}(\bfx)\cap \Om} \cone{\bfz}.
  \end{align*}
  Together with~\eqref{lem:local-balls:eq2}, this shows (\ref{item:lem:local-balls-ii}).
  Furthermore, choosing $r = \alpha \eps(\bfx)$ in Lemma~\ref{lem:balls}, (\ref{item:lem:balls-i}),
  we see that
  \begin{align*}
    \bigcup_{\bfz \in B_{\alpha \eps(\bfx)}(\bfx)\cap \Om} \cone{\bfz} \qquad
    \text{ is star-shaped w.r.t. } B_{\delta\eps(\bfx)}(\bfx+\tau\alpha\eps(\bfx)\bfe_d).
  \end{align*}
  Together with~\eqref{lem:local-balls:eq1}, this shows statement (\ref{item:lem:local-balls-iii}).
  Finally, statement (\ref{item:lem:local-balls-i}) follows from assertion (\ref{item:lem:local-balls-ii})
  since $\veebar_\bfx \subset \Om$. 
\end{proof}
\begin{lem}\label{lem:G}
  Let $\rho$ be a mollifier of order $k_{\max}$, and let $\Om$ be a
  special Lipschitz domain with Lipschitz constant $L_{\partial\Om}$. Assume that 
  $\eps \in C^\infty(\Omega)$ is a $\Lambda$-admissible length scale
  function, $\Lambda= \bigl( \LL,\left( \Lambda_{\bfr} \right)_{\bfr\in\N_0^d}
  \bigr)$. Fix a compact set $\overline{\Omega^\prime}\subset\Omega$.
  Choose $\alpha$, $\beta$, $\delta$, $L$, $\tau$ according to Lemma~\ref{lem:local-balls}.
  Then define the operator
  \begin{align}
    \cG u(\bfx_0) :=
    \sum_{\sn{\bfk}\leq k_{\max}}
    \frac{(-\delta \eps(\bfx_0)L\bfe_d)^{\bfk}}{\bfk !}
    \left( D^{\bfk} (u \star \rho_{\delta \eps(\bfx_0)}) \right)\left(\bfx_0 + \delta \eps(\bfx_0)L\bfe_d\right), 
    \qquad \bfx_0 \in \Om.
    \label{lem:G:def}
  \end{align}
\begin{enumerate}[(i)]
\item
\label{item:lem:G-i}
  Suppose $(s,p) \in \N_0 \times [1,\infty]$ satisfies $s\leq k_{\max}+1$. Assume 
  $(s\leq r\in\R{}\text{ and } q\in[1,\infty))$ or
  $(s\leq r\in\N_0\text{ and }q\in[1,\infty])$. Then it holds
  \begin{subequations}\label{eq:lem:G}
  \begin{align}\label{eq:lem:G:a}
    \sn{\cG v}_{r,q,\Bax\cap\Om} &\leq C_{r,q,s,p,\Lambda,L_{\partial\Om}}
    \eps(\bfx)^{s-r+d(1/q-1/p)}\sn{v}_{s,p,\veebar_\bfx},
  \end{align}
  where $C_{r,q,s,p,\Lambda,L_{\partial\Om}}$ depends only on $r$, $q$, $s$, $p$,
  $\LL$,
  $(\Lambda_{\bfr'})_{\sn{\bfr'}\leq \lceil r \rceil}$, $L_{\partial\Om}$, 
  as well as on $\rho$, $k_{\max}$, $\alpha$, $\beta$, $\delta$.
\item
\label{item:lem:G-ii}
  Suppose $0\leq s\in\R{}$, $r\in\N_0$ with $s\leq r\leq k_{\max}+1$,
  and $1\leq p\leq q<\infty$. Define $\mu:=d(p^{-1}-q^{-1})$. Assume that
  $(r=s+\mu\text{ and }p>1)$ or $(r>s+\mu)$. Then it holds that
  \begin{align}\label{eq:lem:G:b}
    \sn{v - \cG v}_{s,q,\Bax\cap\Om} &\leq C_{s,q,r,p,\Lambda,L_{\partial\Om}}
    \eps(\bfx)^{r-s+d(1/q-1/p)}\sn{v}_{r,p,\veebar_\bfx},
  \end{align}
  where $C_{s,q,r,p,\Lambda,L_{\partial\Om}}$ depends only on $s$, $q$, $r$, $p$,
  $\LL$,
  $(\Lambda_{\bfs'})_{\sn{\bfs'}\leq \lceil s \rceil}$, $L_{\partial\Om}$
  as well as on $\rho$, $k_{\max}$, $\alpha$, $\beta$, $\delta$.
\item
\label{item:lem:G-iii}
  Suppose $s$, $r\in\N_0$ with $s\leq r\leq k_{\max}+1$, and $1\leq p < \infty$.
  Define $\mu:=d/p$. Assume that $(r=s+\mu\text{ and }p=1)$ or $(r>s+\mu  \text{ and } p > 1)$.
  Then it holds that
  \begin{align}\label{eq:lem:G:c}
    \sn{v - \cG v}_{s,\infty,\Bax\cap\Om} &\leq C_{s,r,p,\Lambda,L_{\partial\Om}}
    \eps(\bfx)^{r-s-d/p} \sn{v}_{r,p,\veebar_\bfx},
  \end{align}
  \end{subequations}
  where $C_{s,r,p,\Lambda,L_{\partial\Om}}$ depends only on $s$, $r$, $p$,
  $\LL$,
  $(\Lambda_{\bfs'})_{\sn{\bfs'}\leq s}$, $L_{\partial\Om}$
  as well as on $\rho$, $k_{\max}$, $\alpha$, $\beta$, $\delta$.
\end{enumerate}
\end{lem}
\begin{proof}
  Since $u\mapsto u \star \rho_{\delta \eps(\bfx_0)}$ is a classical convolution operator with fixed length scale,
  we may write
  \begin{align*}
    D^\bfk (u \star \rho_{\delta \eps(\bfx_0)}) =
    u \star (D^\bfk \rho_{\delta \eps(\bfx_0)})
    = u \star (D^\bfk \rho)_{\delta \eps(\bfx_0)} (\delta \eps(\bfx_0))^{-\sn{\bfk}},
  \end{align*}
  and a change of variables gives (assuming $u$ is smooth)
  \begin{align*}
    D^\bfr_{\bfx_0} \cG u(\bfx_0)
    &= \sum_{\sn{\bfk}\leq k_{\max}}\frac{(-L\bfe_d)^{\bfk}}{\bfk!} \int_{\bfy \in B_1(0)}
    D^{\bfr}_{\bfx_0} u\left(\bfx_0'\right)
    D^{\bfk}\rho(\bfy),
  \end{align*}
  where $\bfx_0' := \bfx_0+\eps(\bfx_0)\left( \delta \bfe_d L - \delta \bfy\right)$.
  The Fa\`a di Bruno formula from Lemma~\ref{lem:faa-1-estimate} shows 
  \begin{align*}
    \begin{split}
      D^{\bfr}_{\bfx_0} u\left(\bfx_0'\right)
      &= \left( D^{\bfr}u \right)\left(\bfx_0'\right) +
      \sum_{\sn{\bft}\leq \sn{\bfr}}\left( D^{\bft}u \right)\left(\bfx_0'\right)
      E_{\bfr,\bft}\left( \delta \bfe_d L-\delta \bfy,\bfx_0 \right)\\
      &=\left( -\delta\eps(\bfx_0) \right)^{-\sn{\bfr}} D^{\bfr}_\bfy u\left(\bfx_0'\right) +
      \sum_{\sn{\bft}\leq \sn{\bfr}}
      \left( -\delta\eps(\bfx_0) \right)^{-\sn{\bft}}
      D^{\bft}_\bfy u\left(\bfx_0'\right)
      E_{\bfr,\bft}\left( \delta \bfe_d L-\delta \bfy,\bfx_0 \right).
    \end{split}
  \end{align*}
  We obtain with integration by parts and the product rule
  \begin{align}\label{eq:lem:G:100}
    \begin{split}
    D^\bfr_{\bfx_0} \cG u(\bfx_0)
    &= \sum_{\sn{\bfk}\leq k_{\max}}\frac{(-L\bfe_d)^{\bfk}}{\bfk!}
    (-1)^{\sn{\bfr}}
    \int_{\bfy \in B_1(0)}
      u\left(\bfx_0'\right)
      D^{\bfk+\bfr}\rho(\bfy)
      \left( -\delta\eps(\bfx_0) \right)^{-\sn{\bfr}}\\
    &\quad+ \sum_{\sn{\bfk}\leq k_{\max}}\frac{(-L\bfe_d)^{\bfk}}{\bfk!}
    \sum_{\sn{\bft}\leq\sn{\bfr}}(-1)^{\sn{\bft}}
    \sum_{\bfs\leq\bft}\binom{\bfs}{\bft}\\
    &\qquad\quad\int_{\bfy \in B_1(0)}
    u(\bfx_0') D^\bfs_\bfy E_{\bfr,\bft}(\delta\bfe_dL-\delta\bfy,\bfx_0)
      D^{\bfk+\bft-\bfs}\rho(\bfy)
      \left( -\delta\eps(\bfx_0) \right)^{-\sn{\bft}}.
    \end{split}
  \end{align}
  The estimates now follow as in Lemma~\ref{lem:approx:refel}. We apply
  Lemma~\ref{lem:faa-1-estimate}, Lemma~\ref{lem:local-balls},~\eqref{item:lem:local-balls-ii},
  and integration by parts to obtain from identity~\eqref{eq:lem:G:100} the estimate
  \begin{align*}
    \sn{\cG u}_{r,\infty,\Bax\cap\Om} \leq C_{r,\Lambda,L_{\partial\Om}}
    \eps(\bfx)^{-r}
    \vn{u}_{0,1,\veebar_\bfx}.
  \end{align*}
  This is the analogue of~\eqref{lem:approx:refel:eq:30}, such that the remainder
  of the proof follows as in Lemma~\ref{lem:approx:refel}. Note that
  we now employ the embedding results of Theorem~\ref{thm:embedding} as well
  as the Bramble-Hilbert Lemma~\ref{lem:bramblehilbert} on the domain
  $\veebar_\bfx$, which we may since
  by Lemma~\ref{lem:local-balls} the chunkiness $\eta(\veebar_\bfx)$ is bounded uniformly in $\bfx$, 
  i.e., $\eta(\veebar_\bfx)\lesssim 1$ uniformly in $\bfx$.
\end{proof}
\subsection{Proof of Theorem~\ref{thm:hpsmooth}}\label{section:proof:hpsmooth}
In the last section, we derived results on the stability and approximation
properties of the two operators $\cE$ and $\cG$. These results, however,
are strongly localized. Now we show how they can be globalized.
The main ingredients are a partition of unity, the local properties of
the smoothing operators, and Besicovitch's Covering Theorem.

We start with a lemma that follows from Besicovitch's Covering Theorem (Proposition~\ref{thm:besicovitch}):
\begin{lem}
\label{lem:besicovitch} 
Let $\Omega\subset\R{d}$ be an arbitrary domain and $\eps \in C^\infty(\Omega)$ be a
$\Lambda = \bigl( \LL,\left( \Lambda_{\bfr} \right)_{\bfr\in\N_0^d} \bigr)$-admissible
length scale function. Let 
$\alpha$, $\beta > 0$ satisfy (\ref{eq:lem:eps-10}). Let $N_d$ be given by 
Proposition~\ref{thm:besicovitch}. 

Let $\omega\subset \Omega$ be arbitrary. Then there exist points 
$\bfx_{ij} \in \omega$, $i=1,\ldots,N_d$, $j \in \N$, such that for the 
closed balls 
\begin{equation}
\label{eq:lem:besicovitch-1}
B_{ij}:= \overline{B_{\frac{\alpha}{2} \eps(\bfx_{ij})}(\bfx_{ij})} 
\subset \widetilde B_{ij}:= \overline{B_{\alpha \eps(\bfx_{ij})}(\bfx_{ij})} 
\subset 
\widehat B_{ij}:= \overline{B_{\beta \eps(\bfx_{ij})}(\bfx_{ij})}, 
\end{equation}
the following is true: 
\begin{enumerate}[(i)]
\item 
\label{item:lem:besicovitch-i} 
 $\displaystyle \omega \subset \cup_{i=1}^{N_d} \cup_{j\in\N} B_{ij}$. 
\item 
\label{item:lem:besicovitch-ii} 
For each $i\in \{1,\ldots,N_d\}$, the balls $\{B_{ij}\,|\, j \in \N\}$ are pairwise disjoint.
\item 
\label{item:lem:besicovitch-iii} 
For each $i \in \{1,\ldots,N_d\}$, the balls $\widehat B_{ij}$, $j \in \N$ satisfy an overlap property: 
There exists $C_{\textrm{overlap}} >0$, which depends solely on $d$, $\alpha$,
$\beta$, $\LL$, such that 
$$
\operatorname*{card} \{j^\prime \,|\, \widehat B_{ij} \cap \widehat B_{ij^\prime} \ne\emptyset\} \leq C_{\textrm{overlap}}
\qquad \forall j \in \N.
$$
\item 
\label{item:lem:besicovitch-iv} Let $q \in [1,\infty)$ and $\sigma \in (0,1)$. Then for a constant $C$ that depends solely
on $\sigma$, $q$, $d$, and $\alpha$: 
\begin{align}
\label{eq:lem:besicovitch-10}
\sum_{i=1}^{N_d} \sum_{j \in \N} \int_{\bfx \in B_{ij}} \int_{\bfy \in \omega\setminus \widetilde B_{ij}} 
  \frac{|v(\bfx)|^q}{|\bfx - \bfy|^{\sigma q + d}}\,d\bfy\,d\bfx 
\leq C \sum_{i=1}^{N_d} \sum_{j \in \N} \eps(\bfx_{ij})^{-\sigma q} \|v\|^q_{0,q,\widetilde B_{ij}}, \\
\label{eq:lem:besicovitch-20}
\sum_{i=1}^{N_d} \sum_{j \in \N} \int_{\bfx \in B_{ij}} \int_{\bfy \in \omega\setminus \widetilde B_{ij}} 
  \frac{|v(\bfy)|^q}{|\bfx - \bfy|^{\sigma q + d}}\,d\bfy\,d\bfx 
\leq C \sum_{i=1}^{N_d} \sum_{j \in \N} \eps(\bfx_{ij})^{-\sigma q} \|v\|^q_{0,q,\widetilde B_{ij}}. 
\end{align}
\end{enumerate}
\end{lem}
\begin{proof}
  {\em Proof of} (\ref{item:lem:besicovitch-i}), (\ref{item:lem:besicovitch-ii}):
  The set $\cF := \left\{ \overline{B_{\alpha \eps (\bfx)/2}(\bfx)} \mid \bfx \in \omega \right\}$
  is a closed cover of $\omega$. According to Besicovitch's Covering Theorem (Proposition~\ref{thm:besicovitch}),
  there are countable subsets $\cG_1, \dots, \cG_{N_d}$ of $\cF$, the elements of every subset
  being pairwise disjoint, which have the form $\cG_i = \{B_{ij}\,|\, j \in \N\}$ with 
  $B_{ij} = \overline{B_{\alpha \eps (\bfx_{ij})/2}(\bfx_{ij})}$ for some $\bfx_{ij} \in \omega$, and 
  \begin{align*}
    \omega \subset \bigcup_{i=1}^{N_d} \bigcup_{j \in \N}
    \overline{B_{\alpha \eps (\bfx_{ij})/2}(\bfx_{ij})}.
  \end{align*}

  {\em Proof of} (\ref{item:lem:besicovitch-iii}):
  Suppose $\widehat{B}_{ij} \cap \widehat{B}_{ij'} \neq \emptyset$. Then, the Lipschitz continuity of $\eps$ gives 
  \begin{align*}
    \sn{\eps(\bfx_{ij})-\eps(\bfx_{ij'})} \leq \LL
    \vn{\bfx_{ij}-\bfx_{ij'}} \leq \LL \beta \left( \eps(\bfx_{ij})+\eps(\bfx_{ij'}) \right).
  \end{align*}
  Since $\beta \LL<1$ due to our assumption (\ref{eq:lem:eps-10}), we conclude 
  \begin{align}\label{eq:eps:compare}
    \widehat{B}_{ij} \cap \widehat{B}_{ij'} \neq \emptyset \implies
    \eps(\bfx_{ij}) \leq \eps(\bfx_{ij'})\frac{1+\LL \beta}{1-\LL \beta}.
  \end{align}
  For $\bfz \in B_{ij'}$ with $\widehat{B}_{ij} \cap \widehat{B}_{ij'} \neq \emptyset$
  we use~\eqref{eq:eps:compare} to estimate
  \begin{align}
    \begin{split}
    \vn{\bfz - \bfx_{ij}} &\leq \vn{\bfz - \bfx_{ij'}} + \vn{\bfx_{ij'} - \bfx_{ij}}
    \leq \alpha \eps(\bfx_{ij'})/2 + \beta \left( \eps(\bfx_{ij}) + \eps(\bfx_{ij'}) \right)\\
    &\leq \eps(\bfx_{ij}) \left( \frac{\alpha}{2} 
    \frac{1+\LL \beta}{1-\LL \beta} + \beta + \beta \frac{1+\LL \beta}{1-\LL \beta} \right) 
=: \eps(\bfx_{ij}) C_{\textrm{big}}. 
  \end{split}
  \label{eq:aux1}
  \end{align}
  For fixed $i$, the balls $B_{ij'}$, $j' \in \N$ are pairwise disjoint.
  Thus, with $C_d$ denoting the volume of the unit sphere in $\R{d}$
  \begin{align}\label{eq:aux100}
    \sum_{\substack{j' \in \N\colon   \widehat{B}_{ij} \cap \widehat{B}_{ij'} \neq \emptyset}} \sn{B_{ij'}}
    \leq \sn{B_{\eps(\bfx_{ij}) C_{\textrm{big}}}(\bfx_{ij})} = C_d \left( \eps(\bfx_{ij})
    C_{\textrm{big}} \right)^d.
  \end{align}
  We use~\eqref{eq:eps:compare} and~\eqref{eq:aux100} to bound the number of balls that intersect a given one:
  \begin{align}
\label{eq:aux2}
      \operatorname*{card} \left\{ j' \mid \widehat{B}_{ij} \cap \widehat{B}_{ij'} \neq \emptyset \right\} &=
      \sum_{\substack{j' \in \N \colon\\  \widehat{B}_{ij} \cap \widehat{B}_{ij'} \neq \emptyset}} \frac{\sn{B_{ij'}}}{\sn{B_{ij'}}} =
      \frac{2^d}{\alpha^d C_d}\sum_{\substack{j' \in \N \\  \widehat{B}_{ij} \cap \widehat{B}_{ij'} \neq \emptyset}} \frac{\sn{B_{ij'}}}{\eps(\bfx_{ij'})^d}
\\ &
\nonumber 
\leq
      \frac{2^d}{\alpha^d C_d} \left( \frac{1+\LL \beta}{1-\LL \beta} \right)^d \eps(\bfx_{ij})^{-d}
      \sum_{\substack{j' \in \N \colon \\  \widehat{B}_{ij} \cap \widehat{B}_{ij'} \neq \emptyset}}\sn{B_{ij'}}
\leq \left( \frac{2 C_{\textrm{big}}}{\alpha} \frac{1+\LL \beta}{1-\LL \beta} \right)^d =: C_{\textrm{overlap}}.
  \end{align}
  {\em Proof of} (\ref{item:lem:besicovitch-iv}): We start with the simpler estimate, (\ref{eq:lem:besicovitch-10}). It follows
directly from the observation that for $\bfx \in B_{ij}$ we have 
$B_{\frac{\alpha}{2} \eps(\bfx_{ij})}(\bfx) \subset B_{\alpha \eps(\bfx_{ij})}(\bfx_{ij}) = \widetilde B_{ij}$ so that  
$$
\int_{\bfy \in \omega\setminus \widetilde B_{ij}} \frac{1}{|\bfx - \bfy|^{\sigma q + d}}\,d\bfy
\leq \int_{\bfy \in \R{d}\setminus B_{\frac{\alpha}{2}\eps(\bfx_{ij})}(0)} \frac{1}{|\bfy|^{\sigma q + d}}\,d\bfy
 = C_{\alpha,d,\sigma,q} \eps(\bfx_{ij})^{-\sigma q}, 
$$
where the constant $C_{\alpha,d,\sigma,q}$ depends on the quantities indicated. 

We turn to the estimate (\ref{eq:lem:besicovitch-20}). We essentially repeat the arguments of \cite[Lemma~{3.1}]{f02}.  
Using $\omega \subset \cup_{ij} B_{ij}$ and the notation $\chi_A$ for the characteristic function of a set $A$, 
we have to estimate
\begin{align}
\nonumber 
& \sum_{i=1}^{N_d} \sum_{i^\prime=1}^{N_d} \sum_{j \in \N} \sum_{j^\prime \in \N} 
\int_{\bfx \in B_{ij}} \int_{\bfy \in B_{i^\prime j^\prime} \setminus \widetilde B_{ij}} 
 \frac{|v(\bfy)|^q}{|\bfx - \bfy|^{\sigma q + d}}\,d\bfy \,d\bfx \\
\label{eq:lem:besicovitch-50}
& = 
\sum_{i^\prime=1}^{N_d} \sum_{j^\prime \in \N} \int_{\bfy \in B_{i^\prime j^\prime}} |v(\bfy)|^q 
\int_{\bfx \in \R{d}} 
\sum_{i=1}^{N_d} \sum_{j \in \N} \chi_{B_{ij}}(\bfx) \chi_{B_{i^\prime j^\prime}\setminus \widetilde B_{ij}}(\bfy)
\frac{1}{|\bfx - \bfy|^{\sigma q + d}}\,d\bfx \,d\bfy. 
\end{align}
To  that end, we analyze 
$\sum_{i=1}^{N_d} \sum_{j \in \N} \chi_{B_{ij}}(\bfx) \chi_{B_{i^\prime j^\prime}\setminus \widetilde B_{ij}}(\bfy)$ 
in more detail. Pick $\lambda > 0$ such that 
\begin{equation} 
\label{eq:lem:besicovitch-200}
1 - \frac{\alpha}{2} \LL  - \lambda \LL  < 1 
\qquad \mbox{ and } 
\qquad 
\frac{\alpha}{2} \frac{1 - \frac{\alpha}{2} \LL - \lambda \LL}{1 +\alpha/2} > \lambda > 0; 
\end{equation}
this is possible due to (\ref{eq:lem:eps-10}). 
We claim that the following is true: 
\begin{equation}
\label{eq:lem:besicovitch-100}
\bfy \in B_{i^\prime j^\prime} \qquad \Longrightarrow \qquad 
C(\bfx,\bfy):= \sum_{i=1}^{N_d} \sum_{j \in \N} \chi_{B_{ij}}(\bfx) \chi_{B_{i^\prime j^\prime}\setminus \widetilde B_{ij}}(\bfy) 
\leq 
\begin{cases} 
N_d & \mbox{ if } |\bfx - \bfy| \ge \lambda \eps(\bfx_{i^\prime j^\prime}) \\
0   & \mbox{ if } |\bfx - \bfy| < \lambda \eps(\bfx_{i^\prime j^\prime}). 
\end{cases} 
\end{equation}
The desired final estimate (\ref{eq:lem:besicovitch-20}) then follows from inserting (\ref{eq:lem:besicovitch-100}) into
(\ref{eq:lem:besicovitch-50}) and the introduction of polar coordinates to evaluate the integral in $\bfx$. 
In order to see (\ref{eq:lem:besicovitch-100}), fix $(i^\prime,j^\prime) \in \{1,\ldots,N_d\} \times \N$. 
Since the sets $\{B_{ij}\,|\, j \in \N\}$ are pairwise disjoint for each $i$, it is clear that 
$C(\bfx,\bfy) \leq N_d$. We have therefore to ascertain that $|\bfx -  \bfy| < \lambda \eps(\bfx_{i^\prime j^\prime})$ 
implies $C(\bfx,\bfy) = 0$. Let $|\bfx - \bfy| < \lambda \eps(\bfx_{i^\prime j^\prime})$. We proceed by 
contradiction. Suppose $C(\bfx,\bfy) \ne 0$. 
Then, we must have $\bfx \in B_{ij}$, $\bfy \in B_{i^\prime j^\prime}$ and $\bfy \not\in \widetilde B_{ij}$. Hence, we conclude 
$|\bfy - \bfx| \ge \frac{\alpha}{2} \eps(\bfx_{ij})$. Thus, 
\begin{equation}
\label{eq:lem:besicovitch-300}
\frac{\alpha}{2} \eps(\bfx_{ij}) \leq |\bfx - \bfy| < \lambda \eps(\bfx_{i^\prime j^\prime}). 
\end{equation}
Next, we use the Lipschitz continuity of $\eps$: 
\begin{align*}
\eps(\bfx_{i^\prime j^\prime}) \leq 
\eps(\bfx_{i j})  + \LL |\bfx_{ij} - \bfx_{i^\prime j^\prime}| 
& \leq 
\eps(\bfx_{i j})  + \LL |\bfx_{ij} - \bfx| + \LL |\bfx - \bfy | + \LL |\bfy - \bfx_{i^\prime j^\prime}|  \\
& \leq 
\eps(\bfx_{i j})  + \frac{\alpha}{2} \LL \eps(\bfx_{ij}) + \lambda \LL \eps(\bfx_{i^\prime j^\prime}) 
 + \frac{\alpha}{2} \LL \eps(\bfx_{i^\prime j^\prime}). 
\end{align*}
Rearranging the terms yields 
$$
\left( 1 - \frac{\alpha}{2} \LL - \lambda \LL\right) \eps(\bfx_{i^\prime j^\prime}) \leq  \left( 1 + \frac{\alpha}{2} \right) \eps(\bfx_{ij})
\stackrel{\eqref{eq:lem:besicovitch-300}}{\leq} \left(1+\frac{\alpha}{2}\right) \frac{2}{\alpha} \lambda 
\eps(\bfx_{i^\prime j^\prime}), 
$$
which contradicts (\ref{eq:lem:besicovitch-200}). 
\end{proof}
\begin{proof}[Proof of Theorem~\ref{thm:hpsmooth}]
  As it is standard in the treatment of bounded Lipschitz domains, we proceed with a localization 
  procedure. Let $\left\{ U^j \right\}_{j\geq 0}$ be an open cover of $\Omega$
  such that $\left\{ U^j \right\}_{j\geq 1}$ is an open cover of
  $\partial\Omega$. According to~\cite[Thm.~3.15]{adams-fournier03},
  there is a $C^\infty$-partition of unity of $\Omega$ subordinate to $\left\{ U^j \right\}_{j\geq 0}$,
  i.e., $\sum_{j\geq 0} \eta^j = 1 \text{ on } \overline{\Om},
  \text{ and } 0 \leq \eta^j \in C^\infty_0(U^j)$. 

  For $u \in L^1_{loc}(\Om)$ we write $u = u \eta^0 + \sum_{j\geq 1} u \eta^j$,
  and we extend $u \eta^j$ to $\R{d}$ by zero. For the definition of $\II_\eps$, we will apply the operator $\cE$ of
  Lemma~\ref{lem:approx} to $u\eta^0$, while we will apply the operator $\cG$ of
  Lemma~\ref{lem:G} to $u\eta^j$ for $j\geq 1$.
  We may assume that $\partial \Om \cap U^j$ can (after translation and rotation) be extended to
  $\left\{ \bfx \in \R{d} \mid \bfx_d = g^j(\bfx_1, \dots, \bfx_{d-1}) \right\}$
  and $g^j$ has the same Lipschitz constant as $\partial \Om$, i.e.,
  we consider $\partial \Om \cap U^j$ as a special Lipschitz domain. 
  Thus, we see that the operator
  \begin{align*}
    \II u := \cE (u \eta^0) + \sum_{j\geq 1} \cG\circ Q^j (u\eta^j),
  \end{align*}
  $Q^j$ being an appropriate Euclidean coordinate transformation, is well defined. 
The operators $\cE$ and $\cG$ are defined in terms of parameters $\alpha$, $\beta$, $\delta$, $\tau$, and $L$. 
These are chosen according to Lemma~\ref{lem:local-balls} 
where we set $\overline{\Omega^\prime} = \supp(\eta^0)$.
(Note that the key parameter $L_{\partial\Omega}$ is determined by the Lipschitz constant of the special 
Lipschitz domains, which in turn are controlled by the Lipschitz character of $\Omega$). 
This implies in particular that
\begin{equation}
\label{eq:thmhpsmooth-100}
\bfx \in \Omega \qquad \Longrightarrow \qquad 
\mbox{ either } \quad \overline{B_{2 \beta \eps(\bfx)}(\bfx)} \subset \Omega \quad \mbox{ or }  \quad 
\overline{B_{2\beta \eps(\bfx)}(\bfx)} \cap \overline{\Omega^\prime} = \emptyset. 
\end{equation}
  With these choices, we may apply Lemmas~\ref{lem:approx} and~\ref{lem:G}.

Let $\omega\subset \Omega$ and let the closed balls $B_{ij}$, $\widetilde B_{ij}$, $\widehat B_{ij}$ 
be given by (\ref{eq:lem:besicovitch-1}) of Lemma~\ref{lem:besicovitch}. 
  We first show the stability~\eqref{thm:hpsmooth:stab} of $\II_\eps$ in the case that $r\in\N_0$ and $s\in\N_0$.
  Applying the triangle inequality to the definition gives
  \begin{align}\label{thm:hpsmooth:eq1}
    \sn{\II_\eps u}_{r,q,\omega} \leq \sn{\cE(u\eta^0)}_{r,q,\omega}
    + \sum_{\ell=1}^m \sn{\cG\circ Q^j(u\eta^\ell)}_{r,q,\omega}.
  \end{align}

  We start by focusing on the contribution $\cE (u \eta^0)$ in (\ref{thm:hpsmooth:eq1}). 
First, we remark for $\cE (u \eta^0)$ that (\ref{eq:thmhpsmooth-100})  implies 
\begin{equation}
\label{eq:thmhpsmooth-110}
\bfx \in \widetilde U^0:= \operatorname*{supp} (\cE (u \eta^0)) 
\qquad \Longrightarrow 
\qquad B_{\beta \eps(\bfx)}(\bfx) \subset \Omega. 
\end{equation}

  We note the covering property stated in Lemma~\ref{lem:besicovitch}, (\ref{item:lem:besicovitch-i}). 
  With the estimate~\eqref{lem:approx:eq2:a} of Lemma~\ref{lem:approx}  we get 
  \begin{align*}
    \sn{\cE(u\eta^0)}_{r,q,\omega}^q & = 
    \sn{\cE(u\eta^0)}_{r,q,\omega\cap \widetilde U^0}^q \leq
    \sum_{i=1}^{N_d}\sum_{j\in\N}
    \sn{\cE(u\eta^0)}_{r,q,B_{ij}\cap \widetilde U^0}^q 
\stackrel{(\ref{eq:thmhpsmooth-110}), (\ref{lem:approx:eq2:a})}{\lesssim}  
    \sum_{i=1}^{N_d}\sum_{j\in\N} \eps(\bfx_{ij})^{q(s-r) + d(1-q/p)}
    \sn{u\eta^0}_{s,p,\widehat B_{ij}}^q.
  \end{align*}
  Noting that $\eps(\bfx_{ij}) \simeq \eps(\bfz)$ for all $\bfz\in \widehat B_{ij} \cap \Omega$ we obtain 
  \begin{align*}
    \eps(\bfx_{ij})^{q(s-r) + d(1-q/p)}\sn{u\eta^0}_{s,p,\widehat B_{ij}}^q
    &\lesssim \eps(\bfx_{ij})^{q(s-r)+d(1-q/p)}
    \vn{u}_{s,p,\Om\cap \widehat B_{ij}}^q
    \lesssim \sum_{\sn{\bfs}\leq s}
    \vn{\eps^{s-r+d(1/q-1/p)}D^\bfs u}_{0,p,\Om\cap \widehat B_{ij}}^q. 
  \end{align*}
  We combine the last two estimates and~\eqref{eq:aux2} to arrive at 
  \begin{align*}
    \sn{\cE(u\eta^0)}_{r,q,\omega}^q &\lesssim
    \sum_{i=1}^{N_d}\sum_{j\in\N} \sum_{\sn{\bfs}\leq s}
    \vn{\eps^{s-r+d(1/q-1/p)}D^\bfs u}_{0,p,\Om\cap \widehat B_{ij}}^q
    \lesssim N_d C_{\textrm{overlap}} \sum_{\sn{\bfs}\leq s}
    \vn{\eps^{s-r+d(1/q-1/p)}D^\bfs u}_{0,p,\omega_\eps}^q. 
  \end{align*}
  This concludes the proof of the stability bound (\ref{thm:hpsmooth:stab}) for 
  $r$, $s \in \N_0$. We turn to the case $r\in\R{}\setminus\N_0$ and $q\in[1,\infty)$ together with $s \in \N_0$. 
For $\sn{\bfr} = \lfloor r \rfloor$
  and $\sigma = r-\lfloor r \rfloor\in(0,1)$, we recall the definition of the sets $B_{ij}$ and $\widetilde B_{ij}$ 
and write 
  \begin{align*}
    \sn{D^\bfr \cE (u\eta^0)}_{\sigma,q,\omega}^q
& = 
    \sn{D^\bfr \cE (u\eta^0)}_{\sigma,q,\omega\cap \widetilde U^0}^q
    \leq \sum_{i=1}^{N_d}\sum_{j\in\N} \int_{B_{ij}\cap \widetilde U^0} \int_{\omega\cap \widetilde U^0}
    \frac{\sn{D^\bfr \cE (u\eta^0) (\bfx) - D^\bfr \cE (u\eta^0) (\bfy)}^q}{\sn{\bfx-\bfy}^{\sigma q + d}}\;d\bfy\;d\bfx\\
    &\leq \sum_{i=1}^{N_d}\sum_{j\in\N}
    \sn{D^\bfr \cE (u\eta^0) }^q_{\sigma,q,\widetilde B_{ij} \cap \widetilde U^0}
    +\int_{B_{ij} \cap \widetilde U^0} 
    \int_{(\omega \cap \widetilde U^0)\setminus \widetilde B_{ij}}
    \frac{\sn{D^\bfr \cE (u\eta^0) (\bfx) - D^\bfr \cE (u\eta^0) (\bfy)}^q}{\sn{\bfx-\bfy}^{\sigma q + d}}\;d\bfy\;d\bfx.
  \end{align*}
  Lemma~\ref{lem:besicovitch}, (\ref{item:lem:besicovitch-iv}) leads to 
  \begin{align*}
    \sn{D^\bfr \cE (u\eta^0)}_{\sigma,q,\omega}^q
    \lesssim\sum_{i=1}^{N_d}\sum_{j\in\N}
    \sn{D^\bfr \cE (u\eta^0) }_{\sigma,q,\widetilde B_{ij}\cap \widetilde U^0}^q
    +\eps(\bfx_{ij})^{-\sigma q}
    \vn{D^\bfr\cE (u\eta^0)}_{0,q,\widetilde B_{ij}}^q.
  \end{align*}
  Recalling that (\ref{eq:thmhpsmooth-110}) ensures that nontrivial terms in this sum correspond to pairs 
  $(i,j)$ with $\widehat B_{ij} \subset \Omega$, we may use once more 
  estimate~\eqref{lem:approx:eq2:a} of Lemma~\ref{lem:approx}. The finite overlap property of the sets 
  $\widehat B_{ij}$ then shows the
  stability estimate~\eqref{thm:hpsmooth:stab} for $r\in\R{}\setminus\N_0$ and $q\in[1,\infty)$.
  The same arguments as above apply for the parts with $\cG$ in \eqref{thm:hpsmooth:eq1}
  if we use estimate~\eqref{eq:lem:G:a} of Lemma~\ref{lem:G}.

  Finally, the estimates~\eqref{thm:hpsmooth:apx} and~\eqref{thm:hpsmooth:apx:infty}
  can be shown exactly as~\eqref{thm:hpsmooth:stab} if we use
  estimates~\eqref{lem:approx:eq2:b} or~\eqref{lem:approx:eq2:c} of
  Lemma~\ref{lem:approx}
  and estimates~\eqref{eq:lem:G:b} or~\eqref{eq:lem:G:c} of Lemma~\ref{lem:G}.
\end{proof}
\def\appendixname{}
\appendix
\newcommand{\Lpsi}{\operatorname*{Lip}(\psi)}
\section{Sobolev embedding theorems (Proof of Theorem~\ref{thm:embedding})}
\label{appendix}
The purpose of the appendix is the proof of the core of 
Theorem~\ref{thm:embedding}, that is, we show for domains that are star-shaped with
respect to a ball, that the constants in some Sobolev embedding theorems depend solely
on the ``chunkiness parameter'' and the diameter
of the domain. This can be seen by tracking the domain dependence in the proof given in 
\cite{muramatu67,muramatu67a}. For the reader's convenience, we present below the essential
steps of this proof with an emphasis on the dependence on the geometry. 
\begin{thm}
  \label{thm:sobolev-embedding}
  Let $\Omega \subset \R{d}$ be a bounded domain with $\operatorname*{diam} (\Omega) = 1$. Assume 
  $\Omega$ is star-shaped with respect to the ball $B_\rho:=B_\rho(0)$ of radius $\rho > 0$. 
  Let $0 \leq s \leq r < \infty$ and  $1 \leq p \leq q < \infty$.  Set 
  \begin{equation}
  \label{eq:mu} 
  \mu:= d \left(\frac{1}{p} - \frac{1}{q}\right). 
  \end{equation}
  Assume that one of the following two possibilities takes place:  
  \begin{enumerate}[(a)]
  \item 
  $r = s+\mu$ and $p > 1$; 
  \item
  $r > s +\mu$ and $p \ge 1$. 
  \end{enumerate}
  Then there exists $C = C(s,q,r,p,\rho,d)$ depending only on the constants indicated such that 
  $$
  |u|_{s,q,\Omega} \leq C(s,q,r,p,\rho,d) \|u\|_{r,p,\Omega}. 
  $$
\end{thm}
\begin{proof}
  The case $s = 0$ is handled in Lemmas~\ref{lemma:sobolev-embedding-Lq-left}, \ref{lemma:sobolev-embedding-Lq-left:2},
  while the case $s \in (0,1)$ can be found in 
Lemmas~\ref{lemma:sobolev-embedding-Ws-left}, \ref{lemma:sobolev-embedding-Ws-left:2}.
  For $\sn{\bft}=\lfloor s \rfloor$, these two cases imply for any derivative
  $D^\bft u$ the estimate
  $|D^\bft u|_{s - \lfloor s \rfloor,q,\Omega} \leq C \|u\|_{r,p,\Omega}$. 
\end{proof}
\begin{rem}
\begin{enumerate}
\item 
The case $p = 1$ in conjunction with $r = s +\mu$ is excluded in Theorem~\ref{thm:sobolev-embedding}. 
This is due to our method of proof. The Sobolev embedding theorem in the form given in 
\cite[Thm.~{7.38}]{adams-fournier03} suggests that Theorem~\ref{thm:sobolev-embedding} also holds 
in this case. 
\item
The star-shapedness in Theorem~\ref{thm:sobolev-embedding} is not the essential
ingredient for our control of the constants of the embedding theorems. It suffices that $\Omega$
satisfies the interior cone condition (\ref{eq:cone-condition}) with explicit control 
of the Lipschitz constant of the function  $\psi$ and the parameter $T$. That is, the impact 
of the geometry on the final estimates is captured by the Lipschitz constant $\operatorname*{Lip}(\psi)$ of $\psi$, 
the parameter $T$, and $d$. 
\end{enumerate}
\end{rem}
\begin{lem}
  \label{lemma:psi}
  Let $\Omega$ be as in Theorem~\ref{thm:sobolev-embedding}. 
  Then there exist a Lipschitz continuous function\footnote{In fact, the mapping is smooth.} 
  $\psi:\R{d} \rightarrow \R{d}$ 
  and constants $C$, $\widetilde C>0$, $T \in (0,1]$, 
  which depend solely on the chunkiness parameter $\rho$, such that the following is true: 
  \begin{enumerate}[(i)]
  \item 
  \label{item:lemma:psi-i}
  For every 
  $\bfx \in \Omega$ it holds
  \begin{equation}
    \label{eq:cone-condition}
    C_{\bfx,t}:=
    \{\bfx + t (\psi(\bfx) + \bfz)\,|\, \bfz \in B_1, \quad 0 < t < T\} \subset \Omega.
  \end{equation}
  \item 
  \label{item:lemma:psi-ii}
  $\|\psi\|_{L^\infty(\R{d})} + \|\nabla \psi\|_{L^\infty(\R{d})} \leq C$.  
  \item 
  \label{item:lemma:psi-iii}
  For every $t \in [0,T]$, the map $\Psi_t: \R{d} \rightarrow \R{d}$ given by 
  $\Psi_t(\bfx):= \bfx + t \psi(\bfx)$ is invertible and bilipschitz,
  i.e., $\Psi_t$ and its inverse $\Psi_t^{-1}:\R{d} \rightarrow \R{d}$
  are Lipschitz continuous. Furthermore, 
  $ \|\nabla \Psi_t \|_{L^\infty(\R{d})} \leq 1 + t \widetilde C$ as well as 
  $ \|\nabla \Psi_t^{-1} \|_{L^\infty(\R{d})} \leq 1 + t \widetilde C$. Additionally, 
  $\Psi_t(\Omega)+B_1 \subset \Omega$. 
  \end{enumerate}
\end{lem}
\begin{proof}
Recall that $\Omega$ is star-shaped with respect to the ball $B_{\rho}$, whose center is the origin. 
Let $\chi$ be a smooth cut-off function supported by $B_{\rho/2}$ with $\chi \equiv 1$ on $B_{\rho/4}$ 
and $0 \leq \chi \leq 1$. Let $\psi(\bfx):= L \frac{\bfx}{|\bfx|} (\chi(\bfx)-1)$, where the parameter $L > 0$ 
will be chosen sufficiently large below. 
Then $\|\psi\|_{L^\infty(\R{d})} \leq L$ (if the space $\R{d}$ is endowed with the Euclidean norm) 
and $\|\nabla \psi\|_{L^\infty(\R{d})} \leq C L$ for a constant $C>0$ that depends solely on 
the choice of $\chi$. For $\bfx \in \Omega \setminus B_{\rho}$, geometric considerations and 
$\operatorname*{diam} \Omega =1$ show that 
$\{\bfx + t (\psi(\bfx) + \bfz)\,|\, \bfz \in B_1\}$ is contained in the infinite cone with apex $\bfx$ that contains 
the ball $B_\rho$ provided that $L$ is sufficiently large, specifically, $L \sim 1/\rho$. Hence, 
by taking $T$ sufficiently small (essentially, $T \sim 1/L$) we can ensure the condition (\ref{eq:cone-condition}). 
This shows (\ref{item:lemma:psi-i}) and (\ref{item:lemma:psi-ii}). The assertion
(\ref{item:lemma:psi-iii}) follows from suitably reducing $T$: We note that for $t$ with 
$t \|\nabla \psi\|_{L^\infty(\R{d})} < 1$, the map $\Psi_t(\bfx) = \bfx + t \psi(\bfx)$ is invertible as a 
map $\R{d} \rightarrow \R{d}$ by the Banach Fixed Point Theorem. To see that $\Psi_t^{-1}$ is Lipschitz continuous,
we let $\bfy$, $\bfy^\prime \in \R{d}$ and let $\bfx$, $\bfx^\prime$ satisfy $\Psi_t(\bfx) = \bfy$, $\Psi_t(\bfx^\prime) = \bfy^\prime$. 
Then $\bfx - \bfx^\prime = \bfy - \bfy^\prime - t (\psi(\bfx) - \psi(\bfx^\prime))$ so that 
$|\bfx - \bfx^\prime| \leq | \bfy - \bfy^\prime| + t \|\nabla \psi \|_{L^\infty(\R{d})} |\bfx - \bfx^\prime|$. The assumption 
$t \|\nabla \psi\|_{L^\infty(\R{d})} < 1$ then implies the result. The argument also shows that 
$\nabla \Psi_t$ has the stated $t$-dependence.
\end{proof}
The method of proof of Theorem~\ref{thm:sobolev-embedding} relies on appropriate smoothing.
The following lemma provides two different representations of a function $u$ in terms of an
averaged version $M(u)$. These two presentations will be needed to treat both the case of
fractional and integer order Sobolev regularity. 
Let $\omega \in C^\infty_0(\R{d})$ 
with $\operatorname*{supp}(\omega) \subset B_1$ and $\int_{B_1} \omega(\bfz)\,d\bfz = 1$.
Then we have the following representation formulas:
\begin{lem} 
  \label{lemma:taylor}
  Let $\psi$ and $T$ be as in Lemma~\ref{lemma:psi}. For $u \in C^\infty(\Omega)$ and $t \in [0,T]$  define 
  \begin{equation}
  \label{eq:averaged-u}
  M(u)(t,\bfx):= \int_{\bfz \in B_1} \omega(\bfz) u(\bfx + t (\bfz + \psi(\bfx)))\,d\bfz. 
  \end{equation}
  Then we have the two representation formulas
  \begin{align*}
    u(\bfx) - M(u)(t,\bfx) = M_R(u)(t,\bfx) = M_S(u)(t,\bfx)
  \end{align*}
  for any $t \in [0,T]$, where
  \begin{align}
  \label{eq:lemma:taylor-1}
    M_R(u)(t,\bfx) &:=
    -\int_{\tau=0}^t \int_{\bfz \in B_1} \underbrace{ \omega(\bfz) (\bfz + \psi(\bfx))}_{=:\omega_1(\bfx,\bfz)}  \cdot \nabla u(\bfx 
    + \tau(\bfz + \psi(\bfx)))\,d\bfz\,d\tau,
\\
    \begin{split}\label{eq:lemma:taylor-2}
    M_S(u)(t,\bfx) &:= - \int_{\tau=0}^t \int_{\bfz \in B_1} \int_{\bfz^\prime \in B_1} \omega(\bfz^\prime) 
    \frac{\omega_2(\bfx,\bfz) }{\tau} 
\left[ u(\bfx + \tau(\bfz + \psi(\bfx))) - u(\bfx + \tau(\bfz^\prime + \psi(\bfx)))
    \right]\,d\bfz^\prime\,d\bfz\,d\tau,
    \end{split}
  \end{align}
  and the function $\omega_2$ is given by 
  \begin{equation}
  \label{eq:lemma:taylor-3}
  \omega_2 (\bfx,\bfz):= - \left( d \omega(\bfz) + \nabla \omega(\bfz) \cdot (\bfz + \psi(\bfx))\right). 
  \end{equation}
\end{lem}
\begin{proof}
  Since $M(u)(0,\bfx) = u(\bfx)$, we have
  \begin{align}\label{eq:lemma:taylor-10}
    u(\bfx) = M(u)(t,\bfx) - \int_{\tau=0}^t \partial_\tau M(u)(\tau,\bfx)\,d\tau. 
  \end{align}
  Interchanging differentiation and integration yields 
  $\partial_{\tau} M(u)(\tau,\bfx) = \int_{\bfz \in B_1} \omega(\bfz) (\bfz + \psi(\bfx)) \cdot \nabla u(\bfx + t(\bfz +\psi(\bfx)))\,d\bfz$,
  which is formula (\ref{eq:lemma:taylor-1}). 
  In order to see (\ref{eq:lemma:taylor-2}), we start again with (\ref{eq:lemma:taylor-10}). A change of 
  variables shows ($u$ is implicitly extended by zero outside $\Omega$)
  $$
  M(u)(\tau,\bfx) = \int_{\bfy \in \R{d}} \underbrace{ \tau^{-d} \omega((\bfy-\bfx)/\tau - \psi(\bfx))}_{=:\widetilde \omega(\bfy,\bfx,\tau)}  u(\bfy)\,d\bfy. 
  $$ 
  Hence, 
  $
  \partial_\tau M(u)(\tau,\bfx) = \int_{\bfy \in \R{d}} \partial_\tau \widetilde\omega(\bfy,\bfx,\tau) u(\bfy)\,d\bfy, 
  $ where $\partial_\tau \widetilde \omega$ is given explicitly by 
  \begin{align*}
  \partial_\tau \widetilde \omega(\bfy,\bfx,\tau) & = 
  - \frac{1}{\tau^{d+1}} \left\{ d \omega((\bfy-\bfz)/\tau - \psi(\bfx)) + \nabla \omega((\bfy-\bfx)/\tau-\psi(\bfx)) \cdot (\bfy-\bfx)/\tau\right\} 
\\ &
= \tau^{-(d+1)} \omega_2(\bfx,(\bfy-\bfz)/\tau-\psi(\bfx)). 
  \end{align*}
  As $M(1) \equiv 1$, we get $\partial_\tau M(1) \equiv 0$, i.e., 
  $\displaystyle\int_{\bfy \in \R{d}} \partial_\tau \widetilde\omega(\bfy,\bfx,\tau)\,d\bfy \equiv 0$.
  Therefore, for arbitrary $\bfy^\prime$
  $$
  \partial_\tau M(u)(\tau,\bfx) = \int_{\bfy \in \R{d}} \partial_\tau \widetilde\omega(\bfx,\bfy,\tau) (u(\bfy)  - u(\bfy^\prime))\,d\bfy. 
  $$
  Multiplication with $\omega((\bfy^\prime - \bfx)/\tau - \psi(\bfx))$, integration over $\bfy^\prime$, and a change of 
  variables yield
  $$
  \partial_\tau M(u)(\tau,\bfx) = \int_{\bfz^\prime \in B_1} \int_{\bfz \in B_1} \omega(\bfz^\prime) \frac{\omega_2(\bfx,\bfz)}{\tau}
  \bigl(u(\bfx + \tau (\bfz + \psi(\bfx)))  - u(\bfx + \tau(\bfz^\prime+\psi(\bfx)))\bigr)\,d\bfz\,d\bfz^\prime. 
  $$
  Inserting this  in 
  (\ref{eq:lemma:taylor-10}) then gives the representation (\ref{eq:lemma:taylor-2}). 
\end{proof}
\begin{rem} 
Higher order representation formulas are possible, see, e.g., \cite{muramatu67}.  
\eremk
\end{rem}
\subsection{The case of integer $s$ in Theorem~\ref{thm:sobolev-embedding}}
The limiting case $p=1$ is special. In the results below, it will appear often in conjunction
with a parameter $\sigma \ge 0$. In the interest of brevity, we formulate a condition that we 
will require repeatedly in the sequel: 
\begin{equation}
\label{eq:condition-on-p-sigma}
\left( \tilde\sigma = 0 \quad \mbox{ and } \quad \tilde p > 1\right) 
\qquad \mbox{ or } \qquad 
\left( \tilde\sigma > 0 \quad \mbox{ and } \quad \tilde p \ge 1\right). 
\end{equation}
For a ball $B_t\subset \R{d}$ of radius $t > 0$, we write $|B_t| \sim t^d$ 
for its (Lebesgue) measure. 
\begin{lem}
  \label{lemma:estimate-U}
  Let $1 \leq \tilde p \leq \tilde q < \infty$ and $u \in L^{\tilde p}(\Omega)$. 
  Let $\Omega$, $T$, and $\psi$ be as in Lemma~\ref{lemma:psi}. 
  Set $\tilde\mu = d(\tilde p^{-1} - \tilde q^{-1})$. 
  Define, for $t \in (0,T)$, the function 
  \begin{equation}\label{eq:lemma:estimate-U-0}
    U(\bfx,t):= \int_{\bfz \in B_1} |u(\bfx + t (\bfz + \psi(\bfx)))|\,d\bfz. 
  \end{equation}
  Then the following two estimates hold: 
  \begin{enumerate}[(i)]
  \item There exists $C_1 = C_1(\tilde p,\tilde q,\Lpsi,d)>0$ depending only on the
    quantities indicated such that 
    \begin{eqnarray}\label{eq:lemma:estimate-U-1}
      \|U(t,\cdot)\|_{0,\tilde q,\Omega} &\leq& C_1 t^{-\tilde\mu}
      \|u\|_{0,\tilde p,\Omega}. 
    \end{eqnarray}
  \item 
  If the pair $(\tilde\sigma,\tilde p)$ satisfies (\ref{eq:condition-on-p-sigma}), 
  then there exists $C_2 = C_2(\tilde p,\tilde q,\Lpsi,\tilde\sigma,d)>0$ depending only on the quantities indicated such that 
  \begin{eqnarray}
    \label{eq:lemma:estimate-U-2}
    \left\|\int_{t=0}^T t^{-1+\tilde\mu+\tilde\sigma}
    U(t,\cdot)\,dt\right\|_{0,\tilde q,\Omega} &\leq& 
    C_2 T^{\tilde\sigma} \left[ T^{\tilde\mu-d} \|u\|_{0,1,\Omega} +
      \|u\|_{0,\tilde p,\Omega}\right]. 
  \end{eqnarray}
  \end{enumerate}
\end{lem}
\begin{proof}
  {\sl Proof of (\ref{eq:lemma:estimate-U-1}):} 
  We will use the elementary estimate 
  \begin{align}\label{eq:lemma:estimate-U-10}
    \|v\|_{0,\tilde q,\Omega} \leq 
    \|v\|_{0,\tilde p,\Omega}^{\tilde p/\tilde q} \|v\|_{0,\infty,\Omega}^{1-\tilde p/\tilde q} 
    \qquad \mbox{ for $1 \leq \tilde p \leq \tilde q < \infty$ and $v \in L^\infty(\Omega)$.}
  \end{align}
  We start with $\|U(t,\cdot)\|_{0,\infty,\Omega}$.
  Letting $\tilde p^\prime = \tilde p/(\tilde p-1)$ be the conjugate
  exponent of $\tilde p$, we compute: 
  \begin{align}\label{eq:lemma:estimate-U-20}
    \begin{split}
      |U(t,\bfx)| &=  \int_{\bfz \in B_1} |u(\bfx + t(\bfz + \psi(\bfx)))|\,d\bfz 
           \leq |B_1|^{1/\tilde p^\prime} \left(\int_{\bfz \in B_1}
	   |u(\bfx + t (\bfz + \psi(\bfx)))|^{\tilde p}\,d\bfz\right)^{1/\tilde p} \\
           &= |B_1|^{1/\tilde p^\prime} t^{-d/\tilde p}\left(\int_{\bfz \in t B_1}
	   |u(\bfx + \bfz + t \psi(\bfx))|^{\tilde p}\,d\bfz\right)^{1/\tilde p} 
	   \leq |B_1|^{1/\tilde p^\prime} t^{-d/\tilde p} \|u\|_{0,\tilde p,\Omega}.
    \end{split}
  \end{align}
  For $\|U(t,\cdot)\|_{0,\tilde p,\Omega}$, we use the Minkowski inequality (\ref{eq:minkowski}) and 
  the change of variables formula (cf., e.g., \cite[Thm.~2, Sec.~{3.3.3}]{eva1} for the case
  of bi-Lipschitz changes of variables although in the present case, the mapping is smooth!) to compute
  \begin{align}\label{eq:lemma:estimate-U-30}
    \begin{split}
      \|U(t,\cdot)\|_{0,\tilde p,\Omega} 
      &=\left(\int_{\bfx \in \Omega} \left| \int_{\bfz \in B_1}
      |u(\Psi_t(\bfx) + t \bfz)|\,d\bfz\right|^{\tilde p}\,d\bfx\right)^{1/\tilde p} 
\stackrel{(\ref{eq:minkowski})}{\leq}
      \int_{\bfz \in B_1} \left( \int_{\bfx \in \Omega}
      |u(\Psi_t(\bfx) + t \bfz)|^{\tilde p}\,d\bfx\right)^{1/\tilde p}\,d\bfz \\
      &\leq \int_{\bfz \in B_1} \left( 
      \int_{\bfx \in \Omega} |u(\bfx)|^{\tilde p} \|\operatorname*{det}
      D\Psi_t^{-1}\|_{L^\infty(\R{d})}  \,d\bfx \right)^{1/\tilde p}\,d\bfz 
      \leq C \|u\|_{0,\tilde p,\Omega}, 
    \end{split}
  \end{align}  
  where the constant $C$ depends only on $T$ and $\Lpsi$ (cf.~Lemma~\ref{lemma:psi}
  for the $t$-dependence
  of $\Psi_t$).
  Inserting (\ref{eq:lemma:estimate-U-20}) and (\ref{eq:lemma:estimate-U-30})
  in (\ref{eq:lemma:estimate-U-10}) 
  yields (\ref{eq:lemma:estimate-U-1}). 
  
  {\sl Proof of (\ref{eq:lemma:estimate-U-2}) for the case $\tilde \sigma > 0 $ together
  with $\tilde p \ge 1$:} 
  This is a simple consequence of (\ref{eq:lemma:estimate-U-1}). 
  
  {\sl Proof of (\ref{eq:lemma:estimate-U-2}) for the case $\tilde\sigma = 0$
  together with $\tilde p > 1$:} 
  From Lemma~\ref{lemma:psi} we have $\|\psi\|_{L^\infty(\R{d})} \leq C $. Hence, 
  $ B_1 + \psi(\bfx) \subset B_{1+ C}$ uniformly in $\bfx$. If we implicitly assume that
  $u$ is extended by zero outside of $\Omega$, we can estimate 
  \begin{align*}
    |U(t,\bfx)| &=\int_{\bfz \in B_1} |u(\bfx + t \bfz + t \psi(\bfx))|\,d\bfz = \int_{\bfz \in B_1 + \psi(\bfx)} |u(\bfx + t \bfz)|\,d\bfz 
    \leq \int_{\bfz \in B_{1+C}} |u(\bfx + t \bfz)|\,d\bfz.
  \end{align*}
  Our goal is to show that the function
  $$
    \bfx \mapsto \int_{t=0}^T t^{-1+\tilde \mu} \int_{\bfz \in B_1}
    u(\bfx + t (\bfz + \psi(\bfx)))\,d\bfz
  $$
  is in $L^{\tilde q}(\Omega)$ provided that $u\in L^{\tilde p}(\Omega)$.
  This is shown with Lemma~\ref{lemma:marcinkiewicz} below. To that end, 
  we assume that $u$ is extended by zero outside of $\Omega$ and bound the $L^{\tilde q}(\R{d})$-norm of 
  $$
    \bfx \mapsto \int_{t=0}^T t^{-1+\tilde\mu} \int_{\bfz \in B_{1+C}} |u(\bfx + t\bfz)|\,d\bfz. 
  $$
  Lemma~\ref{lemma:hardy-style} is applicable with $s = 1-\tilde \mu$, since $\tilde p > 1$ 
  implies $\tilde \mu \in (0,d)$. From Lemma~\ref{lemma:hardy-style} (and a density argument to be able 
  to work with $|u|$ instead of $u$) we get 
  \begin{align}\label{eq:lemma:estimate-U-100}
    \begin{split}
      &
      \int_{t=0}^T t^{-1+\tilde\mu} \int_{\bfz \in B_{1+C}} |u(\bfx + t\bfz)|\,d\bfz =  
\\ & 
\frac{1}{d-\tilde\mu} \left[ 
(1+C)^{d-\tilde\mu}
      \int_{\bfz \in B_{(1+C)T}} |\bfz|^{\tilde\mu-d} |u(\bfx + \bfz)|\,d\bfz 
- T^{\tilde\mu-d} \int_{\bfz \in B_{(1+C)T}} |u(\bfx + \bfz)|\,d\bfz 
\right].
    \end{split}
  \end{align} 
  The second contribution in (\ref{eq:lemma:estimate-U-100}) is estimated directly. 
  For the first contribution, we obtain from Lemma~\ref{lemma:marcinkiewicz} 
  (with $\lambda = d-\tilde\mu$, $r^{-1} = 1 + \frac{\tilde\mu}{d} - \tilde p^{-1}$ so that 
  $1 - r^{-1} =  \tilde p^{-1} - \frac{\tilde \mu}{d} = \tilde q^{-1}$) 
  $$
  \left\| \int_{\bfz \in B_{(1+C)T}} |\bfz|^{\tilde\mu-d} |u(\cdot +
  \bfz)|\,d\bfz \right\|_{0,\tilde q,\R{d}} 
  \lesssim \|u\|_{0,\tilde p,\R{d}}. 
\qedhere
  $$
\end{proof}
\begin{lem}
  \label{lemma:estimate-difference-U}
  Let $\Omega$, $T$, and $\psi$ be as in Lemma~\ref{lemma:psi}. 
  Let $1 \leq \tilde p \leq \tilde q < \infty$ and define
  $\tilde\mu:= d (\tilde p^{-1} - \tilde q^{-1})$ as in (\ref{eq:mu}).
  Let $\tilde s \in (0,1)$ and $u \in W^{\tilde s,\tilde p}(\Omega)$.  Define, for $t \in (0,T)$, the function 
  \begin{equation}
  V(t,\bfx):= \int_{\bfz \in B_1}\int_{\bfz^\prime \in B_1}
  |u(\bfx + t (\bfz + \psi(\bfx))) - u(\bfx + t(\bfz^\prime+\psi(\bfx)))|\,d\bfz^\prime \,d\bfz.
  \end{equation}
  Then, the following two assertions hold true: 
  \begin{enumerate}[(i)]
  \item 
  There exists $C_1 = C_1(\tilde p,\tilde q,\tilde s,d,\Lpsi,T)$ such that  
  \begin{eqnarray}
  \label{eq:lemma:estimate-difference-U-20}
  \|V(t,\cdot)\|_{0,\tilde q,\Omega} \leq C_1 t^{\tilde s-\tilde\mu}
  |u|_{\tilde s,\tilde p,\Omega} \qquad \forall t \in (0,T). 
  \end{eqnarray}
  \item 
  If the pair $(\tilde\sigma,\tilde p)$ satisfies (\ref{eq:condition-on-p-sigma}), then 
  there is a constant $C_2 = C_2(\tilde p,\tilde q,\tilde s,d,\Lpsi,\sigma,T)$ such that 
  \begin{eqnarray}
  \label{eq:lemma:estimate-difference-U-2}
  \left\|\int_{t=0}^T t^{-1+\tilde\mu-\tilde s+\tilde\sigma }
  V(t,\cdot)\,dt\right\|_{0,\tilde q,\Omega} &\leq& C_2 |u|_{\tilde s,\tilde p,\Omega}. 
  \end{eqnarray}
  \end{enumerate}
\end{lem}
\begin{proof}
  The proof is structurally similar to that of Lemma~\ref{lemma:estimate-U}. 
  Define 
  $$
  v(\bfy,\bfy^\prime) = \frac{|u(\bfy) - u(\bfy^\prime)|}{|\bfy - \bfy^\prime|^{\tilde s+d/\tilde p}}. 
  $$
  Letting $\tilde p^\prime = \tilde p/(\tilde p-1)$ be the conjugate exponent of $\tilde p$, we compute 
  \begin{align}\label{eq:lemma:estimate-difference-U-2000} 
    \begin{split}
      V(t,\bfx) &= 
      \int_{\bfz \in B_1} \int_{\bfz^\prime \in B_1} 
      \frac{|u(\bfx+ t(\bfz + \psi(\bfx))) - u(\bfx+t(\bfz^\prime+ \psi(\bfx)))|}
           {|\bfx+ t (\bfz + \psi(\bfx)) - (\bfx + t(\bfz^\prime+\psi(\bfx)))|^{\tilde s+d/\tilde p}}
           (t|\bfz - \bfz^\prime|)^{\tilde s+d/\tilde p}\,d\bfz^\prime\,d\bfz  \\
      &\leq (2t)^{\tilde s+d/\tilde p} \int_{\bfz \in B_1}
      \int_{\bfz^\prime \in B_1} v(\bfx + t(\bfz+\psi(\bfx)),\bfx+t(\bfz^\prime+\psi(\bfx)))\,d\bfz^\prime\,d\bfz \\
      &\leq (2t)^{\tilde s+d/\tilde p}  t^{-d} |B_t|^{1/\tilde p^\prime} \int_{\bfz \in B_1} 
      \|v(\bfx + t (\bfz + \psi(\bfx)),\cdot)\|_{0,\tilde p,\Omega}\,d\bfz, \\
      &\leq C t^{\tilde s} \int_{\bfz \in B_1} \|v(\bfx + t (\bfz +
      \psi(\bfx)),\cdot)\|_{0,\tilde p,\Omega} \,d\bfz. 
    \end{split}
  \end{align}
  We recognize that, up to the factor $t^{\tilde s}$, the right-hand side is a function of the form 
  studied in Lemma~\ref{lemma:estimate-U}. 
  The bounds 
  (\ref{eq:lemma:estimate-difference-U-20}) and 
  (\ref{eq:lemma:estimate-difference-U-2}) therefore follow from Lemma~\ref{lemma:estimate-U}. 
\end{proof}
\begin{lem}
  \label{lemma:sobolev-embedding-Lq-left}
  Let $\Omega$ be as in Theorem~\ref{thm:sobolev-embedding}. Assume $r \ge 0$, $1 \leq p \leq q < \infty$ 
  and set $\mu:= d (p^{-1} - q^{-1})$ as in (\ref{eq:mu}).
  Assume that one of the following two cases occurs: 
  \begin{enumerate}[(a)]
    \item 
      $\displaystyle r = \mu$ in conjunction with $p > 1$;  
    \item 
      $\displaystyle r > \mu\notin\N$ in conjunction with $p \ge 1$.  
  \end{enumerate}
  Then there is a constant $C = C(p,q,r,\Lpsi,T,d)$, which depends solely on the quantities indicated (and the 
  assumption that $\operatorname*{diam}\Omega \leq 1$), such that  
  $$
  \|u\|_{L^q(\Omega)} \leq C \|u\|_{W^{r,p}(\Omega)}. 
  $$
\end{lem}
\begin{proof}
  We can assume $\mu>0$, which implies $r>0$ and $q>1$. We start with some preliminaries.
  For $L\in\N$ with
  \begin{equation}
    \label{eq:lemma:sobolev-embedding-Lq-left-5}
    \frac{L}{d} \leq 1 - \frac{1}{q}. 
  \end{equation}
  define recursively the values $p_0$, $p_1,\ldots,p_L \in [1,\infty)$ by 
  \begin{equation}
    \label{eq:lemma:sobolev-embedding-Lq-left-7}
    \frac{1}{p_0}:= \frac{1}{q}, 
    \qquad \frac{1}{p_{i}} =: \frac{1}{p_{i-1}} + \frac{1}{d}, \qquad i=1,\ldots,L. 
  \end{equation}
  Note that indeed $p_i \ge 1$ since~\eqref{eq:lemma:sobolev-embedding-Lq-left-5} implies $p_L \ge 1$ in view of 
  $p_L^{-1} = q^{-1} + L d^{-1}$. Furthermore, we have the stronger assertion
  \begin{equation}
    \label{eq:lemma:sobolev-embedding-Lq-left-8}
    1 < p_i, \qquad i=0,\ldots,L-1. 
  \end{equation}
  We claim that 
  \begin{equation}
    \label{eq:lemma:sobolev-embedding-Lq-left-10}
    \begin{cases} 
    \|u\|_{0,p_{i-1},\Omega} \leq C \|u\|_{1,p_i,\Omega}, 
    & i=1,\ldots,L-1,  \\
    \|u\|_{0,p_{L-1},\Omega} \leq C \|u\|_{1,p_L,\Omega}, 
    & \mbox{ if $p_L > 1$.}
    \end{cases} 
  \end{equation}
  which implies in particular 
  \begin{equation}
    \label{eq:lemma:sobolev-embedding-Lq-left-20}
    \|u\|_{0,q,\Omega} \leq C \|u\|_{L,p_L,\Omega} 
    \quad \mbox{ if $p_L > 1$.}
  \end{equation}
  In order to see (\ref{eq:lemma:sobolev-embedding-Lq-left-10}) we use 
  the representation $u(\bfx) = M(u)(T,\bfx) + M_R(u)(T,\bfx)$ from
  Lemma~\ref{lemma:taylor}
  and the fact that $\omega$ is fixed and bounded and that we have control over $\psi$ and $\nabla \psi$
  by Lemma~\ref{lemma:psi} to get
  \begin{equation}
    \label{eq:lemma:sobolev-embedding-Lq-left-100}
    |u(\bfx)| \leq C \left[ \int_{\bfz \in B_1} |u(\bfx + T (\bfz + \psi(\bfx)))|\,d\bfz 
    + \int_{t=0}^T \int_{\bfz \in B_1} |\nabla u(\bfx + t(\bfz+\psi(\bfx)))|\,d\bfz\,dt\right].
  \end{equation}
  For $i=1,\ldots,L-1$, the first term in~\eqref{eq:lemma:sobolev-embedding-Lq-left-100}
  is bounded by (\ref{eq:lemma:estimate-U-1}).
  The second term in~\eqref{eq:lemma:sobolev-embedding-Lq-left-100}
  is bounded with (\ref{eq:lemma:estimate-U-2}),
  where we choose $\tilde\mu = d(p_{i}^{-1} - p_{i-1}^{-1}) = 1$ and $\tilde\sigma=0$, which
  we may due to (\ref{eq:lemma:sobolev-embedding-Lq-left-8}). This gives
  \begin{align*}
    \|u\|_{0,p_{i-1},\Omega} \leq C \left[ T^{-1} \|u\|_{0,p_i,\Omega}  +
      \|\nabla u\|_{0,p_i,\Omega}\right]
    \leq C \|u\|_{1,p_i,\Omega},
  \end{align*}
  which is the first part of (\ref{eq:lemma:sobolev-embedding-Lq-left-10}). 
  The case $i = L$ in (\ref{eq:lemma:sobolev-embedding-Lq-left-10}) follows in the same way 
  if $p_L > 1$. Using~\eqref{eq:lemma:sobolev-embedding-Lq-left-20} we can show the lemma 
  in the following way:
  \begin{enumerate}[(i)]
    \item The case $r = \mu\in\N$.\\ 
      Note that the choice $L = r$ satisfies 
      (\ref{eq:lemma:sobolev-embedding-Lq-left-5}) and that $p_L = p>1$. 
      Therefore, (\ref{eq:lemma:sobolev-embedding-Lq-left-20}) is the desired estimate.
    \item The case $\mu\notin\N$ and $\mu \leq r < \lceil \mu\rceil$.\\
      We claim $\vn{u}_{0,q,\Om}\lesssim\vn{u}_{\lfloor\mu\rfloor,p_{\lfloor\mu\rfloor},\Om}$.
      To see this, we observe for the case $\lfloor \mu \rfloor \ge 1$ that 
      (\ref{eq:lemma:sobolev-embedding-Lq-left-5}) holds with 
      $L = \lfloor \mu \rfloor$
      and therefore $p_{\lfloor\mu\rfloor}>p\geq1$ together with~\eqref{eq:lemma:sobolev-embedding-Lq-left-20}
      implies $\vn{u}_{0,q,\Om}\lesssim\vn{u}_{\lfloor\mu\rfloor,p_{\lfloor\mu\rfloor},\Om}$. 
      In the case $\lfloor \mu \rfloor = 0$, the assertion is trivial.
      
      For $\sn{\bft} = \lfloor\mu\rfloor$ we write
      $D^\bft u(\bfx) = M(D^\bft u)(T,\bfx) + M_S(D^\bft u)(T,\bfx)$ and use
      Lemma~\ref{lemma:estimate-difference-U},
      where we set $\tilde q = p_{\lfloor\mu\rfloor}$, $\tilde p = p$, 
      and $\tilde s = r-\lfloor\mu\rfloor$.
      Since $\lfloor\mu\rfloor = d(p_{\lfloor\mu\rfloor}^{-1}-q^{-1})$, we see
      $\tilde\mu - \tilde s = \mu-r$, so that for $\mu=r$ we choose $\tilde\sigma=0$ and hence
      require $p>1$, while for $\mu<r$ we can choose $\tilde\sigma>0$ and hence
      $p\geq1$. This shows
      $\sn{u}_{\lfloor\mu\rfloor,p_{\lfloor\mu\rfloor},\Om}
      \lesssim \sn{u}_{r,p,\Om}$.
  
      For $\sn{\bft} \leq \lfloor\mu\rfloor-1$ (given $\mu>1$) we write
      $D^\bft u(\bfx) = M(D^\bft u)(T,\bfx) + M_R(D^\bft u)(T,\bfx)$ and use
      Lemma~\ref{lemma:estimate-U}, where we set
      $\tilde q = p_{\lfloor\mu\rfloor}$, $\tilde p = p$.
      Since then $\tilde \mu = \mu - \lfloor\mu\rfloor<1$, we obtain for
      any $p\geq1$ that
      $\sn{u}_{\sn{\bft},p_{\lfloor\mu\rfloor},\Om}
      \lesssim \sn{u}_{\sn{\bft}+1,p,\Om}$.
    \item The case $\mu\notin\N$ and $\lceil \mu \rceil \leq r$.\\
      It suffices to consider $r = \lceil\mu\rceil$.
      As $p_{\lfloor\mu\rfloor}>p\geq1$, we obtain with~\eqref{eq:lemma:sobolev-embedding-Lq-left-20}
      the bound $\vn{u}_{0,q,\Om}\lesssim\vn{u}_{\lfloor\mu\rfloor,p_{\lfloor\mu\rfloor},\Om}$.
  
      For $\sn{\bft} \leq \lfloor\mu\rfloor$ we write
      $D^\bft u(\bfx) = M(D^\bft u)(T,\bfx) + M_R(D^\bft u)(T,\bfx)$ and use
      Lemma~\ref{lemma:estimate-U}, where we set
      $\tilde q = p_{\lfloor\mu\rfloor}$, $\tilde p = p$.
      Since then $\tilde \mu = \mu - \lfloor\mu\rfloor<1$, we obtain for
      any $p\geq1$ that
      $\sn{u}_{\sn{\bft},p_{\lfloor\mu\rfloor},\Om}
      \lesssim \sn{u}_{\sn{\bft}+1,p,\Om}$.
      In total, this yields $\vn{u}_{0,q,\Om}\leq\vn{u}_{\lceil\mu\rceil,p,\Om}$.
\qedhere
  \end{enumerate}
\end{proof}
\subsection{The case of fractional $s$ in Theorem~\ref{thm:sobolev-embedding}}
The analog of Lemma~\ref{lemma:estimate-U} is the following result. 
\begin{lem}
  \label{lemma:estimate-U-fractional} 
  Let $\Omega$, $T$, $\psi$ be as in Lemma~\ref{lemma:psi}.
  Let $1 \leq \tilde p \leq \tilde q < \infty$. Set $\tilde\mu = d(\tilde p^{-1} - \tilde q^{-1})$. 
  Let $K = K(\bfx,\bfz)$ be defined on $\Omega \times \R{d}$ with $\operatorname*{supp} K(\bfx,\cdot) \subset B_1$
  for every $\bfx \in \Omega$. Let $K$ be bounded (bound $\|K\|_{L^\infty}$) 
  and Lipschitz continuous
  with Lipschitz constant $\operatorname*{Lip}(K)$. 
  Define the function 
  $$
  V(t,\bfx):= \int_{\bfz \in B_1} K(\bfx,\bfz) u(\bfx + t (\bfz + \psi(\bfx)))\,d\bfz. 
  $$
  \begin{enumerate}[(i)]
  \item 
  \label{item:lemma:estimate-U-fractional-i} 
  Let $\tilde s \in (0,1)$.  
  Then there exists $C_1 = C_1(\tilde p,\tilde q,\tilde s,d,\Lpsi,T,\|K\|_{L^\infty},
  \operatorname*{Lip}(K))$ such that  
  $$
  |V(t,\cdot)|_{\tilde s,\tilde q,\Omega} \leq C_1 t^{-\tilde\mu-\tilde s}
  \|u\|_{0,\tilde p,\Omega}. 
  $$
  \item
  \label{item:lemma:estimate-U-fractional-ii} 
  Let $\tilde s \in (0,1)$ and assume that the pair $(\tilde\sigma,\tilde p)$
  satisfies (\ref{eq:condition-on-p-sigma}). 
  Then 
  $$
  \left|\int_{t=0}^T t^{-1+\tilde s+\tilde\mu+\tilde\sigma}
  V(t,\cdot)\right|_{\tilde s,\tilde q,\Omega}
  \leq C_2 \|u\|_{0,\tilde p,\Omega}  
  $$
for a constant $C_2 = C_2(\tilde p,\tilde q,\tilde s,d,\Lpsi,T,\|K\|_{L^\infty},\operatorname*{Lip}(K),\sigma)$ 
depending only on the quantities indicated.
  \end{enumerate}
\end{lem}
\begin{proof} 
  {\em Proof of (\ref{item:lemma:estimate-U-fractional-i}):} 
  Let $\bfx$, $\bfy \in \Omega$ and $t \in (0,T)$. 
  Define the translation ${\mathbf t}:= \frac{\bfx-\bfy}{t} + \psi(\bfx) - \psi(\bfy)$
  and denote by $\chi_A$ the characteristic function of a set $A$. 
  An affine change of variables gives for $\bfx$, $\bfy \in \Omega$ and $t \in(0,T)$
  \begin{align*}
  V(t,\bfx) &= \int_{\bfz \in B_1} K(\bfx,\bfz) u(\bfx + t (\bfz + \psi(\bfx)))\,d\bfz 
  = \int_{\bfz \in \R{d}} K(\bfx,\bfz) u(\bfx + t (\bfz + \psi(\bfx)))\chi_{B_1}(\bfz)\,d\bfz \\
  & = \int_{\bfz^\prime \in \R{d}}
  K(\bfx,\bfz^\prime - {\mathbf t} ) u(\bfy + t(\bfz^\prime + \psi(\bfy)))\chi_{B_1+{\mathbf t}}(\bfz^\prime)\,d\bfz^\prime\\
  &= V(t,\bfy) + \int_{\bfz \in \R{d}}
  \underbrace{ \left[ K(\bfx,\bfz - {\mathbf t})
  \chi_{B_1 + {\mathbf t}}(\bfz) - K(\bfy,\bfz) \chi_{B_1}(\bfz)\right]}_{=:B(\bfx,\bfy,\bfz)} 
  u(\bfy + t (\bfz + \psi(\bfy)))\,d\bfz. 
  \end{align*}
  We estimate the function $B$. We have the obvious estimate 
  $|B(\bfx,\bfy,\cdot)| \leq \|K\|_{L^\infty} \left[ \chi_{B_1} + \chi_{B_1+{\mathbf t}}\right]$.
  For further estimates, we start by noting 
  \begin{equation}
  \label{eq:estimate-translation-vector}
  |{\mathbf t}| \leq  |\bfx-\bfy|/t + \|\nabla \psi\|_{L^\infty} |\bfx - \bfy| \leq C |\bfx - \bfy|/t, 
  \end{equation}
  where $C$ depends on the Lipschitz constant of $\psi$ and on $T$. We also note 
  \begin{align*}
  \bfz \in B_1 \setminus (B_1 + {\mathbf t}) & \quad \Longrightarrow 
  \left( |\bfz| \leq 1 \quad \wedge \quad |\bfz- {\mathbf t}| \ge 1\right) \quad \Longrightarrow 
  1 - |{\mathbf t}| \leq |\bfz|  \leq 1, \\
  \bfz \in ({\mathbf t} + B_1) \setminus B_1 & \quad \Longrightarrow 
  \left( |\bfz| \ge 1 \quad \wedge \quad |\bfz - {\mathbf t} | \leq 1\right) 
  \quad \Longrightarrow 1 - |{\mathbf t}| \leq |\bfz - {\mathbf t}| \leq 1. 
  \end{align*}
  Since $K$ is Lipschitz continuous on
  $\Omega \times \R{d}$ and $\operatorname*{supp} K(\bfx,\cdot) \subset B_1$
  we get $|K(\bfx,\bfz)| \leq C \sn{{\mathbf t}}$ for every 
  $\bfz \in R:= (B_1 \setminus (B_1 + {\mathbf t})) \cup ((B_1 + {\mathbf t})\setminus B_1)$,
  where the constant 
  $C$ depends only on the Lipschitz constant of $K$. We therefore get 
  $|B(\bfx,\bfy,\bfz)|  \leq C \sn{{\mathbf t}}$ for $\bfz \in R$. For the case 
  $\bfz  \in B_1 \cap (B_1 + {\mathbf t})$, we get from the Lipschitz continuity of $K$ that 
  $|B(\bfx,\bfy,\bfz)| \leq C \left[|\bfx - \bfy| + |\bfx - \bfy|/t\right] \leq C |\bfx - \bfy|/t$.
  Putting together the above estimates for $B$, we arrive at 
  $$
  |B(\bfx,\bfy,\bfz) | \leq C \min\{1,|\bfx - \bfy|/t\} \left[ \chi_{B_1}(\bfz) + \chi_{{\mathbf t} + B_1}(\bfz)\right].
  $$
  In total, we get 
  \begin{align*}
  |V(t,\bfx) - V(t,\bfy)| &\leq C \min\{1, |\bfx - \bfy|/t\}
  \int_{B_1 \cup {\mathbf t} + B_1} |u(\bfy + t (\bfz + \psi(\bfy)))|\,d\bfz\\
   &= C \min\{1,|\bfx - \bfy|/t\} \left[ \int_{\bfz \in B_1} |u(\bfx + t (\bfx + \psi(\bfx)))|\,d\bfz 
                                + \int_{\bfz \in B_1} |u(\bfy + t (\bfy + \psi(\bfy)))|\,d\bfz\right] \\
  & \leq C \min\{1,|\bfx-\bfy|/t\} \left[ U(t,\bfx) + U(t,\bfy)\right],
  \end{align*}
  where, in the last step we have inserted the definition of the function
  $U$ from (\ref{eq:lemma:estimate-U-0}). Therefore, 
  \begin{align}\label{eq:lemma:estimate-U-fractional-100}
    \frac{|V(t,\bfx) - V(t,\bfy)|}{|\bfx - \bfy|^{\tilde s+d/\tilde q}}
    \leq C \frac{\min\{1,|\bfx - \bfy|/t\}}{|\bfx - \bfy|^{\tilde s+d/\tilde q}} 
    \left[ U(t,\bfx) + U(t,\bfy)\right].
  \end{align}
  In view of the symmetry in the variables $\bfx$ and $\bfy$, we will only consider one type of integral. 
  We compute 
  \begin{align}\label{eq:lemma:estimate-U-fractional-200}
    \begin{split}
      |U(t,\bfx)|^{\tilde q} &\int_{\bfy \in \Omega} \frac{\left(\min\{1,|\bfx - \bfy|/t\}\right)^{\tilde q}}
      {|\bfx - \bfy|^{\tilde s\tilde q+d}} \,d\bfy
\\&
\lesssim
      |U(t,\bfx)|^{\tilde q} \left[ \int_{r=0}^t \frac{\left(r/t\right)^{\tilde q} }
	{r^{\tilde s\tilde q+d}}r^{d-1} \,dr
      + \int_{r=t}^\infty r^{-\tilde s\tilde q-d} r^{d-1}\,dr 
      \right]
      \lesssim t^{-\tilde s\tilde q} |U(t,\bfx)|^{\tilde q},
    \end{split}
  \end{align}
  where the hidden constants depend only on $\tilde s$ and $\tilde q$. We conclude 
  \begin{align*} 
    |V(t,\cdot)|^{\tilde q}_{\tilde s,\tilde q,\Omega} 
    \leq C t^{-\tilde s\tilde q} \|U(t,\cdot)\|^{\tilde q}_{0,\tilde q,\Omega} 
    \leq C t^{-\tilde s\tilde q - \tilde\mu \tilde q} \|u\|^{\tilde
    q}_{0,\tilde p,\Omega}, 
  \end{align*}
  where the last step follows from (\ref{eq:lemma:estimate-U-1}).
  
  {\em Proof of (\ref{item:lemma:estimate-U-fractional-ii}):} 
  Starting  from (\ref{eq:lemma:estimate-U-fractional-100}) we have to estimate 
  \begin{align*}
    I& := \int_{\bfx \in \Omega} \int_{\bfy \in \Omega} 
      \frac{1}{|\bfx - \bfy|^{\tilde s\tilde q+d}} \left| \int_{t=0}^T t^{-1 +\tilde\mu + \tilde\sigma}
      \min\{1,|\bfx-\bfy|/t\} U(t,\bfx)\,dt
      \right|^{\tilde q}\,d\bfy\,d\bfx.
  \end{align*}
  Applying the Minkowski inequality (\ref{eq:minkowski}), we obtain 
  (recalling the calculation performed in (\ref{eq:lemma:estimate-U-fractional-200}))
  \begin{align*}
    I &\leq \int_{\bfx \in \Omega} 
    \left\{ 
      \int_{t=0}^T \left( \int_{\bfy \in \Omega} 
                           \frac{1}{|\bfx - \bfy|^{\tilde s\tilde q+d}}
			   t^{(-1 +\tilde\mu + \tilde\sigma) \tilde q}
			   \left(\min\{1,|\bfx-\bfy|/t\}\right)^{\tilde q}  |U(t,\bfx)|^{\tilde q}\,d\bfy
                   \right)^{1/\tilde q}\,dt
		 \right\}^{\tilde q}\,d\bfx \\
    &\lesssim \int_{\bfx \in \Omega} 
    \left\{ 
      \int_{t=0}^T |U(t,\bfx)| t^{-1 + \tilde\mu + \tilde\sigma} \,dt
    \right\}^{\tilde q}\,d\bfx 
    \lesssim  \|u\|_{0,\tilde p,\Omega}^{\tilde q},
  \end{align*}
  where, in the last step, we used (\ref{eq:lemma:estimate-U-2}). 
\end{proof}
The analog of Lemma~\ref{lemma:estimate-difference-U} is as follows: 
\begin{lem}
  \label{lemma:estimate-difference-U-fractional}
  Let $\Omega$, $T$, $\psi$ be as in Lemma~\ref{lemma:psi}. 
  Let $K  = K(\bfx,\bfz,\bfz^\prime)$ be defined on $\Omega \times \R{d} \times \R{d}$ with 
  $\operatorname*{supp} K(\bfx,\cdot,\cdot) \subset B_1 \times B_1$ for every $\bfx \in \Omega$. 
  Let $K$ be bounded (bound $\|K\|_{L^\infty}$) and Lipschitz continuous with Lipschitz constant 
  $\operatorname*{Lip}(K)$. 
  
  Let $\tilde s$, $\tilde r \in (0,1)$. Let $1 \leq \tilde p \leq \tilde q < \infty$.
  Set $\tilde\mu = d (\tilde p^{-1} - \tilde q^{-1})$.
  For $u \in W^{\tilde r,\tilde p}(\Omega)$ define 
  for $t \in (0,T)$, the function
  \begin{equation}
  V(t,\bfx):= \int_{\bfz \in B_1}\int_{\bfz^\prime \in B_1}  K(\bfx,\bfz,\bfz^\prime) 
  \left[ u(\bfx + t (\bfz + \psi(\bfx))) - u(\bfx + t(\bfz^\prime+\psi(\bfx)))\right]\,d\bfz^\prime \,d\bfz. 
  \end{equation}
  Then: 
  \begin{enumerate}[(i)]
  \item
  \label{item:lemma:estimate-difference-U-fractional-i}
  There exists $C_1 = C_1(\tilde p,\tilde q,\tilde r,\tilde s,T,\Lpsi,d,\|K\|_{L^\infty},\operatorname*{Lip}(K))$ such that 
  $$
  |V(t,\cdot)|_{\tilde s,\tilde p,\Omega} \leq C_1 t^{-\tilde
  s-\tilde\mu+\tilde r}|u|_{\tilde r,\tilde p,\Omega}
  \quad \mbox{ for all $t \in (0,T)$.} 
  $$
  \item 
  \label{item:lemma:estimate-difference-U-fractional-ii}
  If the pair $(\tilde\sigma,\tilde p)$ satisfies (\ref{eq:condition-on-p-sigma}), then there 
  exists $C_2 = C_2(\tilde p,\tilde q,\tilde r,\tilde s,T,\Lpsi,d,\|K\|_{L^\infty},\operatorname*{Lip}(K),\tilde\sigma)$ such that 
  \begin{eqnarray*}
  \label{eq:lemma:estimate-difference-U-fractional-2}
  \left\|\int_{t=0}^T t^{-1+\tilde s-\tilde r+\tilde\mu+\tilde\sigma}
  V(t,\cdot)\,dt\right\|_{\tilde s,\tilde q,\Omega}
  &\leq& C_2 |u|_{\tilde r,\tilde p,\Omega}. 
  \end{eqnarray*}
  \end{enumerate}
\end{lem}
\begin{proof}
  We proceed as in the proof of Lemma~\ref{lemma:estimate-U-fractional}. With the translation vector
  ${\mathbf t}$ there and the analogous change of variables in $\bfz$ and $\bfz^\prime$ we obtain 
  \begin{align*}
   V(t,\bfx) - V(t,\bfy) &= 
   \int_{\bfz \in B_1} \int_{\bfz^\prime \in B_1} B(\bfx,\bfy,\bfz,\bfz^\prime) 
  \left[ u(\bfy + t(\bfz + \psi(\bfy))) - u(\bfy + t(\bfz^\prime +\psi(\bfy)))]\right]\,d\bfz^\prime\,d\bfz,
  \end{align*}
  where 
  \begin{align*}
  B(\bfx,\bfy,\bfz,\bfz^\prime) := 
  K(\bfx,\bfz - {\mathbf t},\bfz^\prime - {\mathbf t}) \chi_{B_1+{\mathbf t}}(\bfz)
  \chi_{B_1 + {\mathbf t}}(\bfz^\prime) 
   - K(\bfy,\bfz,\bfz^\prime) \chi_{B_1}(\bfz) \chi_{B_1}(\bfz^\prime). 
  \end{align*}
  As in the proof of Lemma~\ref{lemma:estimate-U-fractional},
  we get with $Z:= B_1 \cup (B_1 + {\mathbf t})$:  
  \begin{equation}
  |B(\bfx,\bfy,\bfz,\bfz^\prime)| 
  \leq C \min\{1,|\bfx-\bfy|/t\} 
  \chi_{Z}(\bfz) \chi_{Z}(\bfz^\prime). 
  \end{equation}
  Upon setting 
  $$
  v(\bfy,\bfy^\prime) = \frac{|u(\bfy) - u(\bfy^\prime)|}{|\bfy - \bfy^\prime|^{\tilde r+d/\tilde p} }
  $$
  we get in analogy to the procedure in (\ref{eq:lemma:estimate-difference-U-2000})
  \begin{align*}
    \left| V(t,\bfx) - V(t,\bfy)\right| & \lesssim
    \min\{1, |\bfx - \bfy|/t\} t^{\tilde r} \int_{\bfz \in Z}
    \|v(\bfy + t(\bfz + \psi(\bfy)),\cdot)\|_{L^{\tilde p}(\Omega)}\,d\bfz \\
    &= \min\{1,|\bfx - \bfy|/t\} t^{\tilde r}
    \int_{\bfz \in B_1} \|v(\bfx + t (\bfz + \psi(\bfx)),\cdot)\|_{L^{\tilde p}(\Omega)}
    + \|v(\bfy + t (\bfz + \psi(\bfy)),\cdot)\|_{L^{\tilde p}(\Omega)}
   \,d\bfz.
  \end{align*}
  We recognize the similarity with the situation in Lemma~\ref{lemma:estimate-difference-U}. We set 
  $U(t,\bfx):= \int_{\bfz \in B_1} \|v(\bfx + t (\bfz + \psi(\bfx)),\cdot)\|_{L^{\tilde p}(\Omega)}$ and arrive at 
  $$
    \frac{|V(t,\bfx) - V(t,\bfy)|}{|\bfx - \bfx|^{\tilde s + d/\tilde q}} \lesssim
    \frac{\min\{1,|\bfx-\bfy|/t\}}{|\bfx - \bfy|^{\tilde s + d/\tilde q}} t^{\tilde r}
    \left[ U(t,\bfx) + U(t,\bfy)\right].
  $$
  Again, given the symmetry in the variables $\bfx$ and $\bfy$,
  we get as in (\ref{eq:lemma:estimate-U-fractional-200}) 
  $$
    |V(t,\cdot)|_{\tilde s,\tilde q,\Omega} \lesssim t^{-\tilde s-\tilde\mu+\tilde r}
    |u|_{\tilde r,\tilde p,\Omega}, 
  $$
  which is the assertion of part (\ref{item:lemma:estimate-difference-U-fractional-i}) of the lemma.
  For part (\ref{item:lemma:estimate-difference-U-fractional-ii}) of the
  lemma we proceed analogously to the proof
  of Lemma~\ref{lemma:estimate-U-fractional}, (\ref{item:lemma:estimate-U-fractional-ii}).
\end{proof}
We now come to the analog of Lemma~\ref{lemma:sobolev-embedding-Lq-left}. 
\begin{lem}
\label{lemma:sobolev-embedding-Ws-left}
Let $\Omega$ be as in Theorem~\ref{thm:sobolev-embedding}.
Let $1 \leq p \leq q < \infty$. Define $\mu = d (p^{-1} - q^{-1})$ as in (\ref{eq:mu}). 
Let $s \in (0,1)$ with $s+\mu\leq 1$ and assume that one of the following cases occurs: 
\begin{enumerate}[(a)]
\item 
$r = s+ \mu$ and $p > 1$. 
\item 
$r  > s+ \mu$ and $p \ge 1$. 
\end{enumerate}
Then, there is a constant $C = C(p,q,r,s,\Lpsi,T,d)$ such that 
$$
|u|_{s,q,\Omega} \leq C \|u\|_{r,p,\Omega}. 
$$
\end{lem}
\begin{proof}
  We do not lose generality if we assume $\mu>0$.
  The proof is divided into several steps.
  \begin{enumerate}[(i)]
    \item 
\label{item:lemma:sobolev-embedding-Ws-left-i}
The case $s+\mu\leq r < 1$.\\
      We use Lemma~\ref{lemma:taylor} and write
      $u(\bfx) = M(u)(T,\bfx) + M_S(u)(T,\bfx)$. 
      From Lemma~\ref{lemma:estimate-U-fractional}, (\ref{item:lemma:estimate-U-fractional-i}) 
      we get $|M(u)(T,\cdot)|_{s,q,\Omega} \leq C \|u\|_{0,p,\Omega}$.  
      Lemma~\ref{lemma:estimate-difference-U-fractional},
      (\ref{item:lemma:estimate-difference-U-fractional-ii})
      implies 
      $|M_S(u)(T,\cdot)|_{s,q,\Omega} \leq C |u|_{r,p,\Omega}$ if 
      either $r - s = \mu$ together  with $p > 1$ or 
      $r - s > \mu$ together  with $p \ge 1$.
    \item 
\label{item:lemma:sobolev-embedding-Ws-left-ii}
The case $s+\mu < 1\leq r$.\\
      It suffices to consider the case $r=1$.
      We use Lemma~\ref{lemma:taylor} and write
      $u(\bfx) = M(u)(T,\bfx) + M_R(u)(T,\bfx)$.
      In Lemma~\ref{lemma:estimate-U-fractional} we choose
      $\tilde s = s$, $\tilde q =q$, $\tilde p = p$ and obtain
      due to 
      Lemma~\ref{lemma:estimate-U-fractional},  (\ref{item:lemma:estimate-U-fractional-i})
      the bound $|M(u)(T,\cdot)|_{s,q,\Omega} \lesssim \|u\|_{0,p,\Omega}$.
      Next, since $\tilde s+\tilde\mu<1$, we can choose $\tilde\sigma>0$ in
      Lemma~\ref{lemma:estimate-U-fractional},  (\ref{item:lemma:estimate-U-fractional-ii})
      to get
      $\sn{M_R(u)(T,\cdot)}_{s,q,\Om}\lesssim \vn{u}_{1,p,\Om}$.
    \item 
\label{item:lemma:sobolev-embedding-Ws-left-iii}
The case $s +\mu = 1\leq r$.\\
      We use again the representation $u(\bfx) = M(u)(T,\bfx)+M_R(u)(T,\bfx)$.
      The contribution $M(u)(T,\cdot)$ is 
      treated again with Lemma~\ref{lemma:estimate-U-fractional},
      (\ref{item:lemma:estimate-U-fractional-i}).
      If $r=1$, the contribution $M_R(u)(T,\cdot)$ is handled by 
      Lemma~\ref{lemma:estimate-U-fractional}, (\ref{item:lemma:estimate-U-fractional-ii}), where
      we choose $\tilde\sigma=0$ and hence require $p>1$.
      If, on the other hand, $r>1$, choose an arbitrary $\mu_\star$ with $0 < \mu_\star < \mu$ and
      $r > 1+\mu-\mu_\star$ and
      define $p_\star$ via $\mu_\star = d(p_\star^{-1}-q^{-1})$.
      Note that $p_\star>p\geq1$. As $s + \mu_\star < 1$, we obtain from
      step (\ref{item:lemma:sobolev-embedding-Ws-left-ii}) that
      $\sn{u}_{s,q,\Om}\lesssim\vn{u}_{1,p_\star,\Om}$. As
      $r-1>\mu-\mu_\star = d(p^{-1}-p_\star^{-1})\notin\N$,
      we obtain from Lemma~\ref{lemma:sobolev-embedding-Lq-left}, (b), that
      $\vn{u}_{1,p_\star,\Om} \lesssim \vn{u}_{r,p,\Om}$.
\qedhere
  \end{enumerate}
\end{proof}
The following result complements Lemma~\ref{lemma:sobolev-embedding-Lq-left} with the cases
$r > \mu \in\N$.
\begin{lem}
  \label{lemma:sobolev-embedding-Lq-left:2}
  Let $\Omega$ be as in Theorem~\ref{thm:sobolev-embedding}. Assume $r \ge 0$, $1 \leq p \leq q < \infty$ 
  and set $\mu:= d (p^{-1} - q^{-1})$ as in (\ref{eq:mu}).
  Assume that $\displaystyle r > \mu\in\N$.  
  Then there is a constant $C = C(p,q,r,\Lpsi,T,d)$, which depends solely on the
  quantities indicated (and the 
  assumption that $\operatorname*{diam}\Omega \leq 1$), such that  
  $$
  \|u\|_{0,q,\Omega} \leq C \|u\|_{r,p,\Omega}. 
  $$
\end{lem}
\begin{proof}
  Choose a $\mu_\star\notin\N$ with $\mu-1 < \mu_\star<\mu$ and
  define $p_\star$ via $\mu_\star = d(p_\star^{-1}-q^{-1})$. Note that $p_\star>p\geq 1$.
  Lemma~\ref{lemma:sobolev-embedding-Lq-left} shows
  $\vn{u}_{q,\Om}\lesssim\vn{u}_{\mu_\star,p_\star,\Om}$.
  In a second step, observe that
  \begin{align*}
    \underbrace{\mu_\star - (\mu-1)}_{<1}+
    \underbrace{\mu-\mu_\star}_{d(p^{-1}-p_\star^{-1})}  = 1 <
    \underbrace{r - (\mu-1)}_{>1}
  \end{align*}
  Hence we can apply Lemma~\ref{lemma:sobolev-embedding-Ws-left}
  for $\sn{\bft}=\mu-1$ and obtain
  $\sn{D^\bft u}_{\mu_\star-(\mu-1),p_\star,\Om} \lesssim
  \vn{D^\bft u}_{r-(\mu-1),p,\Om}$.
  Furthermore, since $\mu-\mu_\star < 1$, we can use Lemma~\ref{lemma:sobolev-embedding-Lq-left}
  for $\sn{\bft}\leq \mu-1$ to obtain
  $\sn{D^\bft u}_{p_\star,\Om}\lesssim \vn{D^\bft u}_{1,p,\Om}$.
  As we took all terms of $\vn{u}_{\mu_\star,p_\star,\Om}$ into account, the result follows.
\end{proof}
The following result complements Lemma~\ref{lemma:sobolev-embedding-Ws-left} with
the case $s+\mu>1$.
\begin{lem}
  \label{lemma:sobolev-embedding-Ws-left:2}
  Let $\Omega$ be as in Theorem~\ref{thm:sobolev-embedding}.
  Let $1 \leq p \leq q < \infty$. Define $\mu = d (p^{-1} - q^{-1})$ as in (\ref{eq:mu}). 
  Let $s \in (0,1)$ with $s+\mu>1$ and assume that one of the following cases occurs: 
  \begin{enumerate}[(a)]
  \item 
  $r = s+ \mu$ and $p > 1$. 
  \item 
  $r  > s+ \mu$ and $p \ge 1$. 
  \end{enumerate}
  Then, there is a constant $C = C(p,q,r,s,\Lpsi,T,d)$ such that 
  $$
  |u|_{s,q,\Omega} \leq C \|u\|_{r,p,\Omega}. 
  $$
\end{lem}
\begin{proof}
  Define $p_\star$ by $1-s = d(p_\star^{-1} - q^{-1})$. As $s+\mu>1$, it holds $p_\star > p \ge 1$.
  By Lemma~\ref{lemma:sobolev-embedding-Ws-left}, we get 
  $|u|_{s,q,\Omega} \lesssim \|u\|_{1,p_\star,\Omega}$.
  In a second step, observe that
  $r - 1 \ge s-1+\mu = d \left( p^{-1} - p_\star^{-1}\right)$,
  and hence Lemmas~\ref{lemma:sobolev-embedding-Lq-left}
  and~\ref{lemma:sobolev-embedding-Lq-left:2} imply
  the estimate $\|u\|_{1,p_\star,\Omega} \leq \|u\|_{r,p,\Omega}$.
  This concludes the proof.
\end{proof}
\subsection{Auxiliary results}
We need the Minkowski inequality (cf. \cite[Appendix A.1]{stein1}, \cite[Chap.~2, eqn.~(1.6)]{devore1})
\begin{equation}
\label{eq:minkowski} 
\left(\int_{\bfy \in {\mathcal Y}} \left( \int_{\bfx \in {\mathcal X}} |F(\bfx,\bfy)| \,d\bfx\right)^p \,d\bfy
\right)^{1/p}
\leq 
\int_{\bfx \in {\mathcal X}} \left( \int_{\bfy \in {\mathcal Y}} |F(\bfx,\bfy)|^p \,d\bfy\right)^{1/p} \,d\bfx, 
\qquad 1 \leq p < \infty.
\end{equation}
The following is an application Marcinkiewicz's interpolation theorem as 
worked out in \cite[Example~4, Sec. IX.4]{reed-simonII}:
\begin{lem}
  \label{lemma:marcinkiewicz}
  Let $1 <p,q<\infty$ and assume $0 <\lambda < d$. Let $p^{-1} + r^{-1} + \lambda d^{-1} = 2$. Then 
  $$
  \int_{\bfx \in \R{d}} \int_{\bfy \in \R{d}} \frac{|f(\bfx)||g(\bfy)|}{|\bfx -
  \bfy|^\lambda}\,d\bfx,\,d\bfy \leq C \|f\|_{0,p,\R{d}}
  \|g\|_{0,r,\R{d}} 
  $$
  for all $f \in L^p(\R{d})$, $g \in L^r(\R{d})$. The constant $C$ depends only on $p$, $r$, $\lambda$, and $d$. 
  That is, the map $f \mapsto \int_{\R{d}} f(\bfy) |\bfx - \bfy|^{-\lambda}\,d\bfy$ is a bounded linear map
  from $L^p(\R{d})$ to $L^{r/(r-1)}(\R{d})$.
\qed
\end{lem}
\begin{lem}
  \label{lemma:hardy-style}
  Let $u \in C^\infty_0(\R{d})$ and $B_R$ be a ball of radius $R$ centered at the origin. 
  Let $s<1$ and $d-1+s \ne 0$. 
Then, for $T > 0$
  \begin{equation*}
  \int_{t=0}^T t^{-s} \int_{\bfz \in B_R} u(t \bfz)\,d\bfz\,dt = 
  \frac{1}{d-1+s} \left[ R^{s+d-1}
  \int_{\bfz \in B_{RT}} |\bfz|^{-s - d + 1} u(\bfz)\,d\bfz - T^{-s-d+1} \int_{B_{RT}} u(\bfz)\,d\bfz
  \right].
  \end{equation*}
\end{lem}
\begin{proof} 
  The proof follows from the introduction of polar coordinates, Fubini, and an integration by parts.  
  We use polar coordinates $(r,\omega)$ with $\omega \in \partial B_1$. Then:
  \begin{align*}
  \int_{t=0}^T t^{-s} \int_{\bfz \in B_R} u(t\bfz) &=
  \int_{\omega \in \partial B_1} \int_{t=0}^T t^{-s} \int_{r=0}^{R}  u(tr \omega ) r^{d-1} 
  = 
  \int_{\omega \in \partial B_1} \int_{t=0}^T t^{-s} t^{-(d-1)-1}\int_{\rho=0}^{Rt}  u(\rho \omega )\rho^{d-1}  \\
  &= \int_{\omega \in \partial B_1} \frac{1}{-s-d+1}\left[ \left. t^{-s-d+1} \int_{\rho=0}^{Rt} u(\rho \omega) \rho^{d-1}\right|_{t=0}^{t = T} - 
  R^d \int_{t=0}^T t^{-s-d+1} u(R t\omega) t^{d-1} 
  \right] \\
  &= \frac{1}{-s-d+1} \left[ T^{-s-d+1} \int_{B_{RT}} u(\bfz)\,d\bfz - R^{s+d-1} \int_{\bfz \in B_{RT}} |\bfz|^{-s-d+1} u(\bfz)\,d\bfz 
  \right].
\tag*{\qedhere}
\end{align*}
\end{proof}
\newcommand{\Ccomp}{C_{\rm comp}}
\section{Element-by-element approximation for variable polynomial degree in 3D}
\label{sec:appendixB}
In this section, we generalize the operator $\Pi^{MPS}_p$ of \cite[Thm.~{B.3}]{mps13} 
to variable ($\gamma_p$-shape regular; cf. (\ref{eq:shape-regular-p}) polynomial degree distributions. 
We consider only the 3D case as the 2D case is very similar. 
Structurally, we proceed as in \cite[Appendix~B]{mps13}: we define polynomial
approximation operators that permit an ``element-by-element'' construction
of a global piecewise polynomial approximation operator. To that end, we will
fix the operator in vertices, edges, faces, and elements separately. 
We need to define approximation operators defined on edges, face, elements. 
Since we need to distinguish between ``low'' and ``high'' polynomial degree, 
we introduce {\bf two} operators for edges, faces, and elements. 
\begin{defn}
Let $\refK \subset \R{3}$ be the reference tetrahedron. Let $\cV$ be the set
of the $4$ vertices, $\cE$ be the set of the $6$ edges, and $\cF$ be the set
of the $4$ faces of $\refK$.
Let $p_e$, $p_f$, $p_{K}$ be polynomial degrees associated with an edge $e$, a face $f$, 
and the element $\refK$. For an edge $e$ we denote by $\cV(e)$ the set of endpoints. For a face $f$, we 
denote by $\cV(f)$ the set of vertices of $f$ and by $\cE(f)$ the set of edges of $f$.  
\begin{enumerate}[(i)]
\item (edge operators) 
For an edge $e$ with associated polynomial degree $p_e$ define the operators 
$\pi_e^{h}: C^\infty(\overline{e}) \rightarrow {\mathcal P}_{p_e}(e)$ and 
$\pi_e^{p}: C^\infty(\overline{e}) \rightarrow {\mathcal P}_{\lfloor p_e/4\rfloor}(e)$ by 
\begin{itemize}
\item 
$(\pi_e^h u) \in {\mathcal P}_{p_e}$ is the unique minimizer of 
$$
v \mapsto \|u - v\|_{0,2,e}
$$
under the constraint $(\pi_e^h u)(V) = u(V)$ for all $V \in \cV(e)$. 
\item 
$(\pi_e^p u) \in {\mathcal P}_{\lfloor p_e/4\rfloor}$ is the unique minimizer of 
$$
v \mapsto p_e^4 \sum_{j=0}^4 p_e^{-j} |u - v|_{j,2,e}
$$
under the constraint that 
$\partial_e^j v(V) = \partial_e^j u(V)$ for $j \in \{0,1,2,3\}$ and all $V \in \cV(e)$. 
Here, $\partial_e$ denote the tangential derivative along $e$.
For $\pi_e^p$ to be meaningful, we have to require $p_e \ge 28$.
\end{itemize}
\item (face operators) 
For a face $f$ with associated polynomial degree $p_f$ define the operators 
$\pi_f^{h}: C^\infty(\overline{f}) \rightarrow {\mathcal P}_{p_f}(f)$ and 
$\pi_f^{p}: C^\infty(\overline{f}) \rightarrow {\mathcal P}_{\lfloor p_f/2\rfloor}(f)$ as follows, 
{\em assuming} that a continuous, piecewise polynomial (of degree $\leq p_f$) 
approximation $\pi_{\partial f} u$ on the boundary of 
$f$ is given: 
\begin{itemize}
\item 
$(\pi_f^h u) \in {\mathcal P}_{p_f}$ is the unique minimizer of 
$$
v \mapsto \|u - v\|_{0,2,f}
$$
under the constraint $(\pi_f^h u)|_e = (\pi_{\partial f} u)|_e$ for all $e \in \cE(f)$. 
\item  
Assume that $\pi_{\partial f} u$ is given by $\pi_e^p u$ for all three edges $e \in \cE(f)$. Then, 
$(\pi_f^p u) \in {\mathcal P}_{2 \lfloor p_f/4\rfloor}$ is the unique minimizer of 
$$
v \mapsto p_f^4 \sum_{j=0}^4 p_f^{-j} |u - v|_{j,2,f}
$$
under the constraint that $(\pi_f^p u)|_e = \pi_e^p u$ for all $e \in \cE(f)$ and additionally 
the mixed derivatives at the 3 vertices $V \in \cV(f)$ satisfy 
$(\partial_{e_1,V} \partial_{e_2,V} \pi_f u)(V) = (\partial_{e_1,V} \partial_{e_2,V} u)(V)$; here, 
at a vertex $V \in \cV(f)$, $e_1$ and $e_2$ are the two edges of $f$ meeting at $V$ and  
$\partial_{e_1,V}$, $\partial_{e_2,V}$ represent derivatives along these directions 
(taking $V$ as the common origin). Note that this requires 
$2 \lfloor p_f/4\rfloor \ge \max_{e \in \cE(f)} \lfloor p_e/4\rfloor$ and thus $p_f \ge 16$. 
\end{itemize}
\item (element operators)
Assume that $\pi_{\partial \refK} u$ is a continuous, piecewise polynomial (of degree $\leq p_K$) 
approximation on $\partial \refK$. Define 
$(\pi_K^p u) \in {\mathcal P}_{p_K}$ as the unique minimizer of 
$$
v \mapsto p_f^4 \sum_{j=0}^4 p_f^{-j} |u - v|_{j,2,\refK}
$$
under the constraint that $(\pi_K^p u)|_{\partial\refK} = \pi_{\partial \refK} u$. 
\end{enumerate}
\end{defn}
\begin{thm}
\label{thm:element-by-element-variable}
Consider the reference tetrahedron $\refK$. 
Fix $p_{\rm ref} \ge 7$. 
Let the element degree $p_K$, the face degrees $p_f$ and the edge degrees $p_e$ satisfy the ``minimum rule'': 
\begin{align*}
&1 \leq p_e \leq p_f \qquad \forall f \in \cF, \quad \forall e \in \cE(f), \\
&1 \leq p_f \leq p_K \qquad \forall f \in \cF. 
\end{align*}
Assume the polynomial degrees are {\em comparable}, i.e., there is $\Ccomp$ such that 
\begin{align*}
&1 \leq p_e \leq p_f \leq \Ccomp p_e \qquad \forall f \in \cF, \quad \forall e \in \cE(f), \\ 
&1 \leq p_f \leq p_K \leq \Ccomp p_f \qquad \forall f \in \cF. 
\end{align*}
Define 
$$
p:=\min\{p_e\,|\,e \in \cE\}. 
$$
(Note that the minimum rule implies additionally $p \leq p_f$ for all $f \in \cF$ and {\sl a fortiori} $p \leq p_K$.)

Define the approximation operator $\pi:C^\infty(\overline{\refK}) \rightarrow {\mathcal P}_{p_K}$ 
as follows: 
\begin{enumerate} 
\item (vertices)
\label{item:thm:element-by-element-variable-1}
For each vertex $V \in \cV$: Require $(\pi u)(V) = u(V)$. 
\item (edges)
\label{item:thm:element-by-element-variable-2}
For each edge $e\in \cE$: If $\lfloor p_e/4\rfloor \ge p_{\rm ref}$, then set $(\pi u)|_e:= \pi_e^p u$. 
Else, set $(\pi u)|_e = \pi_e^h u$. 
\item (faces) 
\label{item:thm:element-by-element-variable-3}
For each face $f \in \cF$: By step~\ref{item:thm:element-by-element-variable-2}, $(\pi u)|_{\partial f}$ 
is fixed. If $(\pi u)|_{\partial f}$ is given by 
$\pi^p_e u$ for all three edges $e \in \cE(f)$, then set 
$(\pi u)|_f:= \pi^p_f u$. Else, set $(\pi u)|_f = \pi^h_f u$. 
\item (element) 
\label{item:thm:element-by-element-variable-4}
The last three steps have fixed $(\pi u)|_{\partial \refK}$. Set $\pi u:= \pi^p_K u$. 
\end{enumerate} 
Then:
\begin{enumerate}[(i)]
\item (approximation property)
\label{item:thm:element-by-element-variable-i}
For each $s > 5$, there is $C_s > 0$ such that 
\begin{align}
\label{eq:thm:element-by-element-variable-10}
&\sum_{j=0}^2 p^{2-j} \|u - \pi u\|_{j,2,\refK} \leq C_s p^{-(s-2)} \|u\|_{s,2,\refK} 
\qquad \forall u \in H^s(\refK).
\end{align}
\item (polynomial reproduction) 
\label{item:thm:element-by-element-variable-ii}
\begin{align*}
 \pi u = u  &\qquad \forall u \in {\mathcal P}_{\lfloor p/4\rfloor} &&\mbox{ if $ p \ge 4$}, \\
 \pi u = u  &\qquad \forall u \in {\mathcal P}_{p} &&\mbox{ if $ p_e < 4 p_{\rm ref}$ for {\em all} edges $e \in \cE$}.
\end{align*}
\item (locality)
\label{item:thm:element-by-element-variable-iii} 
\begin{itemize}
\item In each vertex $V \in \cV$, $\pi u$ is completely determined by $u|_V$. 
\item On each edge $e \in \cE$, $(\pi u)|_e$ is completely determined by
  $u|_e$, $p_e$, and $p_{\rm ref}$. 
\item On each face $f \in \cF$, $(\pi u)|_f$ is completely determined by $u|_f$, $p_f$, the 
degrees $p_e$, $e \in \cE(f)$, and $p_{\rm ref}$. 
\end{itemize}
\end{enumerate}
\end{thm}
\begin{proof} 
{\em Proof of (\ref{item:thm:element-by-element-variable-iii}):} This follows by construction.

{\em Proof of (\ref{item:thm:element-by-element-variable-ii}):} 
If $p \ge 4$, then inspection of the construction shows that $\pi u = u$ for all polynomials 
of degree $\lfloor p/4\rfloor$. 
If $p_e < 4 p_{\rm ref}$ for {\em all} edges $e \in \cE$, then $(\pi u)|_e =\pi^h_e u = u|_e$ for all polynomials 
$u$ of degree $p$. Since $\pi u$ is of the form $\pi^h_e u$ on all edges, the face values $(\pi u)|_f$ 
are also given by $\pi^h_f u$. Hence, polynomials of degree $p$ are reproduced on all faces and thus also
on the element. 

{\em Proof of (\ref{item:thm:element-by-element-variable-i}):} 
As a first step, we reduce the question to the case that 
$$
\lfloor p_e/4 \rfloor \ge p_{\rm ref} \qquad \forall e \in \cE, 
\qquad 
\lfloor p_f/2 \rfloor \ge p_{\rm ref} \qquad \forall f \in \cF. 
$$
In the converse case, one of the edge polynomials $p_{e'}$ satisfies $p_{e'}
\leq 4 (p_{\rm ref} +1)$ 
or one face $f'$ satisfies $p_{f'} \leq 2 (p_{pref} + 1)$. In view of the comparability of the 
degrees, this implies that $\max_{e \in \cE} p_e \leq \max_{f \in \cF} p_f \leq
p_K \leq 4 \Ccomp (p_{\rm ref}+1)$.
In other words: The polynomial degrees are bounded and therefore only finitely many cases 
for the operator $\pi$ can arise. By norm equivalence on finite dimensional space, the 
bound (\ref{eq:thm:element-by-element-variable-10}) holds. 

We may now assume $\lfloor p_e/4\rfloor \ge p_{\rm ref}$ for all edges $e \in \cE$ 
and $\lfloor p_{f}/2 \rfloor \ge p_{\rm ref}$ for all faces. Recall that 
$p= \min\{p_e\,|\, e \in \cE\}$ and that by our assumption of the ``minimum rule'' we therefore have 
$p \leq p_f \leq p_K$.
Define $\widetilde p:= \lfloor p/4\rfloor$. We may assume $\widetilde p \ge 1$. 
Then we can proceed as in the proof of \cite[Thm.~{B.3}]{mps13} with $\widetilde p$ taking the role
of $p$ in the proof of \cite[Thm.~{B.3}]{mps13}, from where the condition $s > 5$ arises. 
This leads to (\ref{eq:thm:element-by-element-variable-10}). 
\end{proof}
We are now in position to prove Corollary~\ref{cor:MPS}.
\begin{proof} (of Corollary~\ref{cor:MPS})
We only consider the case $d = 3$. Let $s > \max\{5,r_{\max}\}$, with $r_{\max}$ of the statement
of Corollary~\ref{cor:MPS}. Let $\eps$ be defined 
by Lemma~\ref{lem:lsf}. Then, the smoothed function $\II_\eps u$ satisfies by the same reasoning as in 
the proof of Theorem~\ref{thm:qi}
for $0 \leq r \leq q \leq s$ the bound 
\begin{equation}
\label{eq:proof-of-cor-MPS-10}
\|\II_\eps u\|_{q,2,K} \leq C \left(\frac{h_K}{p_K}\right)^{r-q} \|u\|_{r,2,\omega_K} 
\qquad \forall K \in \cT.
\end{equation}
Next, we wish to employ the operator of Theorem~\ref{thm:element-by-element-variable}. To that end, 
we associate with each edge $e$ and each face $f$ of the triangulation $\cT$ a polynomial degree by the 
``minimum rule'', i.e., 
\begin{align*}
p_e:= \min\{ p_K\,|\, \mbox{ $e$ is an edge of $K$}\}, 
\qquad 
p_f:= \min\{ p_K\,|\, \mbox{ $f$ is a face of $K$}\}. 
\end{align*}
The definition of the edge and face polynomial degrees implies 
\begin{align*}
p_f \ge p_e \qquad \forall \mbox{ edges $e$ of a face $f$}, & \qquad \qquad \mbox{ and } \qquad \qquad 
p_K \ge p_f \qquad \forall \mbox{ faces $f$ of an element $K$}, 
\end{align*} 
which is the ``minimum rule'' required in Theorem~\ref{thm:element-by-element-variable}.
Fix an element $K$ and let $p:= \min\{p_e\,|\, \mbox{$e$ is an edge of $K$}\}$. 
The $\gamma$-shape regularity of the mesh and the polynomial degree distribution implies the
existence of $\Ccomp$ such that (independent of $K$)
\begin{equation}
\label{eq:proof-of-cor3.4-5} 
p_{\max} := \max\{p_e\,|\, \mbox{$e$ is an edge of $K$}\} \leq \Ccomp p = 
\Ccomp \min\{p_e \,|\, \mbox{$e$ is an edge of $K$}\}. 
\end{equation}
We recognize that 
the definition of $\widehat p_K$ in the statement of Corollary~\ref{cor:MPS} coincides 
with $p$: 
\begin{align}
\label{eq:proof-of-cor3.4-10}
\widehat p_K &= p .
\end{align}
Fix $p_{\rm ref} \in \N$ so that 
\begin{equation}
\label{eq:proof-of-cor3.4-20}
p_{\rm ref} \ge \Ccomp r_{\max}. 
\end{equation}
The approximation operator $\pi$ of 
Theorem~\ref{thm:element-by-element-variable} satisfies on the reference element $\refK$ 
\begin{equation}
\label{eq:proof-of-cor3.4-200}
\sum_{j=0}^2 p^{2-j} |v - \pi v|_{j,2,\refK} \leq C p^{-(s-2)} \|v\|_{s,2,\refK}. 
\end{equation}
In order to get the correct powers of $h_K$, we 
exploit that $\pi$ reproduces polynomials by Theorem~\ref{thm:element-by-element-variable}. 
We consider two cases: 
\newline
{\em 1. Case:} $p_{\max} < 4 p_{\rm ref}$. In this case, Theorem~\ref{thm:element-by-element-variable}
implies immediately that $\pi v = v$ for all $v \in {\mathcal P}_p$. 
\newline 
{\em 2. Case:} $p_{\max} \ge 4 p_{\rm ref}$. In this case, 
$\displaystyle 
\Ccomp p \ge p_{\max} \ge 4 p_{\rm ref} \ge 4 \Ccomp r_{\max} 
$
so that 
$
p/4 \ge  r_{\max}. 
$
Hence, Theorem~\ref{thm:element-by-element-variable} implies $\pi v = v$ for
all $v \in {\mathcal P}_{r_{\max}}$. 

Combining the two cases yields 
\begin{equation}
\label{eq:proof-of-cor3.4-50}
\pi v = v \qquad \forall v \in {\mathcal P}_{\min\{p,r_{\max}\}}. 
\end{equation}
Hence combining (\ref{eq:proof-of-cor-MPS-10}), (\ref{eq:proof-of-cor3.4-200}), and
(\ref{eq:proof-of-cor3.4-50}) together with the usual scaling arguments and the
Bramble-Hilbert Lemma~\ref{lem:bramblehilbert} leads to
\begin{align*}
& \sum_{j=0}^2 h_K^j p^{-j} |\II_\eps u - \pi \II_\eps u|_{j,2,K} 
 \lesssim p^{-s} \sum_{q=1 + \min\{p,r_{\max}\} }^s h_K^q |\II_\eps u|_{q,2,K} 
\\ & 
\lesssim  p^{-s} \left[ \sum_{q=1 + \min\{p,r_{\max}\} }^{r}  h_K^q |\II_\eps u|_{q,2,K} + 
   \sum_{q=r+1 }^{s}  h_K^q |\II_\eps u|_{q,2,K} \right] 
\\ & 
\lesssim  p^{-s} h_K^{1 + \min\{p,r_{\max}\}} \|\II_\eps u\|_{r,2,K} + 
   \sum_{q=r+1 }^{s} p^{-s}  h_K^q |\II_\eps u|_{q,2,K}  
\\ & 
\lesssim  p^{-s} h_K^{1 + \min\{p,r_{\max}\}} \|\II_\eps u\|_{r,2,K} + 
  p^{-s} \sum_{q=r+1}^s h_K^q 
\left(\frac{h_K}{p_K}\right)^{r-q}\|u\|_{r,2,\omega_K} 
\\ & 
\lesssim  \left[ \frac{h_K^{r}}{p_K^{r}} + p_K^{-s} h_K^{1 + \min\{p,r\}}
\right]\|u\|_{r,2,\omega_K} 
\stackrel{r \leq r_{\max} \leq s}{\lesssim}  p_K^{-r} \left[ h_K^r + h_K^{1 + \min\{p_K,r\}}\right] \|u\|_{r,2,\omega_K}.
\tag*{\qedhere}
\end{align*}
\end{proof}
\section*{Acknowledgements}
\noindent
The first author gratefully acknowledges financial support 
by CONICYT through FONDECYT project 3140614.
\bibliographystyle{plain}
\section*{References}
\bibliography{ref}

\def\cprime{$'$}
\begin{thebibliography}{10}

\bibitem{a01}
Gabriel Acosta.
\newblock Lagrange and average interpolation over 3{D} anisotropic elements.
\newblock {\em J. Comput. Appl. Math.}, 135(1):91--109, 2001.

\bibitem{adams-fournier03}
Robert~A. Adams and John J.~F. Fournier.
\newblock {\em Sobolev spaces}, volume 140 of {\em Pure and Applied Mathematics
  (Amsterdam)}.
\newblock Elsevier/Academic Press, Amsterdam, second edition, 2003.

\bibitem{ainsworth1}
Mark Ainsworth and David Kay.
\newblock The approximation theory for the {$p$}-version finite element method
  and application to non-linear elliptic {PDE}s.
\newblock {\em Numer. Math.}, 82(3):351--388, 1999.

\bibitem{apel2}
Thomas Apel.
\newblock Interpolation of non-smooth functions on anisotropic finite element
  meshes.
\newblock {\em M2AN}, 33(6):1149--1185, 1999.

\bibitem{babuska2}
I.~Babu{\v{s}}ka, A.~Craig, J.~Mandel, and J.~Pitk{\"a}ranta.
\newblock Efficient preconditioning for the {$p$}-version finite element method
  in two dimensions.
\newblock {\em SIAM J. Numer. Anal.}, 28(3):624--661, 1991.

\bibitem{babuska-suri87a}
I.~Babu{\v{s}}ka and Manil Suri.
\newblock The {$h$}-{$p$} version of the finite element method with
  quasi-uniform meshes.
\newblock {\em RAIRO Mod\'el. Math. Anal. Num\'er.}, 21(2):199--238, 1987.

\bibitem{babuska1}
I.~Babu{\v{s}}ka and Manil Suri.
\newblock The optimal convergence rate of the {$p$}-version of the finite
  element method.
\newblock {\em SIAM J. Numer. Anal.}, 24(4):750--776, 1987.

\bibitem{belgacem94}
F.~Ben Belgacem.
\newblock Polynomial extensions of compatible polynomial traces in three
  dimensions.
\newblock {\em Comput. Meth. Appl. Mech. Engrg.}, 116:235--241, 1994.

\bibitem{bernardi-dauge-maday92}
C.~Bernardi, M.~Dauge, and Y.~Maday.
\newblock Trace liftings which preserve polynomials.
\newblock {\em C.R. Acad. Sci. Paris, S{\'e}rie {I}}, 315:333--338, 1992.

\bibitem{bernardi-dauge-maday99}
C.~Bernardi, M.~Dauge, and Y.~Maday.
\newblock {\em Spectral Methods for Axisymmetric domains}.
\newblock North Holland/Gauthier-Villars, 1999.

\bibitem{bernardi-dauge-maday07}
C.~Bernardi, M.~Dauge, and Y.~Maday.
\newblock Polynomials in the {S}obolev world (version 2).
\newblock Technical Report~14, IRMAR, 2007.

\bibitem{bernardi}
C.~Bernardi and V.~Girault.
\newblock A local regularization operator for triangular and quadrilateral
  finite elements.
\newblock {\em SIAM J. Numer. Anal.}, 35(5):1893--1916 (electronic), 1998.

\bibitem{bernardi-maday92}
C.~Bernardi and Y.~Maday.
\newblock {\em Approximations spectrales de probl\`emes aux limites
  elliptiques}.
\newblock Math\'ematiques \& Applications. Springer Verlag, 1992.

\bibitem{bernardi-maday97}
C.~Bernardi and Y.~Maday.
\newblock Spectral methods.
\newblock In P.G. Ciarlet and J.L. Lions, editors, {\em Handbook of Numerical
  Analysis, Vol. 5}. North Holland, Amsterdam, 1997.

\bibitem{brenner2008}
Susanne~C. Brenner and L.~Ridgway Scott.
\newblock {\em The mathematical theory of finite element methods}, volume~15 of
  {\em Texts in Applied Mathematics}.
\newblock Springer, New York, third edition, 2008.

\bibitem{burenkov}
Victor~I. Burenkov.
\newblock {\em Sobolev spaces on domains}, volume 137 of {\em Teubner-Texte zur
  Mathematik [Teubner Texts in Mathematics]}.
\newblock B. G. Teubner Verlagsgesellschaft mbH, Stuttgart, 1998.

\bibitem{canuto1}
C.~Canuto, M.~Y. Hussaini, A.~Quarteroni, and T.~A. Zang.
\newblock {\em Spectral methods}.
\newblock Scientific Computation. Springer-Verlag, Berlin, 2006.
\newblock Fundamentals in single domains.

\bibitem{canuto-quarteroni82}
C.~Canuto and A.~Quarteroni.
\newblock Approximation results for orthogonal polynomials in {S}obolev spaces.
\newblock {\em Math. Comp.}, 38(257):67--86, 1982.

\bibitem{demkowicz-cao05}
W.~Cao and L.~Demkowicz.
\newblock Optimal error estimate of a projection based interpolation for the
  {$p$}-version approximation in three dimensions.
\newblock {\em Comput. Math. Appl.}, 50(3-4):359--366, 2005.

\bibitem{cchu}
C.~Carstensen and Jun Hu.
\newblock Hanging nodes in the unifying theory of a posteriori finite element
  error control.
\newblock {\em J. Comput. Math.}, 27(2-3):215--236, 2009.

\bibitem{cms01}
C.~Carstensen, M.~Maischak, and E.~P. Stephan.
\newblock A posteriori error estimate and {$h$}-adaptive algorithm on surfaces
  for {S}ymm's integral equation.
\newblock {\em Numer. Math.}, 90(2):197--213, 2001.

\bibitem{c97}
Carsten Carstensen.
\newblock An a posteriori error estimate for a first-kind integral equation.
\newblock {\em Math. Comp.}, 66(217):139--155, 1997.

\bibitem{cc1}
Carsten Carstensen.
\newblock Quasi-interpolation and a posteriori error analysis in finite element
  methods.
\newblock {\em M2AN Math. Model. Numer. Anal.}, 33(6):1187--1202, 1999.

\bibitem{cfs96}
Carsten Carstensen, Stefan~A. Funken, and Ernst~P. Stephan.
\newblock A posteriori error estimates for {$hp$}-boundary element methods.
\newblock {\em Appl. Anal.}, 61(3-4):233--253, 1996.

\bibitem{cmps04}
Carsten Carstensen, M.~Maischak, D.~Praetorius, and E.~P. Stephan.
\newblock Residual-based a posteriori error estimate for hypersingular equation
  on surfaces.
\newblock {\em Numer. Math.}, 97(3):397--425, 2004.

\bibitem{cs95}
Carsten Carstensen and Ernst~P. Stephan.
\newblock A posteriori error estimates for boundary element methods.
\newblock {\em Math. Comp.}, 64(210):483--500, 1995.

\bibitem{cs96}
Carsten Carstensen and Ernst~P. Stephan.
\newblock Adaptive boundary element methods for some first kind integral
  equations.
\newblock {\em SIAM J. Numer. Anal.}, 33(6):2166--2183, 1996.

\bibitem{ciarlet76a}
P.~G. Ciarlet.
\newblock {\em The Finite Element Method for Elliptic Problems}.
\newblock North-Holland Publishing Company, 1976.

\bibitem{clement1}
Ph. Cl{\'e}ment.
\newblock Approximation by finite element functions using local regularization.
\newblock {\em Rev. Fran\c caise Automat. Informat. Recherche
  Op{\'e}rationnelle S{\'e}r. RAIRO Analyse Num{\'e}rique}, 9(R-2):77--84,
  1975.

\bibitem{Costabel_88_BIO}
Martin Costabel.
\newblock Boundary integral operators on {Lipschitz} domains: Elementary
  results.
\newblock {\em SIAM J. Math. Anal.}, 19:613--626, 1988.

\bibitem{demkowicz08}
L.~Demkowicz.
\newblock Polynomial exact sequences and projection-based interpolation with
  applications to {M}axwell's equations.
\newblock In D.~Boffi, F.~Brezzi, L.~Demkowicz, L.F. Dur{\'a}n, R.~Falk, and
  M.~Fortin, editors, {\em Mixed Finite Elements, Compatibility Conditions, and
  Applications}, volume 1939 of {\em Lectures Notes in Mathematics}. Springer
  Verlag, 2008.

\bibitem{demkowicz-buffa05}
L.~Demkowicz and A.~Buffa.
\newblock {$H^1$}, {$H({\rm curl})$} and {$H({\rm div})$}-conforming
  projection-based interpolation in three dimensions. {Q}uasi-optimal
  {$p$}-interpolation estimates.
\newblock {\em Comput. Methods Appl. Mech. Engrg.}, 194(2-5):267--296, 2005.

\bibitem{oden-demkowicz-rachowicz-hardy89}
L.~Demkowicz, T.J. Oden, W.~Rachowicz, and O.~Hardy.
\newblock Towards a universal $hp$ finite element strategy. part 1. constrained
  approximation and data structure.
\newblock {\em Comput. Meth. Appl. Mech. Engrg.}, 77:79--112, 1989.

\bibitem{demkowicz-gopalakrishnan-schoeberl08}
Leszek Demkowicz, Jayadeep Gopalakrishnan, and Joachim Sch{\"o}berl.
\newblock Polynomial extension operators. {I}.
\newblock {\em SIAM J. Numer. Anal.}, 46(6):3006--3031, 2008.

\bibitem{demkowicz-kurtz-pardo-paszynski-rachowicz-zdunek08}
Leszek Demkowicz, Jason Kurtz, David Pardo, Maciej Paszy{\'n}ski, Waldemar
  Rachowicz, and Adam Zdunek.
\newblock {\em Computing with {$hp$}-adaptive finite elements. {V}ol. 2}.
\newblock Chapman \& Hall/CRC Applied Mathematics and Nonlinear Science Series.
  Chapman \& Hall/CRC, Boca Raton, FL, 2008.
\newblock Frontiers: three dimensional elliptic and Maxwell problems with
  applications.

\bibitem{devore1}
Ronald~A. DeVore and George~G. Lorentz.
\newblock {\em Constructive approximation}, volume 303 of {\em Grundlehren der
  Mathematischen Wissenschaften [Fundamental Principles of Mathematical
  Sciences]}.
\newblock Springer-Verlag, Berlin, 1993.

\bibitem{dupont-scott80}
T.~Dupont and L.R. Scott.
\newblock Polynomial approximation of functions in {S}obolev spaces.
\newblock {\em Math. Comp.}, 34:441--463, 1980.

\bibitem{ern-guermond15}
A.~Ern and J.-L. Guermond.
\newblock Finite element quasi-interpolation and best approximation.
\newblock Technical report, 2015.
\newblock arXiv:1505.06931.

\bibitem{eva1}
Lawrence~C. Evans and Ronald~F. Gariepy.
\newblock {\em Measure theory and fine properties of functions}.
\newblock Studies in Advanced Mathematics. CRC Press, Boca Raton, FL, 1992.

\bibitem{f02}
Birgit Faermann.
\newblock Localization of the {A}ronszajn-{S}lobodeckij norm and application to
  adaptive boundary element methods. {II}. {T}he three-dimensional case.
\newblock {\em Numer. Math.}, 92(3):467--499, 2002.

\bibitem{gagliardo57}
E.~Gagliardo.
\newblock Caratterizzazione delle tracce sulla frontiera relative ad alcune
  classi di funzioni in $n$ variabili.
\newblock {\em {R}end. {S}em. {M}at. {U}niv. {P}adova}, 27:284--305, 1957.

\bibitem{girault1}
V.~Girault and L.~R. Scott.
\newblock Hermite interpolation of nonsmooth functions preserving boundary
  conditions.
\newblock {\em Math. Comp.}, 71(239):1043--1074 (electronic), 2002.

\bibitem{guo2}
Benqi Guo.
\newblock Approximation theory for the {$p$}-version of the finite element
  method in three dimensions. {I}. {A}pproximabilities of singular functions in
  the framework of the {J}acobi-weighted {B}esov and {S}obolev spaces.
\newblock {\em SIAM J. Numer. Anal.}, 44(1):246--269 (electronic), 2006.

\bibitem{h14}
Norbert Heuer.
\newblock On the equivalence of fractional-order {S}obolev semi-norms.
\newblock {\em J. Math. Anal. Appl.}, 417(2):505--518, 2014.

\bibitem{Hilbert_Mollifier_MCOM73}
Stephen Hilbert.
\newblock A mollifier useful for approximations in {S}obolev spaces and some
  applications to approximating solutions of differential equations.
\newblock {\em Math. Comp.}, 27:81--89, 1973.

\bibitem{hw08}
George~C. Hsiao and Wolfgang~L. Wendland.
\newblock {\em Boundary integral equations}, volume 164 of {\em Applied
  Mathematical Sciences}.
\newblock Springer-Verlag, Berlin, 2008.

\bibitem{maday89}
Y.~Maday.
\newblock Rel{\`e}vement de traces poly{\^o}miales et interpolations
  hilbertiennes entres espaces de polyn{\^o}mes.
\newblock {\em C.R. Acad. Sci. Paris, S{\'e}rie {I}}, 309:463--468, 1989.

\bibitem{mclean00}
W.~McLean.
\newblock {\em Strongly elliptic systems and boundary integral equations}.
\newblock Cambridge University Press, 2000.

\bibitem{mel1}
J.~M. Melenk.
\newblock hp-interpolation of nonsmooth functions and an application to hp-a
  posteriori error estimation.
\newblock {\em SIAM J. Numer. Anal.}, 43(1):127--155, 2005.

\bibitem{mps13}
J.~M. Melenk, A.~Parsania, and S.~Sauter.
\newblock General {DG}-methods for highly indefinite {H}elmholtz problems.
\newblock {\em J. Sci. Comput.}, 57(3):536--581, 2013.

\bibitem{melenk-sauter10}
J.M. Melenk and S.~Sauter.
\newblock Convergence analysis for finite element discretizations of the
  {H}elmholtz equation with {D}irichlet-to-{N}eumann boundary conditions.
\newblock {\em Math. Comput.}, 79:1871--1914, 2010.

\bibitem{melenk-wihler14}
J.M. Melenk and T.~Wihler.
\newblock A posteriori error analysis of $hp$-{FEM} for singularly perturbed
  problems reaction-diffusion equations.
\newblock Technical report, 2014.
\newblock arXiv:1408.6037.

\bibitem{sola1}
Rafael Mu{\~n}oz-Sola.
\newblock Polynomial liftings on a tetrahedron and applications to the
  {$h$}-{$p$} version of the finite element method in three dimensions.
\newblock {\em SIAM J. Numer. Anal.}, 34(1):282--314, 1997.

\bibitem{muramatu67a}
Tosinobu Muramatu.
\newblock A proof for the imbedding theorems for {S}obolev spaces.
\newblock {\em Proc. Japan Acad.}, 43:87--92, 1967.

\bibitem{muramatu67}
Tosinobu Muramatu.
\newblock On imbedding theorems for {S}obolev spaces and some of their
  generalizations.
\newblock {\em Publ. Res. Inst. Math. Sci. Ser. A}, 3:393--416, 1967/1968.

\bibitem{Nedelec_88_AIE}
J.-C. N{\'e}d{\'e}lec.
\newblock Approximation of integral equations with nonintegrable kernels.
\newblock In {\em The mathematics of finite elements and applications, VI
  (Uxbridge, 1987)}, pages 343--352. Academic Press, London, 1988.

\bibitem{quarteroni84}
Alfio Quarteroni.
\newblock Some results of {B}ernstein and {J}ackson type for polynomial
  approximation in {$L^p$}-spaces.
\newblock {\em Japan J. Appl. Math.}, 1(1):173--181, 1984.

\bibitem{rand12}
Alexander Rand.
\newblock Average interpolation under the maximum angle condition.
\newblock {\em SIAM J. Numer. Anal.}, 50(5):2538--2559, 2012.

\bibitem{rank86a}
E.~Rank.
\newblock Adaptivity and accuracy estimation for finite element and boundary
  integral element methods.
\newblock In I.~Babu\v{s}ka, O.~C. Zienkiewicz, J.~Gago, and E.~R.
  de~A.~Oliveira, editors, {\em Accuracy Estimates and Adaptive Refinements in
  Finite Element Computations}. John Wiley \& Sons, 1986.

\bibitem{r93}
E.~Rank.
\newblock Adaptivity and accuracy estimation for {FEM} and {BIEM}.
\newblock In {\em Adaptive finite and boundary element methods}, Internat. Ser.
  Comput. Engrg., pages 1--46. Comput. Mech., Southampton, 1993.

\bibitem{rank89}
Ernst Rank.
\newblock Adaptive {$h$}-{$,\;p$}- and {$hp$}-versions for boundary integral
  element methods.
\newblock {\em Internat. J. Numer. Methods Engrg.}, 28(6):1335--1349, 1989.

\bibitem{reed-simonII}
Michael Reed and Barry Simon.
\newblock {\em Methods of modern mathematical physics. {II}. {F}ourier
  analysis, self-adjointness}.
\newblock Academic Press [Harcourt Brace Jovanovich, Publishers], New
  York-London, 1975.

\bibitem{sasc:11}
Stefan~A. Sauter and Christoph Schwab.
\newblock {\em Boundary element methods}, volume~39 of {\em Springer Series in
  Computational Mathematics}.
\newblock Springer-Verlag, Berlin, 2011.
\newblock Translated and expanded from the 2004 German original.

\bibitem{schwab1}
Ch. Schwab.
\newblock {\em {$p$}- and {$hp$}-finite element methods}.
\newblock Numerical Mathematics and Scientific Computation. The Clarendon Press
  Oxford University Press, New York, 1998.
\newblock Theory and applications in solid and fluid mechanics.

\bibitem{scott1}
L.~Ridgway Scott and Shangyou Zhang.
\newblock Finite element interpolation of nonsmooth functions satisfying
  boundary conditions.
\newblock {\em Math. Comp.}, 54(190):483--493, 1990.

\bibitem{shankov85}
V.~V. Shan{\cprime}kov.
\newblock The averaging operator with variable radius, and the inverse trace
  theorem.
\newblock {\em Sibirsk. Mat. Zh.}, 26(6):141--152, 191, 1985.

\bibitem{stein1}
Elias~M. Stein.
\newblock {\em Singular integrals and differentiability properties of
  functions}.
\newblock Princeton Mathematical Series, No. 30. Princeton University Press,
  Princeton, N.J., 1970.

\bibitem{s08}
Olaf Steinbach.
\newblock {\em Numerical approximation methods for elliptic boundary value
  problems}.
\newblock Springer, New York, 2008.
\newblock Finite and boundary elements, Translated from the 2003 German
  original.

\bibitem{verfuerth1}
R{\"u}diger Verf{\"u}rth.
\newblock Error estimates for some quasi-interpolation operators.
\newblock {\em M2AN Math. Model. Numer. Anal.}, 33(4):695--713, 1999.

\end{thebibliography}
\end{document}